\documentclass[11pt]{amsart}
\usepackage{amsmath}
\usepackage{amssymb}
\usepackage{amsthm}
\usepackage{latexsym}
\usepackage{hyperref}
\usepackage{enumerate}
\usepackage[all]{xy}
\usepackage{parallel}

\newtheorem{thm}{Theorem}[section]
\newtheorem{cor}[thm]{Corollary}
\newtheorem{lem}[thm]{Lemma}
\newtheorem{prop}[thm]{Proposition}

\newtheorem{thmintro}{Theorem}

\theoremstyle{definition}

\newtheorem{rem}[thm]{Remark}
\newtheorem{cond}[thm]{Conditions}
\newtheorem{ex}[thm]{Example}

\newcommand{\enuma}[1]{\begin{enumerate}[\textup{(}a\textup{)}] {#1} \end{enumerate}}
\newcommand{\iref}[1]{\S \ref{#1}, p. \pageref{#1}}

\newcommand{\mr}{\mathrm}
\newcommand{\mc}{\mathcal}
\newcommand{\mf}{\mathfrak}

\newcommand{\N}{\mathbb N}
\newcommand{\Z}{\mathbb Z}

\newcommand{\R}{\mathbb R}
\newcommand{\C}{\mathbb C}

\newcommand{\cusp}{\mathrm{cusp}}
\newcommand{\reg}{\mathrm{reg}}

\newcommand{\matje}[4]{\left(\begin{smallmatrix} #1 & #2 \\ 
#3 & #4 \end{smallmatrix}\right)}

\def\Hom{{\rm Hom}}
\def\End{{\rm End}}

\def\Irr{{\rm Irr}}

\def\GL{{\rm GL}}

\def\SL{{\rm SL}}

\def\St{{\rm St}}

\def\cH{{\mathcal H}}

\def\Ind{{\rm Ind}}

\def\Mod{{\rm Mod}}
\def\nr{{\rm nr}}

\def\fs{{\mathfrak s}}
\def\ft{{\mathfrak t}}

\def\Rep{{\rm Rep}}
\def\Mod{{\rm Mod}}
\def\Res{{\rm Res}}
\def\Stab{{\rm Stab}}
\def\Nrd{{\rm Nrd}}

\def\Nor{{\rm N}}

\def\der{{\rm der}}

\begin{document}

\title{Hecke algebras for inner forms\\ of $p$-adic special linear groups}

\author[A.-M. Aubert]{Anne-Marie Aubert}
\address{Institut de Math\'ematiques de Jussieu -- Paris Rive Gauche, 
U.M.R. 7586 du C.N.R.S., U.P.M.C., 4 place Jussieu 75005 Paris, France}
\email{aubert@math.jussieu.fr}
\author[P. Baum]{Paul Baum}
\address{Mathematics Department, Pennsylvania State University, University Park, PA 16802, USA}
\email{baum@math.psu.edu}
\author[R. Plymen]{Roger Plymen}
\address{School of Mathematics, Southampton University, Southampton SO17 1BJ,  England 
\emph{and} School of Mathematics, Manchester University, Manchester M13 9PL, England}
\email{r.j.plymen@soton.ac.uk \quad plymen@manchester.ac.uk}
\author[M. Solleveld]{Maarten Solleveld}
\address{Radboud Universiteit Nijmegen, Heyendaalseweg 135, 6525AJ Nijmegen, the Netherlands}
\email{m.solleveld@science.ru.nl}
\date{\today}
\thanks{The second author was partially supported by NSF grant DMS-1200475.}
\maketitle

\begin{abstract}
Let $F$ be a non-archimedean local field and let $G^\sharp$ be the group of $F$-rational
points of an inner form of $\SL_n$. We study Hecke algebras for all Bernstein components 
of $G^\sharp$, via restriction from an inner form $G$ of $\GL_n (F)$. 

For any packet of L-indistinguishable Bernstein components, we exhibit an explicit
algebra whose module category is equivalent to the associated category of complex smooth
$G^\sharp$-representations. This algebra comes from an idempotent in the full Hecke 
algebra of $G^\sharp$, and the idempotent is derived from a type for $G$. We show that the 
Hecke algebras for Bernstein components of $G^\sharp$ are similar to affine Hecke algebras 
of type $A$, yet in many cases are not Morita equivalent to any crossed product of an 
affine Hecke algebra with a finite group.
\end{abstract}

\vspace{4mm}

\tableofcontents

\section*{Introduction}

Let $F$ be a non-archimedean local field and let $D$ be a division algebra, of dimension
$d^2$ over its centre $F$. Then $G = \GL_m (D)$ is the group of $F$-rational points of an 
inner form of $\GL_n$, where $n = md$. We will say simply that $G$ is an inner form of 
$\GL_n (F)$. Its derived group $G^\sharp$, the kernel of the reduced
norm map $G \to F^\times$, is an inner form of $\SL_n (F)$. Every inner form of
$\SL_n (F)$ looks like this.

Since the appearance of the important paper \cite{HiSa} there has been a surge of
interest in these groups, cf. \cite{ChLi,ChGo,ABPS3}. In this paper we continue
our investigations of the (complex) representation theory of inner forms of $\SL_n (F)$.
Following the Bushnell--Kutzko approach \cite{BuKu3}, we study algebras associated
to idempotents in the Hecke algebra of $G^\sharp$. The main idea of this approach is
to understand Bernstein components for a reductive $p$-adic group better by constructing
types and making the ensuing Hecke algebras explicit. 

It turns out that for the groups under consideration, while it is really hard to find 
types, the appropriate Hecke algebras are accessible via different techniques. 
Our starting point is the construction of types for all Bernstein
components of $G$ by S\'echerre--Stevens \cite{SeSt4,SeSt6}. We consider the Hecke
algebra of such a type, which is described in \cite{Sec3}. In several steps we modify
this algebra to one whose module category is equivalent to a union of some Bernstein
blocks for $G^\sharp$. Let us discuss our strategy and our main result. 

We fix a parabolic subgroup $P \subset G$ with Levi factor $L$. A supercuspidal 
$L$-representation $\omega$ gives an inertial equivalence class $\fs = [L,\omega]_G$.
Let $\Rep^\fs (G)$ denote the corresponding Bernstein block of the category of
smooth complex $G$-representations, and let $\Irr^\fs (G)$ denote the set of 
irreducible objects in $\Rep^\fs (G)$.
Let $\Irr^\fs (G^\sharp)$ be the set of irreducible $G^\sharp$-representations that are 
subquotients of $\Res^G_{G^\sharp}(\pi)$ for some $\pi \in \Irr^\fs (G)$. We define 
$\Rep^\fs (G^\sharp)$ as the collection of $G^\sharp$-repre\-sen\-tations 
all whose irreducible subquotients lie in $\Irr^\fs (G^\sharp)$.
We want to investigate the category $\Rep^\fs (G^\sharp)$.
It is a product of finitely many Bernstein blocks for $G^\sharp$:
\begin{equation} \label{eq:procat}
\Rep^\fs (G^\sharp) = 
\prod\nolimits_{{\mf t}^\sharp\prec\fs}\Rep^{{\mf t}^\sharp} (G^\sharp).
\end{equation} 
We note that the Bernstein components $\Irr^{\mf t^\sharp}(G^\sharp)$ which are subordinate
to one $\fs$ form precisely one class of L-indistinguishable components: every L-packet
for $G^\sharp$ which intersects one of them intersects them all.

The structure of $\Rep^\fs (G)$ is largely determined by the torus $T_\fs$ and the finite
group $W_\fs$ associated by Bernstein to $\fs$. Recall that the Bernstein torus of $\fs$ is
\[
T_\fs = \{ \omega \otimes \chi \mid \chi \in X_\nr (L) \} \subset \Irr (L),
\]
where $X_\nr(L)$ denotes the group of unramified characters of $L$. The finite group $W_\fs$ 
equals $\Nor_M(L)/L$ for a suitable Levi subgroup $M \subset G$ containing $L$. 
For this particular reductive $p$-adic group $W_\fs$ is always a Weyl group (in fact a 
direct product of symmetric groups), but for $G^\sharp$ more general finite groups are 
needed. We also have to take into account that we are dealing with several Bernstein 
components simultaneously. 

Let $\cH (G)$ be the Hecke algebra of $G$ and $\cH (G)^\fs$ its two-sided ideal 
corresponding to the Bernstein block $\Rep^\fs (G)$. Similarly, let $\cH (G^\sharp)^\fs$
be the two-sided ideal of $\cH (G^\sharp)$ corresponding to $\Rep^\fs (G^\sharp)$. Then 
\[
\cH (G^\sharp)^\fs = \prod\nolimits_{{\mf t}^\sharp\prec\fs} \cH (G^\sharp )^{\mf t^\sharp} .
\]
Of course we would like to determine $\cH (G^\sharp )^{\mf t^\sharp}$, but it turns
out that $\cH (G^\sharp)^\fs$ is much easier to study. So our main goal is an explicit
description of $\cH (G^\sharp)^\fs$ up to Morita equivalence. From that the subalgebras
$\cH (G^\sharp )^{\mf t^\sharp}$ can in principle be extracted, as maximal indecomposable
subalgebras. We note that sometimes $\cH (G^\sharp)^\fs = \cH (G^\sharp )^{\mf t^\sharp}$,
see Examples \ref{ex:free}, \ref{ex:cocycles}.

From \cite{SeSt4} we know that there exists a simple type $(K,\lambda)$ 
for $[L,\omega]_M$. By \cite{SeSt6} it has a $G$-cover $(K_G,\lambda_G)$. We denote
the associated central idempotent of $\cH (K)$ by $e_\lambda$, and similarly for other
irreducible representations. There is an affine Hecke algebra $\cH (T_\fs,W_\fs,q_\fs)$, 
a tensor product of affine Hecke algebras of type $\GL_e$, such that 
\begin{equation}\label{eq:I1}
e_{\lambda_G} \cH (G) e_{\lambda_G} \cong e_\lambda \cH (M) e_\lambda \cong
\cH (T_\fs,W_\fs,q_\fs) \otimes \End_\C (V_\lambda) ,
\end{equation}
and these algebras are Morita equivalent with $\cH (G)^\fs$.

An important role in the restrictions of representations from $\cH (G)^\fs$ to
$\cH (G^\sharp)^\fs$ is played by the group
\begin{align*}
X^G (\fs) := \{ \gamma \in \Irr (G / G^\sharp Z(G)) \mid \gamma \otimes I_P^G (\omega)
\in \Rep^\fs (G) \} .
\end{align*}
It acts on $\cH (G)$ by pointwise multiplication of functions $G \to \C$. 
For the restriction process we need an idempotent that is invariant under 
$X^G (\fs)$. To that end we replace $\lambda_G$ by the sum of the representations 
$\gamma \otimes \lambda_G$ with $\gamma \in X^G (\fs)$, which we call $\mu_G$. 
Then \eqref{eq:I1} remains valid with $\mu$ instead of $\lambda$, but of course
$V_\mu$ is reducible as a representation of $K$. 

Let $e_{\mu_{G^\sharp}} \in \cH (G^\sharp)$ be the restriction of $e_{\mu_G} :
G \to \C$ to $G^\sharp$. Up to a scalar factor it is also the restriction of
$e_{\lambda_G}$ to $G^\sharp$. We normalize the Haar measures in such a way that
$e_{\mu_{G^\sharp}}$ is idempotent. For any $G^\sharp$-representation $V ,\;
e_{\mu_{G^\sharp}} V$ is the space of vectors in $V$ on which $K_G \cap G^\sharp$
acts as some multiple of the (reducible) representation $\lambda_G 
|_{K_G \cap G^\sharp}$. 

Then $e_{\mu_{G^\sharp}} \cH (G^\sharp) e_{\mu_{G^\sharp}}$ is a nice subalgebra 
of $\cH (G^\sharp)^\fs$, but in 
general it is not Morita equivalent with $\cH (G^\sharp)^\fs$. There is only
an equivalence between the module category of $e_{\mu_{G^\sharp}} \cH (G^\sharp)
e_{\mu_{G^\sharp}}$ and $\prod_{\mf t^\sharp} \Rep^{\mf t^\sharp} (G^\sharp)$,
where $\mf t^\sharp$ runs over some, but not necessarily all, inertial 
equivalence classes $\prec \fs$.
To see the entire category $\Rep^\fs (G^\sharp)$ we need finitely many 
isomorphic but mutually orthogonal algebras 
\[                    
a e_{\mu_{G^\sharp}} a^{-1} \cH (G^\sharp) a e_{\mu_{G^\sharp}} a^{-1} 
\text{ with } a \in G. 
\]
To formulate our main result precisely, we need also the groups
\begin{align*}
& X^L (\fs) = \{ \gamma \in \Irr (L / L^\sharp Z(G)) \mid 
\gamma \otimes \omega \in [L,\omega]_L \} , \\ 
& W_\fs^\sharp = \{ w \in \Nor_G(L) \mid \exists \gamma \in \Irr (L / L^\sharp Z(G)) : 
w (\gamma \otimes \omega) \in [L,\omega]_L \} , \\
& \mf R_\fs^\sharp = W_\fs^\sharp \cap \Nor_G (P \cap M) / L , \\
& X^L (\omega,V_\mu) = \{ \gamma \in \Irr (L / L^\sharp) \mid 
\text{ there exists an } L\text{-isomorphism } \omega \to \omega \otimes \gamma^{-1} \\
& \ \hspace{5cm} \text{which induces the identity on } V_\mu \} .
\end{align*}
Here $L^\sharp = L \cap G^\sharp$, so 
\[
L / L^\sharp \cong G / G^\sharp \cong F^\times .
\]
We observe that $\mf R_\fs^\sharp$ is naturally isomorphic to 
$X^G (\fs) / X^L (\fs)$, and that $W_\fs^\sharp = W_\fs \rtimes \mf R_\fs^\sharp$ 
(see Lemmas~\ref{lem:2.4} and~\ref{lem:2.2}). One can regard $W_\fs^\sharp$ as the 
Bernstein group for $\Rep^\fs (G^\sharp)$. 

\begin{thmintro}\label{thm:3}
\textup{[see Theorem \ref{thm:3.12}.]} \\
The algebra $\cH (G^\sharp)^\fs$ is Morita equivalent with a direct sum of\\
$|X^L (\omega,V_\mu)|$ copies of $e_{\mu_{G^\sharp}} \cH (G^\sharp) 
e_{\mu_{G^\sharp}}$. The latter algebra is isomorphic with
\[
\Big( \cH (T_\fs^\sharp ,W_\fs ,q_\fs) \otimes \End_\C (V_\mu) 
\Big)^{X^L (\fs) / X^L (\omega,V_\mu) X_\nr (L / L^\sharp Z(G))} \rtimes \mf R_\fs^\sharp ,
\]
where $T_\fs^\sharp = T_\fs / X_\nr (L / L^\sharp)$.
The actions of the groups $X^L (\fs)$ and $\mf R_\fs^\sharp$ come from automorphisms of
$T_\fs \rtimes W_\fs$ and projective transformations of $V_\mu$.
\end{thmintro}

The projective actions of $X^L (\fs)$ and $\mf R_\fs^\sharp$ on $V_\mu$ are always 
linear in the split case $G = \GL_n (F)$, but not in general, 
see Examples \ref{ex:decomposition} and \ref{ex:cocycles}.

Contrary to what one might expect from Theorem \ref{thm:3}, the Bernstein torus 
$T_{\mf t^\sharp}$ for $\Rep^{\mf t^\sharp}(G^\sharp)$ is not always $T_\fs^\sharp$, 
see Example \ref{ex:torus}. In general one has to divide by a finite subgroup of 
$T_\fs^\sharp$ coming from $X^L (\fs)$. It is possible that $W_{\mf t^\sharp}$ (with 
$\mf t^\sharp \prec \fs$) is strictly larger than $W_\fs$, and that it acts on 
$T_{\mf t^\sharp}$ without fixed points, see Examples \ref{ex:Weyl} and \ref{ex:free}.

Of course the above has already been done for $\SL_n (F)$ itself, see \cite{BuKu1,BuKu2,
GoRo1,GoRo2}. Indeed, for $\SL_n (F)$ our work has a large intersection with these
papers. But the split case is substantially easier than the non-split case, for
example because every irreducible representation of $\GL_n (F)$ restricts to a representation
of $\SL_n (F)$ without multiplicities. Therefore our methods are necessarily different
from those of Bushnell--Kutzko and Goldberg--Roche, even if our proofs are considered
only for $\SL_n (F)$.

It is interesting to compare Theorem \ref{thm:3} for $\SL_n (F)$ with
the main results of \cite{GoRo2}. Our description of the Hecke algebras is more explicit,
thanks to considering the entire packet $\Rep^\fs (G^\sharp)$ of Bernstein blocks 
simultaneously. In \cite[\S 11]{GoRo2} some 2-cocycle pops up in the Hecke algebras,
which Goldberg--Roche expect to be trivial. From Theorem \ref{thm:3} one can deduce 
that it is indeed trivial, see Remark \ref{rem:4.splitCocycle}

\medskip

Now we describe the contents of the paper in more detail.
We start Section~\ref{sec:restrict} with recalling a few results about restriction of 
representations from $G$ to $G^\sharp$. Then we discuss what happens when one restricts 
an entire Bernstein component of representations at once. We introduce and study several 
finite groups which will be used throughout.

It turns out to be advantageous to restrict from $G$ to $G^\sharp$ in two steps, via 
$G^\sharp Z(G)$. This intermediate group is of finite index in $G$ if the characteristic of
$F$ does not divide $n$. Otherwise $[G : G^\sharp Z(G)] = \infty$ but, when studying only
$\Rep^\fs (G)$, one can apply the same techniques as for a group extension of finite index.
Restriction from $G^\sharp Z(G)$ to $G^\sharp$ is straightforward, so everything comes down
to understanding the decomposition of representations and Bernstein components of $G$ upon
restriction to $G^\sharp Z(G)$.

For any subgroup $H \subset G$ we write $H^\sharp = H \cap G^\sharp$. The correct analogue
of $W_\fs$ for $\Rep^\fs (G)$ combines Weyl groups and characters of the Levi subgroup $L$
that are trivial on $L^\sharp Z(G)$:
\[
\Stab (\fs) := \{ (w,\gamma) \in W(G,L) \times \Irr (L / L^\sharp Z(G)) \mid
w (\gamma \otimes \omega) \in [L,\omega]_L \} .
\]
This group acts naturally on $L$-representations by
\[
((w,\gamma)\pi) (l) = w(\gamma \otimes \pi) (l) := \gamma (l) \pi (w^{-1} l w).
\]
From another angle $\Stab (\fs)$ can be considered as the generalization, for the
inertial class $\fs$, of some groups associated to a single $G$-representation in 
\cite{ChLi,ChGo}. Its relevance is confirmed by the following result.

\begin{thmintro}\label{thm:1} 
\textup{[see Theorem \ref{thm:2.5}.]} \\
Let $\chi_1,\chi_2 \in X_\nr (L)$. The following are equivalent:
\begin{itemize}
\item[(i)] $\Res^G_{G^\sharp Z(G)} (I_P^G (\omega \otimes \chi_1))$ and 
$\Res^G_{G^\sharp Z(G)} (I_P^G (\omega \otimes \chi_2))$ 
have a common irreducible subquotient;
\item[(ii)] $\Res^G_{G^\sharp Z(G)} (I_P^G (\omega \otimes \chi_1))$ and
$\Res^G_{G^\sharp Z(G)} (I_P^G (\omega \otimes \chi_2))$ 
have the same irreducible constituents,
counted with multiplicity;
\item[(iii)] $\omega \otimes \chi_1$ and $\omega \otimes \chi_2$ 
belong to the same $\Stab(\fs)$-orbit. 
\end{itemize}
\end{thmintro}

The proof of Theorem \ref{thm:3} uses almost the entire paper. It contains four 
chains of arguments, in Sections \ref{sec:Mor} and \ref{sec:Hecke}, 
which are largely independent: 
\begin{itemize}
\item The main idea (see Lemmas \ref{lem:3.1} and \ref{lem:3.11}) consists
of Morita equivalences
\begin{equation}\label{eq:I3}
\begin{aligned}
& \cH (G^\sharp)^\fs \sim_M (\cH (G)^\fs )^{X^G (\fs) X_\nr (G)} ,\\
& \cH (G^\sharp Z(G))^\fs \sim_M (\cH (G)^\fs )^{X^G (\fs)} .
\end{aligned}
\end{equation}
This enables us to reduce the study of $\cH (G^\sharp)^\fs$ (up to Morita equivalence) to
$(\cH (G)^\fs )^{X^G (\fs)}$. The analogue of Theorem \ref{thm:3} for $G^\sharp Z(G)$
is Theorem \ref{thm:4.11}.
\item A technically complicated step is the construction of an idempotent 
$e^\fs_M \in \cH (M)$ (in Lemmas \ref{lem:3.30} and \ref{lem:3.2}), which is well-suited
for restriction from $M$ to $M^\sharp$. It relies on the conjugacy of the types 
$(K,w(\lambda) \otimes \gamma)$ with $(w,\gamma) \in \Stab (\fs)$, studied in
Proposition \ref{prop:3.3}. With this idempotent we get a Morita equivalence 
(Proposition \ref{prop:3.2})
\begin{equation}\label{eq:I4}
(\cH (G)^\fs )^{X^G (\fs)} \sim_M
(e^\fs_M \cH (M) e^\fs_M)^{X_L (\fs)} \rtimes \mf R_\fs^\sharp .
\end{equation}
This allows us to perform many calculations entirely in $M$, which is easier than in $G$.
We exhibit an idempotent $e^\sharp_{\lambda_G}$, larger than $e_{\lambda_G}$, which in 
Proposition \ref{prop:3.I} is used to improve \eqref{eq:I4} to an isomorphism
\[
e^\sharp_{\lambda_G} (\cH (G)^\fs )^{X^G (\fs)} e^\sharp_{\lambda_G} \cong
(e^\fs_M \cH (M) e^\fs_M)^{X_L (\fs)} \rtimes \mf R_\fs^\sharp .
\]
\item To reveal the structure of $(e^\fs_M \cH (M) e^\fs_M)^{X_L (\fs)}\rtimes \mathfrak R_\fs^\sharp$ we first
study (in subsection \ref{subsec:Secherre}) the Hecke algebras associated to the types 
for $[L,\omega]_L$ and $[L,\omega]_M$ constructed in \cite{Sec3}. They are tensor products
of affine Hecke algebras of type $GL_e$ with a matrix algebra. Obviously this part relies 
very much on the work of S\'echerre. These considerations culminate in Theorem \ref{thm:3.6},
which describes the Hecke algebras associated to relevant larger idempotents, in similar 
terms. We make the action of $X^G (\fs)$ on these algebras explicit in Lemmas \ref{lem:4.9} 
and \ref{lem:3.9}.
\item We would like to construct types for $G^\sharp Z(G)$ and for $G^\sharp$, whose
associated Hecke algebras are as described in Theorems \ref{thm:4.11} and \ref{thm:3}.
In Theorems \ref{thm:3.17} and \ref{thm:3.18} we take a step towards this goal, by 
constructing idempotents $e^\sharp_{\lambda_{G^\sharp Z(G)}} \in \cH (G^\sharp Z(G))^\fs$ 
and $e^\sharp_{\lambda_{G^\sharp}} \in \cH (G^\sharp)^\fs$ which see the correct module 
categories and have the desired Hecke algebras. In fact these idempotents are just the
restrictions of $e^\sharp_{\lambda_G} : G \to \C$ to $G^\sharp Z(G)$ and $G^\sharp$, 
respectively. 

However, we encounter serious obstructions to types in $G^\sharp$. The main problem
is that sometimes types $(K,\lambda \otimes \gamma)$ for $[L,\omega]_M$ are conjugate 
in $M$ but not in any compact subgroup of $M$, see Remark \ref{rem:4.notype} and
Examples \ref{ex:5.1}--\ref{ex:5.3}.
\end{itemize}

Interestingly, some of the algebras that turn up
do not look like affine Hecke algebras. In the literature there was hitherto (to the
best of our knowledge) only one example of a Hecke algebra of a type 
which was not Morita equivalent to a crossed product of an affine Hecke algebra
with a finite group, namely \cite[\S 11.8]{GoRo2}.
But in several cases of Theorem \ref{thm:3} the part $\End_\C (V_\mu)$ plays an
essential role, and it cannot be removed via some equivalence. Hence these algebras 
are further away from affine Hecke algebras than any previously known Hecke algebras 
related to types. See especially Example \ref{ex:cocycles}. 

\vspace{3mm}
\textbf{Acknowledgements.} The authors thank Shaun Stevens for several helpful emails 
and discussions.

\section{Notations and conventions}
\label{sec:not}

We start with some generalities, to fix the notations. Good sources for the material
in this section are \cite{Ren,BuKu3}.

Let $G$ be a connected
reductive group over a local non-archimedean field. All our representations are
tacitly assumed to be smooth and over the complex numbers.
We write Rep$(G)$ for the category of such $G$-representations and $\Irr (G)$ for
the collection of isomorphism classes of irreducible representations therein.

Let $P$ be a parabolic subgroup of $G$ with Levi factor $L$. The ``Weyl" group of
$L$ is $W(G,L) = N_G (L) / L$. It acts on equivalence classes of $L$-representations
$\pi$ by
\[
(w \cdot \pi) (g) = \pi (\bar w g \bar{w}^{-1}) , 
\]
where $\bar w \in N_G (L)$ is a chosen representative for $w \in W(G,L)$. 
We write \label{i:67}
\[
W_\pi = \{ w \in W(G,L) \mid w \cdot \pi \cong \pi \} . 
\]
Let $\omega$ be an irreducible supercuspidal $L$-representation. 
The inertial equivalence class $\fs = [L,\omega]_G$ gives rise to a category of smooth 
$G$-representations $\Rep^\fs (G)$ and a subset $\Irr^\fs (G) \subset \Irr (G)$.
Write $X_\nr (L)$ for the group of unramified characters $L \to \C^\times$. \label{i:80}
Then $\Irr^\fs (G)$ consists of all irreducible irreducible constituents of the
parabolically induced representations $I_P^G (\omega \otimes \chi)$ with
$\chi \in X_\nr (L)$. We note that $I_P^G$ always means normalized, smooth parabolic
induction from $L$ via $P$ to $G$.

The set $\Irr^{\fs_L}(L)$ with $\fs_L = [L,\omega]_L$ can be described 
explicitly, namely by \label{i:28} \label{i:47} \label{i:81}
\begin{align}
& X_{\nr} (L,\omega) = \{ \chi \in X_\nr (L) : \omega \otimes \chi \cong \omega \} , \\
& \Irr^{\fs_L} (L) = \{ \omega \otimes \chi : \chi \in X_\nr (L) / X_\nr (L,\omega) \} .
\end{align}
Several objects are attached to the Bernstein component $\Irr^\fs (G)$ of
$\Irr (G)$ \cite{BeDe}. Firstly, there is the torus \label{i:58}                              
\[  
T_\fs := X_\nr (L) / X_\nr (L,\omega) ,
\]
which is homeomorphic to $\Irr^{\fs_L}(L)$. 
Secondly, we have the groups \label{i:37} \label{i:66}
\begin{align*}
N_G (\fs_L) = & \{ g \in N_G (L) \mid g \cdot \omega \in \Irr^{\fs_L}(L) \} \\
= & \{ g \in N_G (L) \mid g \cdot [L,\omega]_L = [L,\omega]_L \} , \\
W_\fs := & \{ w \in W(G,L) \mid w \cdot \omega \in \Irr^{\fs_L}(L) \} = N_G (\fs_L) / L .
\end{align*}
Of course $T_\fs$ and $W_\fs$ are only determined up to isomorphism by $\fs$, actually
they depend on $\fs_L$. To cope with this, we tacitly assume that $\fs_L$ is known
when talking about $\fs$.

The choice of $\omega \in \Irr^{\fs_L}(L)$ fixes a bijection $T_\fs \to \Irr^{\fs_L}(L)$,
and via this bijection the action of $W_\fs$ on $\Irr^{\fs_L}(L)$ is transferred to
$T_\fs$. The finite group $W_\fs$ can be thought of as the "Weyl group" of $\fs$, 
although in general it is not generated by reflections. 

Let $C_c^\infty (G)$ be the vector space of compactly supported locally constant
functions $G \to \C$. The choice of a Haar measure on $G$ determines a convolution
product * on $C_c^\infty (G)$. The algebra $(C_c^\infty (G),*)$ is known as the \label{i:89}
Hecke algebra $\cH (G)$. There is an equivalence between Rep$(G)$ and the category
$\Mod (\cH (G))$ of $\cH (G)$-modules $V$ such that $\cH (G) \cdot V = V$.
We denote the collection of inertial equivalence classes for $G$ by $\mathfrak B (G)$.
The Bernstein decomposition
\[
\Rep (G) = \prod\nolimits_{\fs \in \mathfrak B (G)} \Rep^\fs (G) 
\]
induces a factorization in two-sided ideals
\[
\cH (G) = \prod\nolimits_{\fs \in \mathfrak B (G)} \cH (G)^\fs .
\]
Let $K$ be a compact open subgroup of $K$ and let $(\lambda,V_\lambda)$ be an
irreducible $K$-representation.  Let $e_\lambda \in \cH (K)$ 
be the associated central idempotent and write \label{i:44}
\[
\Rep^\lambda (G) = \{ (\pi,V) \in \Rep (G) \mid \cH (G) e_\lambda \cdot V = V \} . 
\]
Clearly $e_\lambda \cH (G) e_\lambda$ is a subalgebra of $\cH (G)$, and
$V \mapsto e_\lambda \cdot V$ defines a functor from $\Rep (G)$ to 
$\Mod (e_\lambda \cH (G) e_\lambda)$. By \cite[Proposition 3.3]{BuKu3} this
functor restricts to an equivalence of categories $\Rep^\lambda (G) \to
\Mod (e_\lambda \cH (G) e_\lambda)$ if and only if $\Rep^\lambda (G)$ is closed
under taking $G$-subquotients. Moreover, in that case there are finitely
many inertial equivalence classes $\fs_1, \ldots \fs_\kappa$ such that
\[
\Rep^\lambda (G) = \Rep^{\fs_1}(G) \times \cdots \times \Rep^{\fs_\kappa}(G) . 
\]
One calls $(K,\lambda)$ a type for $\{\fs_1,\ldots, \fs_\kappa\}$, or an $\fs_1$-type
if $\kappa = 1$. 

To a type $(K,\lambda)$ one associates the algebra \label{i:20}
\begin{multline*}
\cH (G,\lambda) := \{f : G \to \End_\C (V_\lambda^\vee) \mid 
\text{ supp}(f) \text{ is compact,} \\
f(k_1 g k_2) = \lambda^\vee (k_1) f(g) \lambda^\vee (k_2) \;
\forall g \in G, k_1,k_2 \in K \} .
\end{multline*}
Here $(\lambda^\vee ,V_\lambda^\vee)$ is the contragredient of $(\lambda,V_\lambda)$ and
the product is convolution of functions. By \cite[(2.12)]{BuKu3} there is a canonical
isomorphism
\begin{equation}\label{eq:1.5}
e_\lambda \cH (G) e_\lambda \cong \cH (G,\lambda) \otimes_\C \End_\C (V_\lambda) . 
\end{equation}
From now on we discuss things that are specific for $G = \GL_m (D)$, \label{i:85}
where $D$ is a central simple $F$-algebra. We write $\dim_F (D) = d^2$. Every Levi
subgroup $L$ of $G$ is isomorphic to $\prod_j \GL_{\tilde m_j}(D)$ for some 
$\tilde m_j \in \N$ with $\sum_j \tilde m_j = m$. Hence every irreducible
$L$-representation $\omega$ can be written as $\otimes_j \tilde \omega_j$
with $\tilde \omega_j \in \Irr (\GL_{\tilde m_j}(D))$. Then $\omega$ is supercuspidal 
if and only if every $\tilde \omega_j$ is so. As above, we
assume that this is the case. Replacing $(L,\omega)$ by an inertially equivalent
pair allows us to make the following simplifying assumptions:
\begin{cond} \label{cond} \
\begin{itemize}
\item if $\tilde m_i = \tilde m_j$ and $[\GL_{\tilde m_j}(D),\tilde \omega_i
]_{\GL_{\tilde m_j}(D)} = [\GL_{\tilde m_j}(D),\tilde \omega_j
]_{\GL_{\tilde m_j}(D)}$, then $\tilde \omega_i = \tilde \omega_j$;
\item $\omega = \prod_i \omega_i^{\otimes e_i}$, such that $\omega_i$ and $\omega_j$
are not inertially equivalent if $i \neq j$;
\item $L = \prod_i L_i^{e_i} = \prod_i \GL_{m_i}(D)^{e_i}$, embedded diagonally in 
$\GL_m (D)$ such that factors $L_i$ with the same $(m_i,e_i)$ are in subsequent positions;
\item as representatives for the elements of $W(G,L)$ we take permutation matrices;
\item $P$ is the parabolic subgroup of $G$ generated by $L$ and the upper triangular
matrices;
\item if $m_i = m_j, e_i = e_j$ and $\omega_i$ is isomorphic to $\omega_j \otimes \gamma$
for some character $\gamma$ of $\GL_{m_i}(D)$, then $\omega_i = \omega_j \otimes \gamma \chi$
for some $\chi \in X_{\nr}(\GL_{m_i}(D))$.
\end{itemize}
\end{cond}
We remark that these conditions are natural generalizations of \cite[\S 1.2]{GoRo2}
to our setting. Most of the time we will not need the conditions for stating the results, 
but they are useful in many proofs. Under Conditions \ref{cond} we define \label{i:84}
\begin{equation}\label{eq:1.1}
M = \prod_i M_i = \prod_i Z_G \Big( \prod_{j \neq i} L_j^{e_j} \Big) = 
\prod_i \GL_{m_i e_i}(D) ,
\end{equation}
a Levi subgroup of $G$ containing $L$. For $\fs = [L,\omega]_G$ we have
\begin{equation}\label{eq:1.2}
W_\fs = W(M,L) = N_M (L) / L = \prod\nolimits_i N_{M_i}(L_i^{e_i}) / L_i^{e_i} 
\cong \prod\nolimits_i S_{e_i} , 
\end{equation}
a direct product of symmetric groups. Writing $\fs_i = [L_i,\omega_i]_{L_i}$, 
the torus associated to $\fs$ becomes
\begin{align}
\label{eq:1.4} & T_\fs = \prod\nolimits_i ( T_{\fs_i} )^{e_i} , \\
& T_{\fs_i} = X_\nr (L_i) / X_\nr (L_i ,\omega_i) .
\end{align}
By our choice of representatives for $W(G,L) ,\, \omega_i^{\otimes e_i}$ is stable under
$N_{M_i}(L_i^{e_i}) / L_i^{e_i} \cong S_{e_i}$. The action of $W_\fs$ on $T_\fs$
is just permuting coordinates in the standard way and
\begin{equation}\label{eq:1.3}
W_\fs = W_\omega .
\end{equation}

\section{Restricting representations}
\label{sec:restrict}

\subsection{Restriction to the derived group} \
\label{par:res1}

We will study the restriction of representations of $G = \GL_m (D)$ to its derived group 
$G^\sharp = \GL_m (D)_{\der}$. For subgroups $H \subset G$ we will write  
\label{i:18}\label{i:19}
\[
H^\sharp = H \cap G^\sharp . 
\]
Recall that the reduced norm map Nrd$: M_m (D) \to F$ induces a group isomorphism \label{i:39}
\[
\Nrd : G / G^\sharp \to F^\times. 
\]
We start with some important relations between representations of $G$ and $G^\sharp$,
which were proven both by Tadi\'c and by Bushnell--Kutzko.

\begin{prop}\label{prop:2.6}
\enuma{
\item Every irreducible representation of $G^\sharp$ appears in an irreducible
representation of $G$.
\item For $\pi, \pi' \in \Irr (G)$ the following are equivalent:
\begin{itemize}
\item[(i)] $\Res_{G^\sharp}^G (\pi)$ and $\Res_{G^\sharp}^G (\pi')$ 
have a common irreducible subquotient;
\item[(ii)] $\Res_{G^\sharp}^G (\pi) \cong \Res_{G^\sharp}^G (\pi')$;
\item[(iii)] there is a $\gamma \in \Irr (G/G^\sharp)$ such that $\pi' \cong \pi \otimes \gamma$.
\end{itemize}
\item The restriction of $(\pi,V) \in \Irr (G)$ to $G^\sharp$ is a finite direct sum of 
irreducible $G^\sharp$-representations, each one appearing with the same multiplicity.
\item Let $(\pi',V')$ be an irreducible $G^\sharp$-subrepresentation of $(\pi,V)$. 
Then the stabilizer in $G$ of $V'$ is a open, normal, finite index subgroup 
$H_\pi \subset G$ which contains $G^\sharp$ and the centre of $G$.
} 
\end{prop}
\begin{proof}
All these results can be found in \cite[\S 2]{Tad}, where they are in fact shown for
any reductive group over a local non-archimedean field.
For $G = \GL_n (F)$, these statements were proven in \cite[Propositions 1.7 and 1.17]{BuKu1}
and \cite[Proposition 1.5]{BuKu2}. The proofs in \cite{BuKu1,BuKu2} also apply to $G = \GL_m (D)$. 
\end{proof}

Let $\pi \in \Irr (G)$. By Proposition \ref{prop:2.6}.d 
\begin{equation}\label{eq:2.9}
\End_{G^\sharp}(V) = \End_{H_\pi} (V) ,
\end{equation}
which allows us to use \cite[Chapter 2]{HiSa} and \cite[Section 2]{GeKn} 
(which is needed for \cite{HiSa}). We put \label{i:70}
\[
X^G (\pi) := \{ \gamma \in \Irr(G / G^\sharp) \mid \pi \otimes \gamma \cong \pi \} . 
\]
As worked out in \cite[Chapter 2]{HiSa}, this group governs the reducibility of 
$\mr{Res}^G_{G^\sharp} (\pi)$. (We will use this definition of $X^G (\pi)$ more 
generally if $\pi \in \Rep (G)$ admits a central character.) By \eqref{eq:2.9}
every element of $X^G (\pi)$ is trivial on $H_\pi$, so $X^G (\pi)$ is finite.
Via the local Langlands correspondence for $G$, the group $X^G (\pi)$
corresponds to the geometric R-group of the L-packet for 
$G^\sharp$ obtained from $\Res_{G^\sharp}^G (\pi)$, see \cite[\S 3]{ABPS3}.
We note that 
\[
X^G (\pi) \cap X_\nr (G) = X_\nr (G,\pi) .
\]
For every $\gamma \in X^G (\pi)$ there exists a nonzero intertwining operator
\begin{equation}\label{eq:2.30}
I(\gamma,\pi) \in \Hom_G (\pi \otimes \gamma ,\pi) = 
\Hom_G (\pi, \pi \otimes \gamma^{-1}),
\end{equation}
which is unique up to a scalar. As $G^\sharp \subset \ker (\gamma) ,\; I(\gamma,\pi)$
can also be considered as an element of $\End_{G^\sharp}(\pi)$. As such, 
these operators determine a 2-cocycle $\kappa_\pi$ by \label{i:26}
\begin{equation}\label{eq:2.21}
I(\gamma,\pi) \circ I(\gamma',\pi) = 
\kappa_\pi (\gamma,\gamma') I(\gamma \gamma' ,\pi) .
\end{equation}
By \cite[Lemma 2.4]{HiSa} they span the $G^\sharp$-intertwining algebra of $\pi$:
\begin{equation}\label{eq:2.1}
\End_{G^\sharp}\big( \mr{Res}^G_{G^\sharp} \pi \big) \cong \C [X^G (\pi) ,\kappa_\pi] ,
\end{equation}
where the right hand side denotes the twisted group algebra of $X^G (\pi)$.
By \cite[Corollary 2.10]{HiSa}
\begin{equation}\label{eq:2.2}
\mr{Res}^G_{G^\sharp} \pi \cong \bigoplus_{\rho \in \Irr (\C [X^G (\pi) ,\kappa_\pi])}
\Hom_{\C [X^G (\pi) ,\kappa_\pi]} (\rho,\pi) \otimes \rho
\end{equation}
as representations of $G^\sharp \times X^G (\pi)$. 

Let $P$ be a parabolic subgroup of $G$ with Levi factor $L$. The inclusion $L \to G$ 
induces isomorphisms
\begin{equation}\label{eq:2.3}
L / L^\sharp \to G / G^\sharp \cong F^\times .
\end{equation}
Let $\omega \in \Irr (L)$ be supercuspidal and unitary. Using \eqref{eq:2.3} we
can identify the $G^\sharp$-representations 
\[
\mr{Res}^G_{G^\sharp} (I_P^G (\omega)) \quad \text{and} \quad
I_{P^\sharp}^{G^\sharp}(\mr{Res}^L_{L^\sharp}(\omega)),
\]
which yields an inclusion $X^L (\omega) \to X^G (I_P^G (\omega))$. 
Every intertwining operator $I(\gamma,\omega)$ for $\gamma \in X^L (\omega)$ induces 
an intertwining operator
\begin{equation}\label{eq:2.25}
I(\gamma,I_P^G (\omega)) := I_P^G (I(\gamma,\omega)) \in 
\Hom_G (\gamma \otimes I_P^G (\omega),I_P^G (\omega)) ,
\end{equation}
for $\gamma$ as an element of $X^G (I_P^G (\omega))$. We warn that, even though 
$X^L (\omega)$ is finite abelian and $\omega$ is supercuspidal, it is still 
possible that the 2-cocycle $\kappa_\omega$ is nontrivial and that 
\[
\End_{L^\sharp}(\Res_{L^\sharp}^L (\omega)) \cong \C [X^L (\omega),\kappa_\omega]
\]
is noncommutative, see \cite[Example 6.3.3]{ChLi}.

We introduce the groups \label{i:53} \label{i:69}
\begin{align}
& \label{eq:2.10} W_\omega^\sharp = 
\{ w \in W(G,L) \mid \exists \gamma \in \Irr (L / L^\sharp) 
\mid w \cdot (\gamma \otimes \omega) \cong \omega \} , \\
& \label{eq:2.11} \Stab (\omega) = \{ (w,\gamma) \in W(G,L) \times \Irr (L/L^\sharp) 
\mid w \cdot (\gamma \otimes \omega) \cong \omega \} .
\end{align}
Notice that the actions of $W(G,L)$ and $\Irr (L/L^\sharp)$ on $\Irr (L)$ commute
because every element of $\Irr (L/L^\sharp)$ extends to a character of $G$ which is 
trivial on the derived subgroup of $G$. Clearly $W_\omega \times X^L (\omega)$ 
is a normal subgroup of Stab$(\omega)$ and there is a short exact sequence
\begin{equation}
1 \to X^L (\omega) \to \Stab (\omega) \to W_\omega^\sharp \to 1 .
\end{equation}
By \cite[Proposition 6.2.2]{ChLi} the projection of Stab$(\omega)$ 
on the second coordinate gives rise to a short exact sequence
\begin{equation}\label{eq:2.12}
1 \to W_\omega \to \Stab (\omega) \to X^G (I_P^G (\omega)) \to 1 
\end{equation}
and the group \label{i:46}
\begin{equation}\label{eq:2.5}
\mf R_\omega^\sharp := \Stab (\omega) / \big( W_\omega \times X^L (\omega) \big) \cong 
X^G (I_P^G (\omega)) / X^L (\omega) \cong W_\omega^\sharp / W_\omega
\end{equation}
is naturally isomorphic to the ``dual R-group" of the L-packet for 
$G^\sharp$ obtained from $\Res_{G^\sharp}^G (I_P^G (\omega))$. We remark that, with the
method of Lemma \ref{lem:2.2}.c, it is also possible to realize $\mf R_\omega^\sharp$ 
as a subgroup of Stab$(\omega)$.

When $\omega$ is unitary, the $G$-representation $I_P^G (\omega)$ is unitary, and hence
completely reducible as $G^\sharp$-representation. In this case \eqref{eq:2.2} shows that
the intertwining operators associated to elements of Stab$(\omega)$ span
$\End_{G^\sharp}(I_P^G (\omega))$. By \cite[Theorem 1.6.a]{ABPS2} that holds more generally 
for $I_P^G (\omega \otimes \chi)$ when $\chi$ is in Langlands position with respect to $P$.

The group $\Stab (\omega)$ also acts on the set of irreducible $L^\sharp$-representations 
appearing in $\Res_{L^\sharp}^L (\omega)$. For an irreducible subrepresentation 
$\sigma^\sharp$ of $\Res_{L^\sharp}^L (\omega)$ \cite[Proposition 6.2.3]{ChLi} says that 
\begin{equation}\label{eq:2.4}
W_\omega \subset W_{\sigma^\sharp} \subset W_\omega^\sharp
\end{equation}
and that the analytic R-group of $I_{P^\sharp}^{G^\sharp}(\sigma^\sharp)$ is \label{i:45}
\[
\mf R_{\sigma^\sharp} := W_{\sigma^\sharp} / W_\omega ,
\]
the stabilizer of $\sigma^\sharp$ in $\mf R_\omega^\sharp$. It is possible that 
$W_{\sigma^\sharp} \neq W_\omega^\sharp$ and $\mf R_{\sigma^\sharp} \neq \mf R_\omega^\sharp$, 
see \cite[Example 6.3.4]{ChLi}.

In view of \cite[Section 1]{ABPS2} the above results remain valid if $\omega \in \Irr (L)$ 
is assumed to be supercuspidal but not necessarily unitary. Just one modification is 
required: if $I_P^G (\omega)$ is reducible, one should consider the L-packet for $G^\sharp$ 
obtained from the (unique) Langlands constituent of $I_P^G (\omega)$.

\subsection{Restriction of Bernstein components} \
\label{par:res2}

Next we study the restriction of an entire Bernstein component $\Irr^\fs (G)$ to $G^\sharp$.
Let \label{i:Irr} $\Irr^\fs (G^\sharp)$ be the set of irreducible $G^\sharp$-representations that are 
subquotients of $\Res^G_{G^\sharp}(\pi)$ for some $\pi \in \Irr^\fs (G)$.

\begin{lem}\label{lem:2.1}
$\Irr^\fs (G^\sharp)$ is a union of finitely many Bernstein components for $G^\sharp$.
\end{lem}
\begin{proof}
Consider any $\pi^\sharp \in \Irr^\fs (G^\sharp)$. It is a subquotient of
\[
\Res^G_{G^\sharp}(I_P^G (\omega \otimes \chi_1)) = 
I_{P^\sharp}^{G^\sharp}(\Res^L_{L^\sharp}(\omega \otimes \chi_1))
\]
for some $\chi_1 \in X_\nr (L)$. Choose an irreducible summand $\sigma_1$ 
of the supercuspidal $L^\sharp$-representation $\Res^L_{L^\sharp}(\omega \otimes \chi_1)$, 
such that $\pi^\sharp$ 
is a subquotient of $I_{P^\sharp}^{G^\sharp}(\sigma_1)$. Then $\pi^\sharp$ lies in 
the Bernstein component $\Irr^{[L^\sharp,\sigma_1]_{G^\sharp}}(G^\sharp)$. 
Any unramified character of $\chi_2^\sharp$ of $L^\sharp$ lifts to an unramified character 
of $L$, say $\chi_2$. Now
\[
I_{P^\sharp}^{G^\sharp}(\Res^L_{L^\sharp}(\sigma_1) \otimes \chi_2^\sharp) \subset 
I_{P^\sharp}^{G^\sharp}(\Res^L_{L^\sharp}(\omega \otimes \chi_1 \chi_2 )) = 
\Res^G_{G^\sharp}(I_P^G (\omega \otimes \chi_1 \chi_2)) ,
\]
which shows that all irreducible subquotients of $I_{P^\sharp}^{G^\sharp}(\Res^L_{L^\sharp}
(\sigma_1) \otimes \chi_2^\sharp)$ belong to $\Irr^\fs (G^\sharp)$. It follows that 
$\Irr^{[L^\sharp,\sigma_1]_{G^\sharp}}(G^\sharp) \subset \Irr^\fs (G^\sharp)$.

The above also shows that any inertial equivalence class $\mf t^\sharp$ with \label{i:54}
$\Irr^{\mf t^\sharp}(G^\sharp) \subset \Irr^\fs (G^\sharp)$ must be of the form 
\begin{equation} \label{eq:2.6}
\mf t^\sharp = [L^\sharp , \sigma_2]_{G^\sharp}
\end{equation}
for some irreducible constituent $\sigma_2$ of $\Res^L_{L^\sharp} (\omega \otimes \chi_2)$. 
So up to an unramified twist $\sigma_2$ is an irreducible constituent of 
$\Res^L_{L^\sharp}(\omega)$. Now Proposition \ref{prop:2.6}.c shows that there are only 
finitely many possibilities for $\mf t^\sharp$.
\end{proof}

Motivated by this lemma, we write $\mf t^\sharp \prec \fs$ if $\Irr^{\mf t^\sharp}(G^\sharp) 
\subset \Irr^\fs (G^\sharp)$. In other words, \label{i:55}
\begin{equation}\label{eq:2.23}
\Irr^\fs (G^\sharp) = \bigcup\nolimits_{\mf t^\sharp \prec \fs} \Irr^{\mf t^\sharp}(G^\sharp) .
\end{equation}
The last part of the proof of Lemma \ref{lem:2.1} shows that every $\mf t^\sharp \prec \fs$
is of the form $[L^\sharp , \sigma^\sharp]_{G^\sharp}$ for some irreducible constituent 
$\sigma^\sharp$ of $\Res^L_{L^\sharp} (\omega)$. 
Recall from \eqref{eq:2.2} that constituents $\sigma^\sharp$ as above are parametrized 
by irreducible representations of the twisted group algebra
$\C [X^L(\omega),\kappa_\omega]$. 
However, non-isomorphic $\sigma^\sharp$ may give rise to the
same inertial equivalence class $\mf t^\sharp$ for $G^\sharp$. It is quite difficult
to determine the Bernstein tori $T_{\mf t^\sharp}$ precisely.

The finite group associated by Bernstein to 
$\mf t^\sharp = [L^\sharp,\sigma^\sharp]_{G^\sharp}$ 
is its stabilizer in $W (G,L) = W(G^\sharp,L^\sharp)$:
\[
W_{\mf t^\sharp} = \{ w \in W (G,L) \mid w \cdot \sigma^\sharp \in 
[L^\sharp,\sigma^\sharp]_{L^\sharp} \} .
\]
As the different $\sigma^\sharp$ are $M$-conjugate, they all produce the same
group $W_{\mf t^\sharp}$. So it depends only on $[M,\sigma]_M$.
It is quite possible that $W_{\mf t^\sharp}$ is strictly larger than $W_\fs$, we
already saw this in Example \ref{ex:Weyl}. First steps to study such cases were 
sketched (for $\SL_n (F)$) in \cite[\S 9]{BuKu2}.

For every $w \in W_{\mf t^\sharp}$, Proposition \ref{prop:2.6}.b (for $L^\sharp$)
guarantees the existence of a $\gamma \in \Irr (L / L^\sharp)$ such that
$w (\omega) \otimes \gamma \in [L,\omega]_L$. Up to an unramified character 
$\gamma$ has to be trivial on $Z(G)$, so we may take $\gamma \in \Irr (L/L^\sharp Z(G))$.
Hence $W_{\mf t^\sharp}$ is contained in the group \label{i:68}
\[
W_\fs^\sharp := \{ w \in W(G,L) \mid \exists \gamma \in \Irr (L / L^\sharp Z(G)) 
\text{ such that } w (\gamma \otimes \omega) \in [L,\omega]_L \}
\]
By \eqref{eq:2.4} the subgroup $W_\omega$ fixes every $[L^\sharp,\sigma^\sharp]_{L^\sharp}$,
so $W_\omega \subset W_{\mf t^\sharp}$ and 
\begin{equation}\label{eq:2.7}
\Stab_{W_\fs^\sharp / W_\omega}([L^\sharp,\sigma^\sharp]_{L^\sharp}) = 
W_{\mf t^\sharp} / W_\omega \supset W_{\sigma^\sharp} / W_\omega .
\end{equation}

\begin{lem}\label{lem:2.2}
\enuma{
\item $W_\fs$ is a normal subgroup of $W_\fs^\sharp$ which fixes all the 
$[L^\sharp,\sigma^\sharp]_{L^\sharp}$ with $[L^\sharp,\sigma^\sharp]_{G^\sharp} \prec \fs$.
\item $W_{\mf t^\sharp}$ is the normal subgroup of $W_\fs^\sharp$ consisting of all elements 
that fix every $[L^\sharp,\sigma^\sharp]_{L^\sharp}$ with 
$[L^\sharp,\sigma^\sharp]_{G^\sharp} \prec \fs$. In particular it contains $W_\fs$.
\item There exist subgroups $\mf R_\fs^\sharp \subset W_\fs^\sharp$, $\mf R_{\mf t^\sharp}
\subset W_{\mf t^\sharp}$ and $\mf R_\omega^\sharp \subset W_\omega^\sharp$ such that 
$W_\fs^\sharp = W_\fs \rtimes \mf R_\fs^\sharp$, $W_{\mf t^\sharp} = W_\fs \rtimes 
\mf R_{\mf t^\sharp}$ and $W_\omega^\sharp = W_\omega \rtimes \mf R_\omega^\sharp$.
}
\end{lem}
\begin{proof}\label{i:30}
(a) We may use the conditions \ref{cond}. Then $W_\omega = W_\fs$. 
For any $\sigma^\sharp$ as above the root systems $R_{\sigma^\sharp}$ and $R_\omega$ 
are equal by \cite[Lemma 3.4.1]{ChLi}. With \eqref{eq:1.3} we get
\[
W_\fs = W_\omega = W (R_\omega) = W(R_{\sigma^\sharp}) \subset W_{\sigma^\sharp} ,
\]
showing that $W_\fs$ fixes $\sigma^\sharp$ and $[L^\sharp,\sigma^\sharp]_{L^\sharp}$. \\
(b) As we observed above, we can arrange that every $\sigma^\sharp$ is a subquotient of
$\Res_{L^\sharp}^L (\omega)$. Since all constituents of this representation are associate
under $L$ and $W (G,L)$ acts trivially on $L/L^\sharp ,\; W_{\mf t^\sharp}$ fixes every
possible $[L^\sharp,\sigma^\sharp]_{L^\sharp}$. This gives the description of 
$W_{\mf t^\sharp}$. Because all the $[L^\sharp,\sigma^\sharp]_{L^\sharp}$ together form
a $W_\fs^\sharp$-space, $W_{\mf t^\sharp}$ is normal in that group.\\
(c) In the special case $G^\sharp = \SL_n (F)$, this was proven for $W_{\mf t^\sharp}$
in \cite[Proposition 2.3]{GoRo2}. Our proof is a generalization of that in \cite{GoRo2}.

Recall the description of $M = \prod_i M_i$ and $W_\fs$ from equations \eqref{eq:1.1}
and \eqref{eq:1.2}. We note that $P \cap M$ is a parabolic subgroup of $M$ containing $L$,
and that the group $W(M,L) = W_\fs$ acts simply transitively on the collection of such
parabolic subgroups. This implies that
\begin{equation}
\mf R_\fs^\sharp := W_\fs^\sharp \cap N_G (P \cap M) / L 
\end{equation}
is a complement to $W_\fs$ in $W_\fs^\sharp$. For $W_{\mf t^\sharp}$ \eqref{eq:2.7} shows
that we may take
\begin{equation}
\mf R_{\mf t^\sharp} := W_{\mf t^\sharp} \cap N_G (P \cap M) / L . 
\end{equation}
Similarly, for $W_\omega^\sharp$ \eqref{eq:2.5} leads us to
\begin{equation}
\mf R_\omega^\sharp := W_\omega^\sharp \cap N_G (P \cap M) / L . \qedhere
\end{equation}
\end{proof}

As analogues of $X^L (\omega), X^G (I_P^G (\omega))$ and Stab$(\omega)$ for 
$\fs = [L,\omega]_G$ we introduce \label{i:71} \label{i:73}
\begin{align}
\label{eq:2.14} & X^L (\fs) = \{ \gamma \in \Irr (L / L^\sharp Z(G)) \mid 
\gamma \otimes \omega \in [L,\omega]_L \} , \\
& X^G (\fs) = \{ \gamma \in \Irr (G / G^\sharp Z(G)) \mid 
\gamma \otimes I_P^G (\omega) \in \fs \} , \\
& \Stab(\fs) = \{ (w,\gamma) \in W(G,L) \times  \Irr (L / L^\sharp Z(G)) \mid
w (\gamma \otimes \omega) \in [L,\omega]_L \} .
\end{align}
Notice that $\Stab (\fs)$ contains $\Stab (\omega \otimes \chi)$ for every \label{i:48}
$\chi \in X_{\nr}(L / L^\sharp)$. It is easy to see that $W_\fs \times X^L (\fs)$ is a
normal subgroup of Stab$(\fs)$ and that there are short exact sequences
\begin{align}
\label{eq:2.24} & 1 \to X^L (\fs) \to \Stab (\fs) \to W_\fs^\sharp \to 1 , \\
\label{eq:2.20} & 1 \to X^L (\fs) \times W_\fs \to \Stab (\fs) \to 
W_\fs^\sharp / W_\fs \cong \mf R_\fs^\sharp \to 1 .
\end{align}
Furthermore we define \label{i:50}
\[
\Stab (\fs,P \cap M) = \{ (w,\gamma) \in \Stab (\fs) \mid w \in N_G (P \cap M) / L \} .
\]
The reduced norm map induces isomorphisms
\begin{equation}\label{eq:2.26}
L / L^\sharp Z(G) \to G / G^\sharp Z(G) \to F^\times / \mathrm{Nrd}(Z(G)) . 
\end{equation}
The right hand side is an abelian group of exponent $md$, 
but it is not necessarily finite, see \eqref{eq:2.27}.

\begin{lem}\label{lem:2.4}
\enuma{
\item $\Stab(\fs) = \Stab (\fs ,P \cap M) \ltimes W_\fs$.
\item The projection of $\Stab(\fs)$ on the second coordinate gives a group isomorphism
\[
\Stab (\fs,P \cap M) \cong \Stab (\fs) / W_\fs \to X^G (\fs) .
\]
\item The groups $X^L (\fs), X^G (\fs)$ and $\Stab(\fs)$ are finite.
\item There are natural isomorphisms
\[
X^G (\fs) / X^L (\fs) \cong \Stab (\fs,P \cap M) / X^L (\fs) \cong \mf R_\fs^\sharp .
\]
}
\end{lem}
\begin{proof}
(a) This can be shown in the same way as Lemma \ref{lem:2.2}.c. \\
(b) If $(w,\gamma) \in \Stab (\fs)$, then 
\[
\gamma \otimes I_P^G (\omega) \cong \gamma \otimes I_P^G (w \cdot \omega) \cong
I_P^G (\gamma \otimes w \cdot \omega) \cong I_P^G (w (\gamma \otimes \omega)) \in \fs ,
\]
so $\gamma \in X^G (\fs)$. Conversely, if $\gamma \in X^G (\fs)$, then 
$I_P^G (\omega \otimes \gamma) \in \fs$. Hence 
$\omega \otimes \gamma \in w^{-1} \cdot [L,\omega]_L = [L,w^{-1} \cdot \omega]_L$
for some $w \in W(G,L)$, and $(w,\gamma) \in \Stab (\fs)$. 

As $W(G,L)$ and $\Irr (L / L^\sharp)$ commute, the projection map Stab$(\fs) \to X^G (\fs)$ 
is a group homomorphism. In view of \eqref{eq:1.3}, we may assume that $\omega$ is such 
that $W_\fs = W_\omega$. Then the kernel of this group homomorphism is $W_\omega = W_\fs$. \\
(c) Suppose that $\omega \otimes \gamma \cong \omega \otimes \chi$ for some 
$\chi \in X_{\nr}(L)$. Then $\chi$ is trivial on $Z(G)$ and $\gamma^{-1} \chi \in 
X^L (\omega)$. We already know from Proposition \ref{prop:2.6} and \eqref{eq:2.9} that 
$X^L (\omega)$ finite. Hence $(\gamma^{-1} \chi)^{|X^L (\omega)|} = 1$ and
\[
\chi^{|X^L (\omega)} = \gamma^{-|X^L (\omega)} \in \Irr (L / L^\sharp Z(G)) .
\]
By \eqref{eq:2.26} $L / L^\sharp Z(G)$ is a group of exponent $md$, so
$\chi^{md |X^L (\omega)|} = \gamma^{-md |X^L (\omega)} = 1$. Thus there are only finitely 
many possibilities for $\chi$, and we can conclude that $X^L (\fs)$ is finite.

If $(w,\gamma), (w,\gamma') \in \Stab (\fs)$, then $(w,\gamma)^{-1} (w,\gamma') =
\gamma^{-1} \gamma' \in X^L (\fs)$. As $W(G,L)$ and $X^L (\fs)$ are finite, this
shows that Stab$(\fs)$ is also finite. Now $X^G (\fs)$ is finite by part (b).\\
(d) This follows from \eqref{eq:2.20} and part (b).
\end{proof}

\subsection{The intermediate group} \
\label{par:interm} 

For some calculations it is beneficial to do the restriction of representations from $G$
to $G^\sharp$ in two steps, via the intermediate group $G^\sharp Z(G)$. This is a central
extension of $G^\sharp$, so
\begin{equation}\label{eq:2.8}
\End_{G^\sharp}(\pi) = \End_{G^\sharp Z(G)}(\pi) 
\end{equation}
for all representations $\pi$ of $G$ or $G^\sharp Z(G)$ that admit a central character.
In particular $\Res_{G^\sharp}^{G^\sharp Z(G)}$ preserves irreducibility of representations.
The centre of $G$ is \label{i:86}
\begin{equation}\label{eq:2.27}
Z(G) = G \cap Z(M_m (D)) = G \cap F \cdot I_m = F^\times I_m .
\end{equation}
Recall that $\dim_F (D) = d^2$. Since Nrd$(z I_m) = z^{md}$ for $z \in F^\times$, 
\[
\Nrd (Z(G)) = F^{\times md},
\]
the group of $md$-th powers in $F^\times$. Hence $G / G^\sharp Z(G)$ is an abelian 
group and all its elements have order dividing $md$. In case char$(F)$ is positive and 
divides $md ,\; G^\sharp Z(G)$ is closed but not open in $G$. Otherwise it is closed, 
open and of finite index in $G$. However, $G^\sharp Z(G)$ is never Zariski-closed in $G$.

The intersection of $G^\sharp$ and $Z(G)$ is the finite group 
$\{ z I_m \mid z \in F^\times, z^{md} = 1 \}$, so
\begin{equation}
G^\sharp Z(G) \cong 
(G^\sharp \times Z(G)) / \{ (z I_m, z^{-1}) \mid z \in F^\times, z^{md} = 1 \}.
\end{equation}
As $G^\sharp \times Z(G)$ is a connected reductive algebraic group over $F$, this shows 
that $G^\sharp Z(G)$ is one as well. But this algebraic structure is not induced from the
enveloping group $G$.
The inflation functor $\Rep (G^\sharp Z(G)) \to \Rep (G^\sharp \times Z(G))$ identifies
$\Rep (G^\sharp Z(G))$ with
\[
\{ \pi \in \Rep (G^\sharp \times Z(G)) \mid \pi (z I_m, z^{-1}) = 1 \, \forall
z \in F^\times \text{ with } z^{md} = 1 \} .
\]

\begin{lem}\label{lem:2.6}
\enuma{
\item Every irreducible $G^\sharp$-representation can be lifted to an irreducible 
representation of $G^\sharp Z(G)$.
\item All fibers of
\[
\Res_{G^\sharp}^{G^\sharp Z(G)} : \Irr (G^\sharp Z(G)) \to \Irr (G^\sharp)
\]
are homeomorphic to $\Irr \big( F^{\times md}\big)$.
} 
\end{lem}
\begin{proof}
(a) Any $\pi^\sharp \in \Irr (G^\sharp)$ determines a character $\chi_{md}$ 
of the central subgroup $\{ z \in F^\times \mid z^{md} = 1 \}$. Since there are only
finitely many $md$-th roots of unity in $F ,\; \chi_{md}$ can be lifted to a character
$\chi$ of $F^\times$. Then $\pi^\sharp \otimes \chi$ is a representation of 
$G^\sharp \times Z(G)$ that descends to $G^\sharp Z(G)$. \\
(b) This follows from the proof of part (a) and the short exact sequence
\begin{equation}
1 \to G^\sharp \to G^\sharp Z(G) \xrightarrow{\Nrd} F^{\times md} \to 1 . \qedhere
\end{equation}
\end{proof}

More explicitly, $\chi \in \Irr \big( F^{\times md}\big)$ acts freely on 
$\Irr (G^\sharp Z(G))$ by retraction to $\bar \chi \in \Irr (G^\sharp Z(G))$ 
and tensoring representations of $G^\sharp Z(G)$ with $\bar \chi$. 

For any totally disconnected group $H$ we define $X_\nr (H)$ as the collection of 
smooth characters which are trivial on every compact subgroup of $H$. Then
\[
X_\nr (G^\sharp Z(G)) \cong X_\nr (F^{\times md}) \cong X_\nr (F^{\times md} / 
\mathfrak o_F \cap F^{\times md}) = X_\nr (\varpi_F^{md \Z}) \cong \C^\times ,
\]
for any uniformizer $\varpi_F$ in the ring of integers $\mathfrak o_F$.

It follows that the preimage of $\Irr^{\mf t^\sharp}
(G^\sharp)$ in $\Irr (G^\sharp Z(G))$ consists of countably many Bernstein components
$\Irr^{\mf t}(G^\sharp Z(G))$, each one homeomorphic to 
\[
X_{\nr}\big( F^{\times md} \big) \times \Irr^{\mf t^\sharp} (G^\sharp) \cong
X_{\nr}(G^\sharp Z(G)) \times  \Irr^{\mf t^\sharp} (G^\sharp) \cong
\C^\times \times  \Irr^{\mf t^\sharp} (G^\sharp) .
\]
Two such components differ from each other by a ramified character of $F^{\times md}$,
or equivalently by a character of Nrd$(\mf o_F^\times \cdot 1_m)$.

In comparison, every Bernstein component $\Irr^{\mf t} (G^\sharp Z(G))$ projects onto a
single Bernstein component for $G^\sharp$, say $\Irr^{\mf t^\sharp} (G^\sharp)$.
All the fibers of
\begin{equation}\label{eq:2.15}
\Res_{G^\sharp}^{G^\sharp Z(G)} : 
\Irr^{\mf t} (G^\sharp Z(G)) \to \Irr^{\mf t^\sharp} (G^\sharp)
\end{equation}
are homeomorphic to $X_\nr (Z(G)) = X_{\nr}\big( F^{\times md} \big) \cong \C^\times$.
In particular
\begin{equation}\label{eq:2.16} 
T_{\mf t^\sharp} = T_{\mf t} / X_{\nr}\big( \Nrd (Z(G)) \big) .
\end{equation}

\begin{lem}\label{lem:2.3}
The finite groups associated to $\mf t$ and $\mf t^\sharp$ are equal:
$W_{\mf t} = W_{\mf t^\sharp}$. 
\end{lem}
\begin{proof}
As we observed above, $\mf t^\sharp$ is the only Bernstein component involved in the 
restriction of $\Irr^{\mf t}(G^\sharp Z(G))$ to $G^\sharp$. Hence 
$W^{\mf t} \subset W^{\mf t^\sharp}$.
Conversely, if $w \in W(L^\sharp)$ stabilizes $\mf t^\sharp$, then it stabilizes the set
of Bernstein components $\Irr^{\mf t'}(G^\sharp Z(G))$ which project onto 
$\Irr^{\mf t^\sharp}(G^\sharp)$. But any such $\mf t'$ differs from $\mf t$ only by a ramified 
character of $Z(G)$. Since conjugation by elements of $N_G (L)$ does not affect 
$Z(G) ,\; w(\mf t)$ cannot be another $\mf t'$, and so $w \in W_{\mf t}$.  
\end{proof}

The above provides a complete picture of $\Res_{G^\sharp}^{G^\sharp Z(G)}$, so we can focus
on $\Res^G_{G^\sharp Z(G)}$. Although $[G : G^\sharp Z(G)]$ is sometimes infinite 
(e.g. if char$(F)$ divides $md$), only 
finitely many characters of $G / G^\sharp Z(G)$ occur in relation to a fixed Bernstein
component. This follows from Lemma \ref{lem:2.4}.c and makes it possible to treat 
$G^\sharp Z(G) \subset G$ as a group extension of finite degree.

Given an inertial equivalence class $\fs = [L,\omega]_G$, we define
$\Irr^\fs (G^\sharp Z(G))$ as the set of all elements of $\Irr (G^\sharp Z(G))$ that can
be obtained as a subquotient of $\Res_{G^\sharp Z(G)}^G (\pi)$ for some $\pi \in
\Irr^\fs (G)$. We also define \label{i: Rep} $\Rep^\fs (G^\sharp Z(G))$ as the collection of $G^\sharp Z(G)
$-representations all whose irreducible subquotients lie in $\Irr^\fs (G^\sharp Z(G))$.
It follows from Lemma \ref{lem:2.1} and the above that $\Irr^\fs (G^\sharp Z(G))$ is a union 
of finitely many Bernstein components $\mf t$ for $G^\sharp Z(G)$. We denote this relation 
between $\fs$ and $\mf t$ by $\mf t \prec \fs$. Thus
\begin{equation}\label{eq:2.19}
\Irr^\fs (G^\sharp Z(G)) = \bigcup\nolimits_{\mf t \prec \fs} \Irr^{\mf t}(G^\sharp Z(G)) .
\end{equation}
All the reducibility of $G$-representations caused by 
restricting them to $G^\sharp$ can already be observed by restricting them to $G^\sharp Z(G)$.
In view of \eqref{eq:2.8} and Lemma \ref{lem:2.3}, all our results on $\Res_{G^\sharp}^G$
remain valid if we replace everywhere $G^\sharp$ by $G^\sharp Z(G)$ and $L^\sharp$ by
$L^\sharp Z(G)$.

In view of \eqref{eq:2.1}, intertwining operators associated to Stab$(\fs)$ span \\
$\End_{G^\sharp Z(G)}(I_P^G (\omega \otimes \chi))$ whenever $\chi \in X_\nr (L)$
is unitary. With results of Harish-Chandra we will show that even more is true. 
For $w \in W (G,L)$ let \label{i:31}
\begin{equation}\label{eq:2.17}
J \big( w, I_P^G (\omega \otimes \chi)) \big) \in \Hom_G \big( I_P^G (\omega \otimes \chi),
I_P^G (w(\omega \otimes \chi)) \big)
\end{equation}
be the intertwining operator constructed in \cite[\S 5.5.1]{Sil} and \cite[\S V.3]{Wal}. 
We recall that it is rational as a function of $\chi \in X_\nr (L)$ and that it is 
regular and invertible if $\chi$ is unitary. In contrast to the intertwining operators
below, \eqref{eq:2.17} can be normalized in a canonical way.

For $(w,\gamma) \in \Stab (\fs,P \cap M)$ there exists a $\chi' \in X_\nr (L)$, unique
up to $X_\nr (L,\omega)$, such that 
$w (\omega \otimes \chi \gamma) \cong \omega \otimes \chi'$. Choose a nonzero
\begin{equation}\label{eq:2.22}
J(\gamma,\omega \otimes \chi) \in 
\Hom_L \big( \omega \otimes \chi, w^{-1}(\omega \otimes \chi' \gamma^{-1}) \big) .
\end{equation}
In view of Lemma \ref{lem:2.4}.b $\gamma$ determines $w$, so this is unambigous and
determines $J(\gamma,\omega \otimes \chi)$ up to a scalar. 
For unramified $\gamma$ we have $\chi' = \chi \gamma$, but nevertheless 
$J(\gamma,\omega \otimes \chi)$ need not be a scalar multiple of identity. The reason
lies in the difference between $T_\fs$ and $X_\nr (L)$, we refer to \cite[\S V]{Wal}
for more background.

Parabolic induction produces
\begin{equation}
J(\gamma,I_P^G (\omega \otimes \chi)) := I_P^G ( J(\gamma,\omega \otimes \chi)) \in
\Hom_G \big( I_P^G (\omega \otimes \chi), I_P^G (w^{-1}(\omega \otimes \chi' \gamma^{-1})) \big) .
\end{equation}
For $w'(w , \gamma) \in \Stab (\fs)$ with $w' \in W_\fs$ and 
$(w,\gamma) \in \Stab (\fs,P\cap M)$ we define
\begin{multline}\label{eq:2.18}
J(w' (w,\gamma), I_P^G (\omega \otimes \chi)) := \\
J(w',I_P^G(\omega \otimes \chi' \gamma^{-1})) \circ  J \big( w, I_P^G (w^{-1}(\omega 
\otimes \chi' \gamma^{-1})) \big) \circ J(\gamma,I_P^G (\omega \otimes \chi)) .
\end{multline}
By construction this lies both in $\Hom_G \big( I_P^G (\omega \otimes \chi),
I_P^G (w'(\omega \otimes \chi' \gamma^{-1})) \big)$ and in \\
$\Hom_{G^\sharp Z(G)} \big( I_P^G (\omega \otimes \chi),I_P^G (w'(\omega \otimes \chi')) \big)$.
We remark that the map from $\Stab (\fs)$ to intertwining operators \eqref{eq:2.18}
is not always multiplicative, some 2-cocycle with values in $\C^\times$ might be 
involved. However, in view of the canonical normalization of \eqref{eq:2.17},
\begin{equation}
W_\fs \ni w' \mapsto \{ J(w',I_P^G (\omega \otimes \chi)) \mid \chi \in X_\nr (L) \} 
\end{equation}
is a group homomorphism.

The following result is the main justification for introducing Stab$(\fs)$ in \eqref{eq:2.14}. 

\begin{thm}\label{thm:2.5}
Let $\omega \in \Irr (L)$ be supercuspidal and let $\chi_1,\chi_2 \in X_\nr (L)$ 
be unramified characters.
\enuma{
\item The following are equivalent:
\begin{itemize}
\item[(i)] $\Res^G_{G^\sharp Z(G)} (I_P^G (\omega \otimes \chi_1))$ and 
$\Res^G_{G^\sharp Z(G)} (I_P^G (\omega \otimes \chi_2))$ have a common irreducible subquotient;
\item[(ii)] $\Res^G_{G^\sharp Z(G)} (I_P^G (\omega \otimes \chi_1))$ and
$\Res^G_{G^\sharp Z(G)} (I_P^G (\omega \otimes \chi_2))$ have the same irreducible constituents,
counted with multiplicity;
\item[(iii)] $\omega \otimes \chi_1$ and $\omega \otimes \chi_2$ 
belong to the same $\Stab(\fs)$-orbit. 
\end{itemize}
\item If $\omega \otimes \chi_1$ and $\omega \otimes \chi_2$ are unitary, then 
$\Hom_{G^\sharp Z(G)}\big( I_P^G (\omega \otimes \chi_1), I_P^G (\omega \otimes \chi_2) \big)$ 
is spanned by intertwining operators $J((w,\gamma), I_P^G (\omega \otimes \chi_1))$ with 
$(w,\gamma) \in \Stab (\fs)$ and $w (\omega \otimes \chi_1 \gamma) \cong \omega \otimes \chi_2$.
}
\end{thm}
\begin{proof}
First we assume that $\omega \otimes \chi_1$ and $\omega \otimes \chi_2$ are unitary.
By Harish-Chandra's Plancherel isomorphism \cite{Wal} and the commuting algebra theorem 
\cite[Theorem 5.5.3.2]{Sil}, the theorem is true for $G$, with $W_\fs$ instead of Stab$(\fs)$.
More generally, for any tempered $\rho_1,\rho_2 \in \Irr (L)$, $\Hom_G (I_P^G (\rho_1),
I_P^G (\rho_2))$ is spanned by intertwining operators $J(w,I_P^G (\rho_1))$ with
$w \in W(G,L)$ and $w \cdot \rho_1 \cong \rho_2$.

For $\pi_1, \pi_2 \in \Irr (G)$ Proposition \ref{prop:2.6}.b says that 
$\Res^G_{G^\sharp Z(G)}(\pi_1)$ and $\Res^G_{G^\sharp Z(G)}(\pi_2)$ are isomorphic if 
$\pi_2 \cong \pi_1 \otimes \gamma$ for some $\gamma \in \Irr (G / G^\sharp Z(G))$, and have 
no common irreducible subquotients otherwise. Together with \eqref{eq:2.1} this implies that 
\[
\Hom_{G^\sharp Z(G)}\big( I_P^G (\omega \otimes \chi_1), I_P^G (\omega \otimes \chi_2) \big)
\]
is spanned by intertwining operators $J( (w,\gamma), I_P^G (\omega \otimes \chi_1))$
with $w (\omega \otimes \chi_1 \gamma) \cong \omega \otimes \chi_2$. Such pairs $(w,\gamma)$
automatically belong to Stab$(\fs)$. Since both factors of $J( (w,\gamma), 
I_P^G (\omega \otimes \chi_1))$ are bijective, the equivalence of (i), (ii) and (iii) follows.
This proves (b) and (a) in the unitary case.

Now we allow $\chi_1$ and $\chi_2$ to be non-unitary. Assume (i). From Proposition 
\ref{prop:2.6}.b we obtain a $\gamma \in \Irr (G / G^\sharp Z(G))$ such that 
$I_P^G (\omega \otimes \chi_1 \gamma)$ and $I_P^G (\omega \otimes \chi_2)$ have a common 
irreducible quotient. The theory of the Bernstein centre for $G$ \cite{BeDe} implies that
$\omega \otimes \chi_1 \gamma$ and $\omega \otimes \chi_2$ are isomorphic via an element
$w \in W (G,L)$. Then $(w,\gamma) \in \Stab (\fs)$, so (iii) holds.

Suppose now that $w (\omega \otimes \chi_1 \gamma) \cong \omega \otimes \chi_2$ for some 
$(w,\gamma) \in \Stab (\fs)$ and consider the map
\begin{equation}\label{eq:2.13}
\cH (G^\sharp Z(G)) \times X_\nr (L) \to \C : (f,\chi) \mapsto \text{tr}(f,I_P^G 
(\omega \otimes \chi)) - \text{tr} \big( f, I_P^G (w (\omega \otimes \gamma \chi)) \big) .
\end{equation}
It is well-defined since $\Res^G_{G^\sharp}  I_P^G (\omega \otimes \chi)$ has finite length,
by Proposition \ref{prop:2.6}.c.
By what we proved above, the value is 0 whenever $\chi$ is unitary. But for a fixed $f \in \cH (G)$
this is a rational function of $\chi \in X_\nr (L)$, and the unitary characters are Zariski-dense
in $X_\nr (L)$. Therefore \eqref{eq:2.13} is identically 0, which shows that 
$(I_P^G (\omega \otimes \chi))$ and $I_P^G (w (\omega \otimes \gamma \chi))$ have the same trace.
By Proposition \ref{prop:2.6}.c these $G^\sharp$-representations have finite length, so by 
\cite[2.3.3]{Cas} their irreducible constituents (and multiplicities) are determined by their traces. 
Thus (iii) implies (ii), which obviously implies (i). 
\end{proof}

\section{Morita equivalences}
\label{sec:Mor}

Let $\fs = [L,\omega]_G$. We want to analyse the two-sided ideal \label{i: H} $\cH (G^\sharp Z(G))^\fs$ 
of $\cH (G^\sharp Z(G))$ associated to the category of representations 
$\Rep^\fs (G^\sharp Z(G))$ introduced in Subsection \ref{par:interm}. In this section we
will transform these algebras to more manageable forms by means of Morita equivalences.

We note that by \eqref{eq:2.19} we can regard $\cH (G^\sharp Z(G))^\fs$  
as a finite direct sum of ideals associated to one Bernstein component:
\begin{equation}\label{eq:3.47}
\cH (G^\sharp Z(G))^\fs = \bigoplus\nolimits_{\mf t \prec \fs} \cH (G^\sharp Z(G))^{\mf t} . 
\end{equation}
Recall that the abelian group $\Irr (G / G^\sharp)$ acts on $\cH (G)$ by 
\[
(\chi \cdot f)(g) = \chi (g) f(g) .
\]
We also introduce an alternative action of $\gamma \in \Irr (G / G^\sharp)$ on $\cH (G)$ 
(and on similar algebras):
\begin{equation}
\alpha_\gamma (f) = \gamma^{-1} \cdot f. 
\end{equation}
Obviously these two actions have the same invariants. An advantage of the latter lies in
the induced action on representations: 
\[
\alpha_\gamma (\pi) = \pi \circ \alpha_\gamma^{-1} = \pi \otimes \gamma.
\]
Suppose for the moment that the characteristic of $F$ does not divide $m d$, so that 
$G / G^\sharp Z(G)$ is finite. Then there are canonical isomorphisms
\begin{multline}\label{eq:3.83}
\bigoplus\nolimits_{\fs \in \mathfrak B (G) / \sim} 
\cH (G^\sharp Z(G))^\fs \cong \cH (G^\sharp Z(G)) \\ 
\cong \cH (G)^{\Irr (G / G^\sharp Z(G))} \cong 
\bigoplus\nolimits_{\fs \in \mathfrak B (G) / \sim} ( \cH (G)^\fs )^{X^G (\fs)} ,
\end{multline}
where $\fs \sim \fs'$ if and only if they differ by a character of $G / G^\sharp Z(G)$.

Unfortunately this is not true if char$(F)$ does divide $m d$. In that case there are no 
nonzero $\Irr (G / G^\sharp Z(G))$-invariant elements in $\cH (G)$, because such elements 
could not be locally constant as functions on $G$. In Subsection \ref{par:morita2}
we will return to this point and show that it remains valid as a Morita equivalence.

Throughout this section we assume that the Conditions \ref{cond} are in force.

\subsection{Construction of a particular idempotent} \

We would like to find a type which behaves well under restriction from $G$ to $G^\sharp$.
As this is rather complicated, we start with a simpler goal: an idempotent in
$\cH (M)$ which is suitable for restriction from $M$ to $M^\sharp$.

Recall from \cite{SeSt6} that there exists an $\fs$-type $(K_G, \lambda_G)$, \label{i:32}
and that it can be constructed as a cover of a type $(K,\lambda)$ for \label{i:82}
$\fs_M = [L,\omega]_M$.  We refer to \cite[Section 8]{BuKu3} for the notion
of a cover of a type.  For the moment, it suffices to know that 
$K = K_G \cap M$ and that $\lambda$ is the restriction of $\lambda_G$ to $K$. 

From \eqref{eq:1.2} we know that $N_G ([L,\omega]_L) \subset M$ and by Condition 
\ref{cond} $PM$ is a parabolic subgroup with Levi factor $M$. In this situation
\cite[Theorem 12.1]{BuKu3} says that there is an algebra isomorphism
\begin{equation}\label{eq:3.50}
e_{\lambda_G} \cH (G) e_{\lambda_G} \cong e_\lambda \cH (M) e_\lambda
\end{equation}
and that the normalized parabolic induction functor
\[
I_{PM}^G : \Rep^{\fs_M}(M) \to \Rep^\fs (G) 
\]
is an equivalence of categories. 

By Conditions \ref{cond} and \eqref{eq:1.1} we may assume that 
$(K,\lambda)$ factors as
\begin{equation}\label{eq:3.3}
\begin{split}
& K = \prod\nolimits_i (K \cap M_i) =: \prod\nolimits_i K_i , \\
& (\lambda, V_\lambda) = 
\Big( \bigotimes\nolimits_i \lambda_i, \bigotimes\nolimits_i V_{\lambda_i} \Big) . 
\end{split}
\end{equation}
Moreover we may assume that, whenever $m_i = m_j$ and $\omega_i$ and $\omega_j$
differ only by a character of $L_i / L_i^\sharp$, $K_i = K_j$ and $\lambda_i$ and
$\lambda_j$ also differ only by a character of $K_i / (K_i \cap L_i^\sharp)$. 
We note that these assumptions imply that $\mf R_\fs^\sharp$ normalizes $K$.
By respectively \eqref{eq:3.50}, \eqref{eq:1.5} and \eqref{eq:3.3} there are isomorphisms
\begin{multline}
e_{\lambda_G} \cH (G) e_{\lambda_G} \cong \cH (M,\lambda) \otimes_\C \End_\C (V_\lambda)
\cong \bigotimes\nolimits_i  \cH (M_i ,\lambda_i) \otimes_\C \End_\C (V_{\lambda_i}) .
\end{multline}

We need more specific information about the type $(K,\lambda)$ in $M$. To study this and 
the related types $(K,w(\lambda) \otimes \gamma)$  we will make ample use of the theory 
developed by S\'echerre and Stevens \cite{Sec3,SeSt4,SeSt6}.

In \cite{Sec3} $(K,\lambda)$ arises as a cover of a $[L,\omega]_L$-type $(K_L,\lambda_L)$.
In particular $\lambda$ is trivial on both $K \cap U$ \label{i:83}
and $K \cap \overline{U}$, where $U$ and $\overline{U}$ are the unipotent radicals of
$P \cap M$ and of the opposite parabolic subgroup of $M$, and $\lambda_L$ is the 
restriction of $\lambda$ to $K_L = K \cap L$.

\begin{prop}\label{prop:3.3}
We can choose the $\fs_M$-type $(K,\lambda)$ such that, for all 
$(w,\gamma) \in \Stab (\fs)$, $(K,w(\lambda) \otimes \gamma)$ is conjugate to 
$(K,\lambda)$ by an element \label{i:c} $c_\gamma \in L$. Moreover $c_\gamma Z(L)$ lies in a compact 
subgroup of $L / Z(L)$ and we can arrange that $c_\gamma$ depends only on the isomorphism 
class of $w (\lambda) \otimes \gamma \in \Irr (K)$.
\end{prop}
\emph{Remark.} For $\GL_n (F)$ very similar results were proven in \cite[\S 4.2]{GoRo1},
using \cite{BuKu4}. 
\begin{proof}
By definition $(K,\lambda)$ and $(K,w(\lambda) \otimes \gamma)$ are both types for
$[L,\omega]_M$. By Conditions \ref{cond} and \eqref{eq:3.3} they differ only by a character
of $M / M^\sharp Z(G) \cong G / G^\sharp Z(G)$, which automatically lies in
$X^L (w(\fs) \otimes \gamma) = X^L (\fs)$. Hence it suffices to prove the proposition in
the case $w = 1, \gamma \in X^L (\fs)$. This setup implies that we consider $(w,\gamma)$
only modulo isomorphism of the representations $w(\lambda) \otimes \gamma$.

In view of the factorizations $M = \prod_i M_i$ and \eqref{eq:3.3}, we can treat
the various $i$'s separately. Thus we may assume that $M_i = G$. 
To get in line with \cite{Sec3}, we temporarily change 
the notation to $G = \GL_{m_i e}(D), L = \GL_{m_i} (D)^e, J_P = K_i$ and $\omega = 
\omega_i^{\otimes e}$. We need a type $(J_P,\lambda_P)$ for $[L,\omega]_G$ with suitable
properties. We will use the one constructed in \cite{Sec3} as a cover of a simple type 
$(J_i^e ,\lambda_i^{\otimes e})$ for $[L,\omega]_L$. Analogously there is a cover 
$(J_P,\lambda_P \otimes \gamma)$ of the $[L,\omega]_L$-type \label{i:29}
\[
(J_i^e,\lambda_i^{\otimes e} \otimes \gamma) = 
(J_i^e, (\lambda_i \otimes \gamma)^{\otimes e}) .
\]
In these constructions $(J_i,\lambda_i)$ and $(J_i,\lambda_i \otimes \gamma)$ are two
maximal simple types for the supercuspidal inertial equivalence class 
$[\GL_{m_i} (D),\omega_i]_{\GL_{m_i} (D)}$. According to
\cite[Corollary 7.3]{SeSt6} they are conjugate, say by $c_i \in \GL_{m_i} (D)$.
Then $(J_i^e,\lambda_i^{\otimes e})$ and $(J_i^e,\lambda_i^{\otimes e} \otimes \gamma)$
are conjugate by
\begin{equation*}
c_{\gamma,i} := \text{diag}(c_i,c_i,\ldots,c_i) \in L .
\end{equation*}
Recall that $P$ is the parabolic subgroup of $G$ generated by $L$ and the upper 
triangular matrices. Let $U$ be the unipotent radical of $P$ and $\overline{U}$ the
unipotent radical of the parabolic subgroup opposite to $P$. The group $J_P$
constructed in \cite[\S 5.2]{Sec3} and \cite[\S 5.5]{SeSt4} admits an Iwahori 
decomposition
\begin{equation}\label{eq:3.8}
J_P = (J_P \cap \overline{U}) (J_P \cap L) (J_P \cap U) =
(H^1 \cap \overline{U}) \, J_i^e \, (J \cap U) . 
\end{equation}
Let us elaborate a little on the subgroups $H^1$ and $J_P \subset J$ of $G$. 
In \cite{SeSt4} a certain stratum $[\mathfrak C,n_0,0,\beta]$ is associated to 
$(\GL_{m_i} (D), \omega_i)$, which gives rise to compact open subgroups $H^1_i$ and 
$J_i$ of $\GL_{m_i} (D)$. From this stratum S\'echerre \cite[5.2.2]{Sec3} defines 
another stratum $[\mathfrak A,n,0,\beta]$ associated to
$(L,\omega_i^{\otimes e})$, which in the same way produces $H^1$ and $J$. The procedure
entails that $H^1$ and $J$ can be obtained by putting together copies of $H_i^1, \; J_i$
and their radicals in block matrix form. 
The proofs of \cite[Theorem 7.2 and Corollary 7.3]{SeSt6} show that we can take 
$c_i$ such that it normalizes $J_i$ and $H^1_i$. Then it follows from the explicit 
relation between the above two strata that $c_{\gamma,i}$ normalizes $J$ and $H^1$.
Notice also that $c_{\gamma,i}$ normalizes $\overline{U}$ and $U$, because it lies in $L$.
Hence $c_{\gamma,i}$ normalizes $J_P$ and its decomposition \eqref{eq:3.8}.

By definition \cite[5.2.3]{Sec3} the representation $\lambda_P$ of $J_P$ is trivial on 
$J_P \cap \overline{U}$ and on $J_P \cap U$, whereas its restriction to
$J_P \cap L$ equals $\lambda_i^{\otimes e}$. As $c_i$ conjugates $\lambda_i$ to
$\lambda_i \otimes \gamma$, we deduce that $c_{\gamma,i}$ conjugates $(J_P,\lambda_P)$ to
$(J_P,\lambda_P \otimes \gamma)$. 

To get back to the general case we recall that $M = \prod_i M_i$ and we put
\begin{equation}\label{eq:3.9}
c_\gamma := \prod\nolimits_i c_{\gamma,i} = 
\prod\nolimits_i\text{diag}(c_i,c_i,\ldots,c_i) . 
\end{equation}
It remains to see that $c_\gamma$ becomes a compact element in
$L / Z(L)$. Since $J_i$ is open and compact, its fixed points in the semisimple 
Bruhat--Tits building $\mathcal B (\GL_{m_i} (D))$ form a nonempty bounded subset. 
Then $c_i$ stabilizes this subset, so by the Bruhat--Tits fixed 
point theorem $c_i$ fixes some point $x_i \in \mathcal B (\GL_{m_i}(D))$. But the 
stabilizer of $x_i$ is a compact modulo centre subgroup, so in particular $c_i$ 
is compact modulo centre. Therefore the image of $c_\gamma$ in
$L / Z(L)$ is a compact element.
\end{proof}

In the above proof it is also possible to replace $(J_P,\lambda_P)$
by a sound simple type in the sense of \cite{SeSt6}. Indeed, the group $J$ 
from \cite[\S 5]{Sec3} is generated by $J_P$ and $J \cap \overline{U}$,
so it is also normalized by $c_{\gamma,i}$. By \cite[Proposition 5.4]{Sec3} 
\[
\big( J,\mathrm{Ind}_{J_P}^J (\lambda_P) \big) \quad \text{and} \quad 
\big( J,\mathrm{Ind}_{J_P}^J (\lambda_P \otimes \gamma) \big) 
\]
are sound simple types. The above proof also shows that they are conjugate by 
$c_{\gamma,i}$. As noted in the proof of \cite[Proposition 5.5]{Sec3}, there is
a canonical support preserving algebra isomorphism
\begin{equation}\label{eq:3.14}
\cH (M_i,\lambda_P) \cong \cH \big( M_i,\mathrm{Ind}_{J_P}^J (\lambda_P) \big) .
\end{equation}
In particular the structure theory of the Hecke algebras in \cite{Sec3} 
also applies to our types $(K,\lambda)$.

We write \label{i:33}
\[
L^1 := \bigcap\nolimits_{\chi \in X_\nr (L)} \ker \chi .
\]
Notice that $G^1 = \{ g \in G \mid \Nrd (g) \in \mf o_F^\times \}$ is the
group generated by all compact subgroups of $G$. Hence $L^1$ is 
the group generated by all compact subgroups of $L$.

We fix a choice of elements $c_\gamma \in L$ as in Proposition \ref{prop:3.3}, 
such that $c_\gamma \in L^1$ whenever possible. 
This determines subgroups \label{i:51} \label{i:74}
\begin{equation}\label{eq:3.86}
\begin{array}{lll}
X^L (\fs)^1 & := & \{ \gamma \in X^L (\fs) \mid c_\gamma \in L^1 \} ,\\
\Stab (\fs, P \cap M)^1 & := & 
\{ (w,\gamma) \in \Stab (\fs, P \cap M) \mid c_\gamma \in L^1 \} .
\end{array}
\end{equation}
Their relevance is that the $c_\gamma \in L^1$ can be used to construct 
larger $\fs_M$-types from $(K,\lambda)$, whereas the $c_\gamma$ with
$\gamma \in X^L (\fs) \setminus X^L (\fs)$ are unsuitable for that purpose.
We remark that in the split case $G = GL_n (F)$ it is known from 
\cite[Proposition 2.2]{BuKu2} that one can find $c_\gamma \in L^1$ for all
$\gamma \in X^L (\fs)$.

Consider the group \label{i:49}
\begin{equation*}
\Stab (\fs,\lambda) = \{ (w,\gamma) \in \Stab (\fs,P \cap M) \mid 
w (\lambda) \otimes \gamma \cong \lambda \text{ as } K\text{-representations} \} .
\end{equation*}
The elements of this group are precisely the $(w,\gamma) \in \Stab (\fs, P \cap M)$ 
for which $e_{w(\lambda) \otimes \gamma} = e_\lambda$.

\begin{lem}\label{lem:4.2}
Projection on the first coordinate gives a short exact sequence
\[
1 \to X^L (\fs) \cap \Stab (\fs,\lambda) \to 
\Stab (\fs,\lambda) \to \mf R_\fs^\sharp \to 1 .
\]
The inclusion $X^L (\fs) \to \Stab (\fs, P \cap M)$ induces a group isomorphism
\[
X^L (\fs) / \big( X^L (\fs) \cap \Stab (\fs,\lambda) \big) \to 
\Stab (\fs, P \cap M) / \Stab (\fs,\lambda) .
\]
\end{lem}
\begin{proof}
Let $(w,\gamma) \in \Stab (\fs, P \cap M)$. Then 
$w(\lambda) \otimes \gamma \cong \lambda \otimes \gamma'$ for some 
$\gamma' \in X^L (\fs)$ and $(w,\gamma) \in \Stab (\fs,\lambda)$ if and only if 
$\gamma' \in \Stab (\fs,\lambda)$. Hence all
the fibers of $\Stab (\fs,\lambda) \to \mf R_\fs^\sharp$ have the same cardinality,
namely $| X^L (\fs) \cap \Stab (\fs,\lambda) |$. The required short exact sequence 
follows. The asserted isomorphism of groups is a direct consequence thereof.
\end{proof}

Motivated by Lemma \ref{lem:4.2} we abbreviate \label{i:76} \label{i:77} \label{i:78}
\begin{equation}\label{eq:3.73}
\begin{aligned}
& X^L (\fs , \lambda) := X^L (\fs) \cap \Stab (\fs,\lambda) ,\\
& X^L (\fs / \lambda) := X^L (\fs) / X^L (\fs , \lambda) ,\\
& X^L (\fs / \lambda)^1 := X^L (\fs)^1 / X^L (\fs , \lambda) .
\end{aligned}
\end{equation}
The latter two groups are isomorphic to respectively 
\[
\Stab (\fs, P \cap M) / \Stab (\fs,\lambda)
\text{ and } \Stab (\fs, P \cap M)^1 / \Stab (\fs,\lambda).
\]
By Lemma \ref{lem:4.2} the element $\sum_{\gamma \in X^L (\fs / \lambda)}
e_{\lambda \otimes \gamma} \in \cH (K)$ is well-defined and idempotent.
Clearly this element is invariant under $X^L (\fs)$, which makes it more suitable
to study the restriction of $\Rep^\fs (G)$ to $G^\sharp$. However, in some
cases this idempotent sees only a too small part of a $G$-representation.
This will become apparant in the proof of Proposition \ref{prop:3.2}.d.
We need to replace it by a larger idempotent, for which we use the following
lemma. \label{i:21} \label{i:34} \label{i:64}

\begin{lem}\label{lem:3.30}
Let $(\omega,V_\omega) \in \Irr (L)$ be supercuspidal and write 
$\fs = [L,\omega]_G$. There exist a subgroup $H_\lambda \subset L$ and a subset 
$[L / H_\lambda] \subset L$ such that:
\enuma{
\item $[L / H_\lambda]$ is a set of representatives for $L / H_\lambda$, where
$H_\lambda \subset L$ is normal, of finite index and contains $L^\sharp Z(L)$.
\item Every element of $[L / H_\lambda]$ commutes with $W_\fs^\sharp$ and 
has finite order in $L / Z(L)$.
\item For every $\chi \in X_\nr (L)$ the space 
\[
\sum\nolimits_{a \in [L / H_\lambda]} \sum\nolimits_{\gamma \in X^L (\fs)}
a e_{\lambda_L \otimes \gamma} a^{-1} V_{\omega \otimes \chi}
\]
intersects every $L^\sharp$-isotypical component of $V_{\omega \otimes \chi}$ 
nontrivially.
\item For every $(\pi,V_\pi) \in \Irr^{\fs_M}(M)$ the space 
\[
\sum\nolimits_{a \in [L / H_\lambda]} \sum\nolimits_{\gamma \in X^L (\fs)}
a e_{\lambda \otimes \gamma} a^{-1} V_\pi
\]
intersects every $M^\sharp$-isotypical component of $V_\pi$ nontrivially.
}
\end{lem}
\begin{proof}
(a) First we identify the group $H_\lambda$. Recall the operators $I (\gamma,\omega)
\in \Hom_L ( \omega \otimes \gamma, \omega)$ from \eqref{eq:2.30}, with 
$\gamma \in X^L (\fs,\lambda) \cap X^L (\omega)$. 
Like in \eqref{eq:2.2} and \cite[Corollary 2.10]{HiSa}, these provide a 
decomposition of $L^\sharp$-representations
\[
\omega = \bigoplus_{\rho \in \Irr (\C[X^L (\fs,\lambda) \cap X^L (\omega), 
\kappa_\omega])} \Hom_{\C[X^L (\fs,\lambda) \cap X^L (\omega), \kappa_\omega]} 
(\rho, \omega) \otimes \rho. 
\]
Let us abbreviate it to \label{i:65}
\begin{equation}\label{eq:3.81}
V_\omega = \bigoplus\nolimits_\rho V_{\omega,\rho} .
\end{equation}
It follows from Proposition \ref{prop:2.6} and \eqref{eq:2.2} that all the 
summands $V_{\omega,\rho}$ are $L$-conjugate and that 
$\Stab_L (V_{\omega,\rho})$ is a finite index normal subgroup which contains
$L^\sharp Z(L)$. This leads to a bijection
\begin{equation}\label{eq:3.79}
\Irr \big( \C [X^L (\fs,\lambda) \cap X^L (\omega),\kappa_\omega] \big) 
\; \longleftrightarrow \; L / \Stab_L (V_{\omega,\rho}) .
\end{equation}
We claim that 
\begin{equation}\label{eq:3.78}
\sum\nolimits_{\gamma \in X^L (\fs)} e_{\lambda_L \otimes \gamma} V_\omega
\end{equation}
is an irreducible representation of a subgroup $N \subset L$ 
that normalizes $K_L$. From \cite[Th\'eor\`eme 4.6]{Sec3} it is known that 
\begin{equation}\label{eq:3.O}
e_{\lambda_L \otimes \gamma} \cH (L) e_{\lambda_L \otimes \gamma} \cong
\mc O (T_\fs) \otimes \End_\C (V_{\lambda_L \otimes \gamma}) ,
\end{equation}
where $\End_\C (V_{\lambda_L \otimes \gamma})$ corresponds to the subalgebra
$e_{\lambda_L \otimes \gamma} \cH (K) e_{\lambda_L \otimes \gamma}$. Since\\
$(K_L,\lambda_L \otimes \gamma)$ is a type for $[L,\omega \otimes \gamma]_L$,
every $e_{\lambda_L \otimes \gamma} V_\omega$ is isomorphic, as module over the
right hand side of \eqref{eq:3.O}, to $\C_t \otimes V_{\lambda_L \otimes \gamma}$ 
for a unique $t \in T^\fs$. The action of $K_L$ goes via 
$e_{\lambda_L \otimes \gamma} \cH (K) e_{\lambda_L \otimes \gamma}$, so 
$e_{\lambda_L \otimes \gamma} V_\omega$ is already irreducible as $K_L$-representation.
Since $Z(L)$-stabilizes this vector space, it also irreducible as representation of
the group $K_L Z(L)$. 

These representations, with $\gamma \in X^L (\fs / \lambda)$ are inequivalent and 
permuted transitively by the elements $c_\gamma$ from Proposition \ref{prop:3.3}. 
Hence \eqref{eq:3.78} is irreducible as a representation of the group $N$ generated 
by $K_L Z(L)$ and the $c_\gamma$.

Suppose now that \eqref{eq:3.78} intersects both $V_{\omega,\rho_1}$ and 
$V_{\omega,\rho_2}$ nontrivially. By the above claim, $N$ contains
an element that maps $V_{\omega,\rho_1}$ to $V_{\omega,\rho_2}$. It follows that,
under the bijection \eqref{eq:3.79}, 
the set of $\rho$'s such that $V_{\omega,\rho}$ intersects \eqref{eq:3.78} 
nontrivially corresponds to a subgroup of $L / \Stab_L (V_{\omega,\rho})$, 
say $H_\lambda / \Stab_L (V_{\omega,\rho})$. Because $L / \Stab_L (V_{\omega,\rho})$
was already finite and abelian, $H_\lambda$ has the desired properties.

We note that none of the above changes if we twist $\omega$ by an unramified
character of $L$.\\
(b) Recall that $L = \prod_i GL_{m_i} (D)^{e_i}$ and that the reduced norm map
$D^\times \to F^\times$ is surjective. It provides a group isomorphism
\[
L / H_\lambda \to F^\times / \Nrd (H_\lambda) ,
\]
and Nrd$(H_\lambda)$ contains Nrd$(Z(L)) = F^{\times e}$ where $e$ is the 
greatest common divisor of the numbers $m_i$. We can choose explicit 
representatives for $L / H_\lambda$. It suffices to use elements $a$ whose 
components $a_i \in \GL_{m_i} (D)$ are powers of some element of the form 
\[
\begin{pmatrix}
0 & 1 & 0 & & \cdots & 0  \\
0 & 0 & 1 & & \cdots & 0 \\
\vdots & \vdots \\
0 & 0 & & \cdots & 0 & 1 \\
d_i & 0 & & \cdots & 0 & 0
\end{pmatrix} \in GL_{m_i}(D) ,
\]
that is, the matrix of the cyclic permutation $(1 \to m_i \to m_i - 1 \to  
\cdots 3 \to 2)$, with one entry replaced by an element $d_i \in D^\times$. 
In this way we assure that $a$ has finite order in $L / Z(L)$, at most $e d$. 

If two factors $L_i = \GL_{m_i}(D)$ and $L_j = \GL_{m_j}(D)$ of $L = \prod_i 
L_i^{e_i}$ are conjugate via an element of $W_\fs^\sharp$, then $m_i = m_j$ and
the corresponding supercuspidal representations $\omega_i$ and $\omega_j$ differ
only by a character of $\GL_{m_i}(D)$, say $\eta$. As in the proof of Proposition
\ref{prop:3.3}, let $(J_i,\lambda_i)$ be a simple type for $(L_i,\omega_i)$. As in 
\eqref{eq:3.3} we use the type $(J_i,\lambda \otimes \eta)$ for $(L_j,\omega_j)$.

Given $\gamma \in X^L (\fs,\lambda) \cap X^L (\omega)$, we can factor 
\[
I (\gamma,\omega) = \prod\nolimits_i I(\gamma,\omega_i)^{\oplus e_i} ,
\]
with $I(\gamma,\omega_i) \in \Hom_{L_i}(\omega_i \otimes \gamma,\omega_i)$.
Here we can simply take $I(\gamma,\omega_j) = I(\gamma, \omega_i)$. Then the 
decomposition of $V_{\omega_j}$ in isotypical subspaces $V_{\omega_i,\rho}$ for 
$\C[X^L (\fs,\lambda) \cap X^L (\omega),\kappa_\omega]$, like \eqref{eq:3.81} is 
the same as that of $V_{\omega_i}$, and 
\[
V_{\omega,\rho} = \bigoplus\nolimits_i V_{\omega_i,\rho}^{\oplus e_i} .
\]
Suppose now that a component $a_i$ of $a$ as above
maps $V_{\omega_i,\rho}$ to $V_{\omega_i,\rho'}$. Then $a_i$ also maps 
$V_{\omega_j,\rho}$ to $V_{\omega_j,\rho'}$, so we may take $a_j = a_i$. With
this construction $a = \prod_i a_i^{\oplus e_i}$ commutes with $W_\fs^\sharp$.

We fix such a set of representatives $a$ and denote it by 
\begin{equation}\label{eq:3.64}
[L / H_\lambda] = \{ a_l \mid l \in L / H_\lambda \} .
\end{equation}
(c) Let $\chi \in X_\nr (L)$ and consider any $\rho \in \Irr (\C[X^L (\fs,\lambda) 
\cap X^L (\omega), \kappa_\omega])$. By construction 
\begin{equation}\label{eq:3.80}
\sum\nolimits_{a \in [L / H_\lambda]} \sum\nolimits_{\gamma \in X^L (\fs)}
a e_{\lambda_L \otimes \gamma} a^{-1} V_{\omega \otimes \chi}
\end{equation}
intersects $V_{\omega \otimes \chi,\rho}$ nontrivially. All the
idempotents $a e_{\lambda_L \otimes \gamma} a^{-1}$ are invariant under 
$X^L (\fs,\lambda) \cap X^L (\omega)$ because $e_{\lambda_L}$ is. Hence 
$a e_{\lambda_L \otimes \gamma} a^{-1} V_{\omega \otimes \chi}$ is nonzero
for at least one of these idempotents. The action of $X^L (\omega) /
X^L (\omega) \cap \Stab (\fs,\lambda)$ permutes the idempotents 
$a e_{\lambda_L \otimes \gamma} a^{-1}$ faithfully, so by Frobenius reciprocity
the space 
\[
\sum\nolimits_{\gamma \in X^L (\fs)} a e_{\lambda_L \otimes \gamma} a^{-1} 
V_{\omega \otimes \chi}
\]
contains all irreducible representations of 
$\C [X^L (\omega \otimes \chi),\kappa_{\omega \otimes \chi}]$ that contain $\rho$.
As $\rho \in \Irr (\C[X^L (\fs,\lambda) \cap X^L (\omega), \kappa_\omega])$ was
arbitrary, \eqref{eq:3.80} contains all irreducible representations of
$\C [X^L (\omega \otimes \chi),\kappa_{\omega \otimes \chi}]$. With \eqref{eq:2.2}
(for $L$) this says that \eqref{eq:3.80} intersects every $L^\sharp$-isotypical
component of $V_{\omega \otimes \chi}$ nontrivially.\\
(d) Let $\pi \in \Irr^{\fs_M}(M)$ and choose $\chi \in X_\nr (L)$ such that $\pi$
is a subquotient of $I_{P \cap M}^M (\omega \otimes \chi)$. 
Lemma \ref{lem:2.4}.d. in combination with the equality $W (M,L) = W_\fs$ shows that 
\[
X^M (\pi) \subset X^L (\omega \otimes \chi) = X^L (\omega) .
\] 
Since $(K,\lambda \otimes \gamma)$ is a type, 
\[
a e_{\lambda \otimes \gamma} a^{-1} V_\pi \neq 0
\]
for all possible $a,\gamma$. The group $X^L (\fs,\lambda) \cap X^M (\pi)$ 
effects a decomposition of the $M^\sharp$-representation $V_\pi$ by means of the 
operators $I_{P \cap M}^M (\gamma, \omega \otimes \chi)$. Analogous to \eqref{eq:3.81}
we write it as \label{i:63}
\begin{equation}\label{eq:3.82}
V_\pi = \bigoplus\nolimits_{\rho\in \Irr (\C[X^L (\fs,\lambda) \cap X^M (\pi), 
\kappa_\omega])} V_{\pi,\rho} .
\end{equation}
The construction of $H_\lambda$ entails that 
\[
\sum\nolimits_{a \in [L / H_\lambda]} \sum\nolimits_{\gamma \in X^L (\fs)}
a e_{\lambda \otimes \gamma} a^{-1} V_\pi
\]
intersects every summand $V_{\pi,\rho}$ of \eqref{eq:3.82} nontrivially. 
Now the same argument as for part (c) shows that this space intersects every
$M^\sharp$-isotypical component of $V_\pi$ nontrivially.
\end{proof}

With \eqref{eq:3.73} and Lemma \ref{lem:3.30} we construct some additonal idempotents:
\label{i:04}\label{i:10}\label{i:12}\label{i:13}
\begin{equation}\label{eq:3.61}
\begin{aligned}
& e_{\mu_L} := \sum\nolimits_{\gamma \in X^L (\fs / \lambda)}
e_{\lambda_L \otimes \gamma} \in \cH (K_L), \\
& e^\fs_L := \sum\nolimits_{a \in [L / H_\lambda]} a e_{\mu_L} a^{-1} \in \cH (L) , \\
& e_\mu := \sum\nolimits_{\gamma \in X^L (\fs / \lambda)}
e_{\lambda \otimes \gamma} \in \cH (K), \\
& e^\fs_M := \sum\nolimits_{a \in [L / H_\lambda]} a e_\mu a^{-1} \in \cH (M).
\end{aligned}
\end{equation}

\begin{lem}\label{lem:3.2}
The four elements in \eqref{eq:3.61} are idempotent and $\Stab (\fs,P \cap M)$-invariant.
Furthermore $e_{\mu_L}, e^\fs_L \in \cH (L)^{\fs_L}$ and $e_\mu, e^\fs_M \in \cH (M)^{\fs_M}$.
\end{lem}
\begin{proof}
We only write down the proof for the last two elements, the argument for the first two
is analogous.

We already observed that the different idempotents $e_{\lambda \otimes \gamma}$ are
orthogonal, so that their sum $e_{\mu}$ is again idempotent. We claim that 
$e = a_l e_{\lambda \otimes \gamma} a_l^{-1}$ and $e' = a_{l'} e_{\lambda \otimes 
\gamma'} a_{l'}^{-1}$ are orthogonal unless $l = l'$ and $\gamma = \gamma'$. 

By construction, the images
of $e$ and $e'$ in $\End_\C ( V_{I_{P \cap M}^M \omega} )$ are orthogonal. This remains 
true if we twist $\omega$ by an unramified character $\chi \in X_\nr (L)$. But the 
$M$-representations $I_{P \cap M}^M (\omega \otimes \chi)$ together generate the entire 
category $\Rep^{\fs_M}(M)$. Hence $e e' = e' e = 0$ on every representation in 
$\Rep^{\fs_M}(M)$. 

Since $e_\lambda \in \cH (M)^{\fs_M}$ and that algebra is stable under conjugation with
elements of $M$ and under $\Stab (\fs)$ by 
\eqref{eq:3.46}, all the $a e_{\lambda \otimes \gamma} a^{-1}$ lie in $\cH (M)^{\fs_M}$.
Thus $e,e' \in \cH (M)^{\fs_M}$, and we can conclude that they are indeed
orthogonal. This implies that $e^\fs_M \in \cH (M)^{\fs_M}$ is idempotent.

Since $e_{\lambda \otimes \gamma}$ is invariant under $X^L (\fs,\lambda)$, so is 
$e_{\mu}$. The action of $X^L (\fs)$ commutes with conjugation by any element of $M$, 
hence the sum over $\gamma \in X^L (\fs / \lambda)$ in the definition of $e_{\mu}$
makes it $X^L (\fs)$-invariant. 

By \eqref{eq:3.3} and the last part of Proposition \ref{prop:3.3}, $e_{\mu}$ is 
invariant under $\Stab (\fs,P \cap M)$ (but not necessarily under $W_\fs$). By Lemma 
\ref{lem:3.30}.b this remains the case after conjugation by any $a \in [L / H_\lambda]$.
Hence $a e_{\mu } a^{-1}$ and $e^\fs_M$ are also invariant under $\Stab (\fs,P \cap M)$.
\end{proof}

We can interpret the group $L / H_\lambda$ from Lemma \ref{lem:3.30} in a different way. 
Define \label{i:60} \label{i:79}
\begin{equation}\label{eq:3.89}
\begin{aligned}
& V_{\mu } := e_{\mu } V_\omega , \\
& X^L (\omega, V_{\mu }) := \big\{ \gamma \in X^L (\omega) \mid 
I(\gamma,\omega) |_{V_{\mu }} \in \C^\times \mathrm{id}_{V_{\mu }} \big\} .
\end{aligned}
\end{equation}

\begin{lem}\label{lem:3.13}
There is a group isomorphism $L / H_\lambda \cong \Irr \big( X^L (\omega,V_\mu ) \big)$. 
\end{lem}
\begin{proof}
We use the notation from the proof of Lemma \ref{lem:3.30}. Consider the twisted 
group algebra 
\begin{equation}\label{eq:3.88}
\C [X^L (\fs ,\lambda) \cap X^L (\omega), \kappa_\omega].
\end{equation} 
We noticed in \eqref{eq:3.81} and \eqref{eq:3.79} that all its irreducible
representations have the same dimension, say $\delta$. Let $C$ be the subgroup of
$X^L (\fs,\lambda) \cap X^L (\omega)$ that consists of all elements $\gamma$ for 
which $I(\gamma,\omega)$ acts as a scalar operator on $V_{\omega,\rho}$. Since all
the $V_{\omega,\rho}$ are $L$-conjugate, this does not depend on $\rho$. As the
dimension of \eqref{eq:3.88} equals $|X^L (\fs ,\lambda) \cap X^L (\omega)|$, we 
find that 
\[
[ X^L (\fs ,\lambda) \cap X^L (\omega) : C ] = \delta^2 
\text{ and } |C| = [ L : \Stab_L (V_{\omega,\rho}) ] .
\]
Since $C$ acts on every $V_{\omega,\rho}$ by a character, we can normalize the
operators $I(\gamma,\omega)$ such that $\kappa_\omega |_{C \times C} = 1$.
The subalgebra of \eqref{eq:3.88} spanned by the $I(\gamma,\omega)$ with
$\gamma \in C$ has dimension $|C|$, so every character of $C$ appears in 
$V_{\omega,\rho}$ for precisely one $\rho \in \Irr (\C [X^L (\fs ,\lambda) 
\cap X^L (\omega),\kappa_\omega])$. Now we see from \eqref{eq:3.79} that 
\[
C = \Irr (L / \Stab_L (V_{\omega,\rho})) \text{ and } 
\Irr (C) \cong L / \Stab_L (V_{\omega,\rho}) .
\]
Under the this isomorphism the subgroup $H_\lambda / \Stab_L (V_{\omega,\rho})$
corresponds to the set of character of $C$ that occur in $V_{\mu }$. That set 
can also be written as $\Irr (C / X^L (\omega,V_{\mu }))$. Hence the quotient 
\[
L / H_\lambda = \big( L / \Stab_L (V_{\omega,\rho}) \big) \Big/ 
\big( H_\lambda / \Stab_L (V_{\omega,\rho}) \big)  
\]
is isomorphic to $\Irr (X^L (\omega,V_{\mu }))$.
\end{proof}

\subsection{Descent to a Levi subgroup} \

Motivated by the isomorphisms \eqref{eq:3.83} we focus on 
$( \cH (G)^\fs )^{X^G (\fs)}$. We would like to replace it by a Morita equivalent
subalgebra of $\cH (M)^{\fs_M}$, where $\fs_M = [L,\omega]_M$ and $\fs = [L,\omega]_G$. 
However, the latter algebra is in general not 
stable under the action of $X^G (\fs)$. In fact, for $(w,\gamma) \in \Stab (\fs)$ we have
\begin{equation}\label{eq:3.46}
\gamma \cdot \cH (M)^{\fs_M} = \cH (M)^{[L,\omega \otimes \gamma^{-1}]_M} =
\cH (M)^{[L,w \cdot \omega]_M} = \cH (M)^{w(\fs_M)} .
\end{equation}
Let us regard $\mf R_\fs^\sharp$ from Lemma \ref{lem:2.2}.c as a group of permutation 
matrices in $G$. Then it acts on $M$ by conjugation and we can form the crossed product
\[
\cH (M \rtimes \mf R_\fs^\sharp ) = \cH (M) \rtimes \mf R_\fs^\sharp .
\]
We define $\cH (M \rtimes \mf R_\fs^\sharp )^\fs$ as the two-sided ideal of
$\cH (M \rtimes \mf R_\fs^\sharp )$ such that \label{i:22}
\[
\mathrm{Ind}_{PM \rtimes \mf R_\fs^\sharp}^G (V) \in \Rep^\fs (G) \text{ for all }
V \in \Mod (\cH (M \rtimes \mf R_\fs^\sharp )^\fs ) .
\] 

\begin{lem}\label{lem:3.14}
The algebra $\cH (M \rtimes \mf R_\fs^\sharp )^\fs$ equals 
$\big( \bigoplus\nolimits_{w \in \mf R_\fs^\sharp} \cH (M)^{w(\fs_M)} \big) 
\rtimes \mf R_\fs^\sharp$ .
\end{lem}
\begin{proof}
First we note that $\Res_{\cH (M)}^{\cH (M \rtimes \mf R_\fs^\sharp)^\fs}(V) \in
\sum_{w \in W_\fs \mf R_\fs^\sharp} \Mod (\cH (M)^{w (\fs_M)})$ for all eligible $V$, 
because these $w(\fs_M)$ are only inertial equivalence classes for $M$ 
which lift to $\fs$. Hence
\[
\cH (M \rtimes \mf R_\fs^\sharp)^\fs \subset
\sum_{r \in \mf R_\fs^\sharp} \sum_{w \in W_\fs} \cH (M)^{w(\fs_M)} r .
\]
The right hand side satisfies the defining property of $\cH (M \rtimes 
\mf R_\fs^\sharp)^\fs$, so both sides are equal.
Because $W_\fs \subset M$ and $\cH (M)^{\fs_M} \cH (M)^{w (\fs_M)} = 0$ for 
$w \in \mf R_\fs^\sharp \setminus \{1\}$, the right hand side is actually
a crossed product in the asserted way.
\end{proof}

By \eqref{eq:3.46} the algebra from Lemma \ref{lem:3.14} is stable under $X^G (\fs)$. 
We extend the action $\alpha$ of $X^L (\fs)$ on $\cH (M)$ to $\Stab (\fs)$ by
\label{i:01}
\begin{equation}\label{eq:3.35}
\alpha_{(w,\gamma)}(f) := w (\gamma^{-1} \cdot f) w^{-1} . 
\end{equation}
Given  $w \in \mf R_\fs^\sharp$, Lemma \ref{lem:2.4}.d shows that there exists a
$\gamma \in \Irr (L / L^\sharp Z(G))$ such that $(w,\gamma) \in \Stab (\fs)$,
and that $\gamma$ is unique up to $X^L (\fs)$. Hence $w \mapsto \alpha_{(w,\gamma)}$
determines a group action of $\mf R_\fs^\sharp$ on $(\cH (M) )^{X^L (\fs)}$. 
By \eqref{eq:3.46} this action stabilizes $(\cH (M)^{\fs_M} )^{X^L (\fs)}$.
Using this action, we can rewrite the $\alpha$-invariant subalgebra of 
$\cH (M \rtimes \mf R_\fs^\sharp )^\fs$ conveniently:

\begin{lem}\label{lem:3.3}
There is a canonical isomorphism 
\[
\Big( \cH (M \rtimes \mf R_\fs^\sharp )^\fs \Big)^{X^G (\fs)}
\cong \big( \cH (M)^{\fs_M} \big)^{X^L (\fs)} \rtimes \mf R_\fs^\sharp .
\]
\end{lem}
\begin{proof}
Using Lemma \ref{lem:3.14} and the fact that $X^G (\fs)$ fixes all elements of
$\C [\mf R_\fs^\sharp]$, we can rewrite 
\[
\Big( \cH (M \rtimes \mf R_\fs^\sharp )^\fs \Big)^{X^G (\fs)} \cong 
\Big( \big( \bigoplus\nolimits_{w \in \mf R_\fs^\sharp} \cH (M)^{w(\fs_M)} 
\big)^{X^L (\fs)} \rtimes \mf R_\fs^\sharp \Big)^{X^G (\fs) / X^L (\fs)} .
\]
By Lemma \ref{lem:2.4}.d this is
\[
\Big( \bigoplus\nolimits_{w_1,w_2 \in \mf R_\fs^\sharp} w_1 (\cH (M)^{\fs_M})^{X^L (\fs)} 
w_2^{-1} \Big)^{\mf R_\fs^\sharp} \cong \Big( \End_\C ( \C [\mf R_\fs^\sharp]) 
\otimes (\cH (M)^{\fs_M})^{X^L (\fs)} \Big)^{\mf R_\fs^\sharp} .
\]
In the right hand side the action of $\mf R_\fs^\sharp$ has become the regular
representation on $\End_\C ( \C [\mf R_\fs^\sharp])$ tensored with the action
$\alpha_{(w,\gamma)}$ as in \eqref{eq:3.35}. By a folklore result
(see \cite[Lemma A.3]{SolT} for a proof) the right hand side is isomorphic to
$\big( \cH (M)^{\fs_M} \big)^{X^L (\fs)} \rtimes \mf R_\fs^\sharp$.
\end{proof}

In Proposition \ref{prop:3.2} we will show that the algebras from Lemma \ref{lem:3.3}
are Morita equivalent with $( \cH (G)^\fs )^{X^G (\fs)}$. 

We recall from Lemma \ref{lem:3.2} that $e^\fs_M$ in \eqref{eq:3.61} is 
$\Stab (\fs,P \cap M)$-invariant, so 
from \eqref{eq:3.35} we obtain an action of $\Stab (\fs, P \cap M)$ on 
$e^\fs_M \cH (M \rtimes \mf R_\fs^\sharp )^\fs e^\fs_M$.

\begin{lem}\label{lem:3.10}
The following algebras are Morita equivalent:\\
$\cH (G)^\fs, \cH (M)^{\fs_M}, \cH (M \rtimes \mf R_\fs^\sharp )^\fs,
e^\fs_M \cH (M) e^\fs_M$ and 
$e^\fs_M \cH (M \rtimes \mf R_\fs^\sharp )^\fs e^\fs_M$.
\end{lem}
\begin{proof}
We will denote Morita equivalence with $\sim_M$. \label{i:35}
The Morita equivalence of $\cH (G)^\fs$ and $\cH (M)^{\fs_M}$ follows from the fact that
$N_G (\fs_L) \subset M$. It is given in one direction by
\begin{equation}\label{eq:3.7}
I_{PM}^G : \Mod (\cH (M)^{\fs_M}) = \Rep^{\fs_M}(M) \to \Rep^\fs (G) = \Mod (\cH (G)^\fs)
\end{equation}
and in the other direction by 
\begin{equation}\label{eq:3.6}
\mathrm{pr}_{\fs_M} \circ r^G_{PM} : \Rep^\fs (G) \to \Rep^{\fs}(M) \to \Rep^{\fs_M}(M),
\end{equation}
the normalized Jacquet restriction functor $r_{PM}^G$ followed by projection on 
the factor $\Rep^{\fs_M}(M)$ of $\Rep^{\fs}(M)$. Lemma \ref{lem:3.14} shows that 
\begin{equation}\label{eq:3.5}
\cH (M)^{\fs_M} \sim_M \cH (M \rtimes \mf R_\fs^\sharp )^\fs ,
\end{equation}
the equivalence being given by 
\[
\mathrm{Ind}_{\cH (M)^{\fs_M}}^{\cH (M \rtimes \mf R_\fs^\sharp )^\fs} 
= \mathrm{Ind}_M^{M \rtimes \mf R_\fs^\sharp} .
\]
With the bimodules $e^\fs_M \cH (M)^{\fs_M}$ and
$\cH (M)^{\fs_M} e^\fs_M$ we see that
\begin{equation}\label{eq:3.49}
e^\fs_M \cH (M) e^\fs_M = 
e^\fs_M \cH (M)^{\fs_M} e^\fs_M \sim_M 
\cH (M)^{\fs_M} e^\fs_M \cH (M)^{\fs_M} .
\end{equation}
Since $(K,\lambda)$ is an $\fs_M$-type, every module of
$\cH (M)^{\fs_M}$ is generated by its $\lambda$-isotypical vectors and a fortiori
by the image of $e^\fs_M$ in such a module. Therefore
\[
\cH (M )^{\fs_M} e^\fs_M \cH (M )^{\fs_M} = \cH (M )^{\fs_M} .
\]
The same argument, now additionally using \eqref{eq:3.5}, also shows that
\[
\cH (M \rtimes \mf R_\fs^\sharp )^\fs \sim_M e^\fs_M 
\cH (M \rtimes \mf R_\fs^\sharp )^\fs  e^\fs_M. \qedhere
\]
\end{proof}

The above lemma serves mainly as preparation for some more involved Morita equivalences:

\begin{prop}\label{prop:3.2} 
The following algebras are Morita equivalent to $(\cH (G)^\fs )^{X^G (\fs)}$:
\enuma{ 
\item $(\cH (M \rtimes \mf R_\fs^\sharp )^\fs )^{X^G (\fs)} \cong
\big( \cH (M)^{\fs_M} \big)^{X^L (\fs)} \rtimes \mf R_\fs^\sharp$;
\item $\cH (M)^{\fs_M} \rtimes \Stab (\fs,P \cap M)$;  
\item $ e^\fs_M \cH (M) e^\fs_M \rtimes \Stab (\fs,P \cap M)$;
\item $\big( e^\fs_M \cH (M) e^\fs_M \big)^{X^L (\fs)} \rtimes \mf R_\fs^\sharp$.
}
\end{prop}
\begin{proof}
For the definitions of the finite groups see page \pageref{eq:2.14}.\\
(a) The isomorphism between the two algebras is Lemma \ref{lem:3.3}.
Let $U$ be the unipotent radical of $PM$. As discussed in \cite{MeSo}, there are
natural isomorphisms \label{i:02} \label{i:03}
\begin{align*}
& I_{PM}^G (V) \cong C_c^\infty (G / U ) \otimes_{\cH (M)} V 
\qquad & V \in \Rep (M) ,\\
& r_{PM}^G (W) \cong C_c^\infty (U \backslash G) \otimes_{\cH (G)} W 
\qquad & W \in \Rep (G) .
\end{align*}
For $V \in \Rep^{\fs_M}(M)$ we may just as well take the bimodule 
$C_c^\infty (G / U ) \cH (M)^{\fs_M}$, and to get \eqref{eq:3.6} the bimodule
$\cH (M)^{\fs_M} C_c^\infty (U \backslash G)$ is suitable. But if we want to
obtain modules over $\cH (M \rtimes \mf R_\fs^\sharp )^\fs$, it is better to use
the bimodules
\begin{align*}
& C_c^\infty (G / U )^\fs := \bigoplus\nolimits_{w \in \mf R_\fs^\sharp} 
C_c^\infty (G / U ) \cH (M)^{w (\fs_M)} = C_c^\infty (G / U ) 
\cH (M \rtimes \mf R_\fs^\sharp )^\fs , \\
& C_c^\infty (U \backslash G)^\fs := \bigoplus\nolimits_{w \in \mf R_\fs^\sharp}  
\cH (M)^{w(\fs_M)} C_c^\infty (U \backslash G) = \cH (M \rtimes \mf R_\fs^\sharp )^\fs  
C_c^\infty (U \backslash G) .
\end{align*}
Indeed, we can rewrite \eqref{eq:3.7} as
\begin{align*}
I_{PM}^G (V) & \cong C_c^\infty (G / U ) \cH (M)^{\fs_M} \otimes_{\cH (M)^{\fs_M}} V \\
& =  C_c^\infty (G / U )^\fs \otimes_{\cH (M)^{\fs_M}} V \\
& \cong C_c^\infty (G / U )^\fs \otimes_{\cH (M \rtimes \mf R_\fs^\sharp)^\fs} 
\cH (M \rtimes \mf R_\fs^\sharp)^\fs  \otimes_{\cH (M)^{\fs_M}} V \\
& \cong C_c^\infty (G / U )^\fs \otimes_{\cH (M \rtimes \mf R_\fs^\sharp)^\fs} 
\mathrm{Ind}_{\cH (M)^{\fs_M}}^{\cH (M \rtimes \mf R_\fs^\sharp )^\fs} (V) .
\end{align*}
Similarly \eqref{eq:3.6} translates to
\begin{align*}
\mathrm{Ind}_M^{M \rtimes \mf R_\fs^\sharp} \circ \mathrm{pr}_{\fs_M} \circ r^G_{PM} 
& = \cH (M \rtimes \mf R_\fs^\sharp )^\fs \otimes_{\cH (M)^{\fs_M}} \cH (M)^{\fs_M}
C_c^\infty (U \backslash G) \otimes_{\cH (G)^\fs} W \\
& = \cH (M \rtimes \mf R_\fs^\sharp )^\fs C_c^\infty (U \backslash G) 
\otimes_{\cH (G)^\fs} W \\
& = C_c^\infty (U \backslash G)^\fs \otimes_{\cH (G)^\fs} W .
\end{align*}
These calculations entail that the bimodules
$C_c^\infty (G / U )^\fs$ and $C_c^\infty (U \backslash G)^\fs$ implement
\begin{equation}\label{eq:3.60}
\cH (G)^\fs \sim_M \cH (M \rtimes \mf R_\fs^\sharp )^\fs .
\end{equation}
These bimodules are naturally endowed with an action of $X^G (\fs)$, by pointwise 
multiplication of functions $G \to \C$. This action is obviously compatible with
the group actions on $\cH (G)^\fs$ and $\cH (M \rtimes \mf R_\fs^\sharp )^\fs$, in
the sense that 
\[
\gamma \cdot (f_1 f_2) = (\gamma \cdot f_1)(\gamma \cdot f_2) \quad \text{and} \quad
\gamma \cdot (f_2 f_3) = (\gamma \cdot f_2)(\gamma \cdot f_3)
\]
for $\gamma \in X^G (\fs), f_1 \in \cH (G)^\fs, f_2 \in C_c^\infty (G / U )^\fs,
f_3 \in \cH (M \rtimes \mf R_\fs^\sharp )^\fs$. Hence we may restrict \eqref{eq:3.60}
to functions supported on $\bigcap_{\gamma \in X^G (\fs)} \ker \gamma$, and we obtain
\begin{equation}\label{eq:3.N}
\begin{aligned}
& (C_c^\infty (G / U )^\fs )^{X^G (\fs)} 
\otimes_{(\cH (M \rtimes \mf R_\fs^\sharp )^\fs )^{X^G (\fs)}}
( C_c^\infty (U \backslash G)^\fs )^{X^G (\fs)} \cong (\cH (G)^\fs )^{X^G (\fs)} , \\
& ( C_c^\infty (U \backslash G)^\fs )^{X^G (\fs)} \otimes_{(\cH (G)^\fs )^{X^G (\fs)}}  
(C_c^\infty (G / U )^\fs )^{X^G (\fs)} \cong 
(\cH (M \rtimes \mf R_\fs^\sharp )^\fs )^{X^G (\fs)} .
\end{aligned}
\end{equation}
(b) Consider the idempotent
\[
p = |X^L (\fs)|^{-1} \sum\nolimits_{\gamma \in X^L (\fs)} \gamma 
\; \in \; \C [X^L (\fs)] .
\]
It is easy to see that the map
\[
\big( \cH (M)^{\fs_M} \big)^{X^L (\fs)} \to 
p (\cH (M)^{\fs_M} \rtimes X^L (\fs)) p \;:\; a \mapsto p a p 
\]
is an isomorphism of algebras \cite[Lemma A.2]{SolT}. 
Therefore $\big( \cH (M)^{\fs_M} \big)^{X^L (\fs)}$ is Morita equivalent with 
$(\cH (M)^{\fs_M} \rtimes X^L (\fs)) p (\cH (M)^{\fs_M} \rtimes X^L (\fs))$, 
via the bimodules $p(\cH (M)^{\fs_M} \rtimes X^L (\fs))$ and 
$(\cH (M)^{\fs_M} \rtimes X^L (\fs)) p$. Suppose that 
\[
(\cH (M)^{\fs_M} \rtimes X^L (\fs)) p (\cH (M)^{\fs_M} \rtimes X^L (\fs)) 
\subsetneq \cH (M)^{\fs_M} \rtimes X^L (\fs) .
\]
Then the quotient algebra
\[
\cH (M)^{\fs_M} \rtimes X^L (\fs) \Big/ 
(\cH (M)^{\fs_M} \rtimes X^L (\fs)) p (\cH (M)^{\fs_M} \rtimes X^L (\fs)) 
\]
is nonzero. This algebra is a direct limit of unital algebras, so it has an
irreducible module $V$ on which it does not act as zero. We can regard $V$ 
as an irreducible $\cH (M)^{\fs_M} \rtimes X^L (\fs)$-module with $p V = 0$. 
For any $\pi \in \Irr^{\fs_M}(M)$ we have $X^M (\pi) \subset X^L (\omega)$ 
since $W (M,L) = W_\fs$ and by Lemma \ref{lem:2.4}.d. By \eqref{eq:2.2} 
and \eqref{eq:2.25}
the decomposition of $V_\pi$ over $M^\sharp Z(G)$ is governed 
by $\C [X^M (\pi),\kappa_\omega]$.
Now Clifford theory (see for example \cite[Appendix A]{SolGHA}) says that, 
for any $\rho \in \Irr (\C [X^M (\pi),\kappa_\omega])$,
\[
\Ind_{\cH (M)^{\fs_M} \rtimes X^M (\pi)}^{\cH (M)^{\fs_M} \rtimes X^L (\fs)} 
(V_\pi \otimes \rho^\vee) 
\]
is an irreducible module over $\cH (M)^{\fs_M} \rtimes X^L (\fs)$. 
Moreover every irreducible $\cH (M)^{\fs_M} \rtimes X^L (\fs)$-module is of
this form, so we may take it as $V$. But by \eqref{eq:2.2} 
\begin{equation}\label{eq:3.76}
\rho \text{ appears in } V_\pi .
\end{equation}
Hence $V_\pi \otimes \rho^\vee$ has nonzero $X^L (\omega)$-invariant
vectors and $p V \neq 0$. This contradiction shows that 
\begin{equation}\label{eq:3.16}
(\cH (M)^{\fs_M} \rtimes X^L (\fs)) p (\cH (M)^{\fs_M} \rtimes X^L (\fs)) = 
\cH (M)^{\fs_M} \rtimes X^L (\fs) .
\end{equation}
Recall from Lemma \ref{lem:3.3} that
\[
(\cH (M^{\fs_M} \rtimes \mf R_\fs^\sharp )^\fs )^{X^G (\fs)} \cong
\big( \cH (M)^{\fs_M} \big)^{X^L (\fs)} \rtimes \mf R_\fs^\sharp 
= p (\cH (M)^{\fs_M} \rtimes \Stab (\fs,P \cap M)) p .
\]
The bimodules $p (\cH (M)^{\fs_M} \rtimes \Stab (\fs,P \cap M))$ and 
$(\cH (M)^{\fs_M} \rtimes \Stab (\fs,P \cap M)) p$ make it Morita equivalent with 
\[
(\cH (M)^{\fs_M} \rtimes \Stab (\fs,P \cap M)) p 
(\cH (M)^{\fs_M} \rtimes \Stab (\fs,P \cap M)),  
\]
which by \eqref{eq:3.16} equals $\cH (M)^{\fs_M} \rtimes \Stab (\fs,P \cap M)$.\\
(c) This follows from \eqref{eq:3.49} and Lemma \ref{lem:3.2}, upon applying 
$\rtimes \Stab (\fs,P \cap M)$ everywhere.\\
(d) First we want to show that
\begin{equation}\label{eq:3.77}
\big( e^\fs_M \cH (M) e^\fs_M \big)^{X^L (\fs)} \sim_M
e^\fs_M \cH (M) e^\fs_M \rtimes X^L (\fs) .
\end{equation}
To this end we use the same argument as in part (b), only with 
$e^\fs_M \cH (M) e^\fs_M$ instead of $\cH (M)^{\fs_M}$. 
Everything goes fine until \eqref{eq:3.76}. The corresponding statement in the present
setting would be that every irreducible module of $\C [X^M (\pi),\kappa_\pi]$
appears in $e^\fs_M V_\pi$. By \ref{eq:2.2} this is equivalent to saying that 
$e^\fs_M V_\pi$ intersects every $M^\sharp$-isotypical component of $V_\pi$
nontrivially, which is exactly Lemma \ref{lem:3.30}.d. Therefore this version of
\eqref{eq:3.76} does hold. The analogue of \eqref{eq:3.16} is now valid, 
and establishes \eqref{eq:3.77}. The bimodules for this Morita equivalence are 
\[
p( e^\fs_M \cH (M) e^\fs_M \rtimes X^L (\fs)) \text{ and } 
(e^\fs_M \cH (M) e^\fs_M \rtimes X^L (\fs)) p .
\]
The same argument as after \eqref{eq:3.16} makes clear how this implies the required
Morita equivalence
\[
e^\fs_M \cH (M) e^\fs_M \rtimes \Stab (\fs,P \cap M) \sim_M
\big( e^\fs_M \cH (M) e^\fs_M \big)^{X^L (\fs)} \rtimes \mf R_\fs^\sharp . \qedhere
\]
\end{proof}

From the above proof one can extract bimodules for the Morita equivalence
\begin{equation}\label{eq:3.63}
\big( e^\fs_M \cH (M) e^\fs_M \big)^{X^L (\fs)} \rtimes \mf R_\fs^\sharp
\sim_M \big( \cH (M)^{\fs_M} \big)^{X^L (\fs)} \rtimes \mf R_\fs^\sharp ,
\end{equation}
namely
\begin{equation}\label{eq:3.25}
\big( \cH (M)^{\fs_M} e^\fs_M \big)^{X^L (\fs)} 
\rtimes \mf R_\fs^\sharp \quad \text{and} \quad \big( e^\fs_M 
\cH (M)^{\fs_M} \big)^{X^L (\fs)} \rtimes \mf R_\fs^\sharp .
\end{equation}
It seems complicated to prove directly that these are Morita bimodules,
without the detour via parts (b) and (c) of Proposition \ref{prop:3.2}.

\subsection{Passage to the derived group} \
\label{par:morita2}

We study how Hecke algebras for $G^\sharp$ and for $G^\sharp Z(G)$
can be replaced by Morita equivalent algebras built from $\cH (G)$.
In the last results of this subsection we will show that a Morita
equivalent subalgebra $\cH (G^\sharp)^\fs$ is isomorphic to
subalgebras of $\cH (G)^\fs$ and of $\cH (M \rtimes \mf R_\fs^\sharp)^\fs$.

\begin{lem}\label{lem:3.1}
The algebra $\cH (G^\sharp Z(G))^\fs$ is Morita equivalent with 
$( \cH (G)^\fs )^{X^G (\fs)}$.
\end{lem}
\begin{proof}
Let $\mf o_D$ be the ring of integers of $D$. Let $C_l$ be the $l$-th congruence 
subgroup of $\GL_m (\mf o_D)$, and put $C'_l = C_l \cap G^\sharp Z(G)$. 
The group $G^\sharp Z(G) C_l$ is of finite index in $G$, because
$\Nrd (G^\sharp Z(G) C_l)$ contains both $F^{\times md}$ and an open neighborhood
$\Nrd (C_l)$ of $1 \in F^\times$.
By Lemma \ref{lem:2.4}.c we can choose $l$ so large, that every element of $X^G (\fs)$ 
is trivial on $C_l$ and that all representations in $\Rep^\fs (G)$ have nonzero 
$C_l$-invariant vectors. Let $e_{C_l} \in \cH (C_l)$ be the central idempotent 
associated to the trivial representation of $C_l$.
It is known from \cite[\S 3]{BeDe} that $(C_l,$triv) is a type, so the algebra 
\[        
\cH (G,C_l)^\fs = e_{C_l} \cH (G)^\fs e_{C_l}
\]
of $C_l$-biinvariant functions in $\cH (G)^\fs$ is Morita equivalent with $\cH (G)^\fs$. 
The Morita bimodules are $e_{C_l} \cH (G)^\fs$ and $\cH (G)^\fs e_{C_l}$. Since $X^G (\fs)$ 
fixes $e_{C_l}$, these bimodules carry an $X^G (\fs)$-action, which clearly is compatible
with the actions on $\cH (G)^\fs$ and $\cH (G,C_l)^\fs$. We can restrict the equations 
which make them Morita bimodules to the subspaces of functions $G \to \C$ supported on
$\bigcap_{\gamma \in X^G (\fs)} \ker \gamma$. We find that the bimodules
$( e_{C_l} \cH (G)^\fs )^{X^G(\fs)}$ and $( \cH (G)^\fs e_{C_l} )^{X^G (\fs)}$ provide
a Morita equivalence between 
\begin{equation}\label{eq:3.1}
( \cH (G)^\fs )^{X^G (\fs)} \quad \text{and} \quad ( \cH (G,C_l)^\fs )^{X^G (\fs)} .
\end{equation}
We saw in \eqref{eq:2.19} that $\Irr^\fs (G^\sharp Z(G))$ is a union of Bernstein 
components, in fact a finite union by Lemma \ref{lem:2.1}. Hence we may assume that every 
representation in $\Irr^\fs (G^\sharp Z(G))$ contains nonzero $C'_l$-invariant vectors. 
As $(C'_l,$triv) is a type, that suffices for a Morita equivalence between
\begin{equation}\label{eq:3.2}
\cH (G^\sharp Z(G))^\fs \quad \text{and} \quad \cH (G^\sharp Z(G),C'_l)^\fs . 
\end{equation}
We may assume that the Haar measures on $G$ and on $G^\sharp Z(G)$ are
chosen such that $C_l$ and $C'_l$ get the same volume. Then the natural injection
\[
C'_l \setminus G^\sharp Z(G) / C'_l \to C_l \setminus G / C_l
\]
provides an injective algebra homomorphism
\begin{equation}\label{eq:3.70}
\cH (G^\sharp Z(G),C'_l) \to \cH (G,C_l) ,
\end{equation}
whose image consists of the $\Irr (G / G^\sharp Z(G) C_l)$-invariants in $\cH (G,C_l)$.
Let $\mathfrak B (G)_l$ be the set of inertial equivalence classes 
for $G$ corresponding to the category of $G$-representations that are 
generated by their $C_l$-invariant vectors. The finite group 
$\Irr (G / G^\sharp Z(G) C_l)$ acts on it, and we denote the set of orbits by 
$\mathfrak B (G)_l / \sim$. Now we can write
\begin{multline*}
\bigoplus\nolimits_{\fs \in \mathfrak B (G)_l / \sim} \cH (G^\sharp Z(G), C'_l)^\fs =
\cH (G^\sharp Z(G),C'_l) \cong \\
\cH (G,C_l)^{\Irr (G / G^\sharp Z(G) C_l)} = \bigoplus\nolimits_{\fs \in \mathfrak B (G)_l 
/ \sim} ( \cH (G, C_l)^\fs )^{X^G (\fs)} .
\end{multline*}
By considering the factors corresponding to one $\fs$ on both sides we obtain an
isomorphism
\[
\cH (G^\sharp Z(G), C'_l)^\fs \cong ( \cH (G,C_l)^\fs )^{X^G (\fs)} .
\]
To conclude, we combine this with \eqref{eq:3.1} and \eqref{eq:3.2}.
\end{proof}

The Morita equivalences in parts (a) and (d) of Proposition \ref{prop:3.2}, 
for algebras associated to $G^\sharp Z(G)$,
have analogues for $G^\sharp$. For parts (b) and (c), which involve crossed
products by $\Stab (\fs,P \cap M)$, this is not clear.

\begin{lem}\label{lem:3.11}
The algebra $\cH (G^\sharp )^\fs$ is Morita equivalent with
\[
(\cH (G)^\fs )^{X^G (\fs) X_\nr (G)} \quad \text{and with} \quad
\big( \cH (M)^\fs \big)^{X^L (\fs) X_\nr (G)} \rtimes \mf R_\fs^\sharp . 
\]
\end{lem}
\begin{proof}
By \eqref{eq:2.15} we have 
\begin{equation}\label{eq:3.75}
\cH (G^\sharp )^\fs \cong \big( \cH (G^\sharp Z(G))^\fs \big)^{X_\nr (Z(G))} .
\end{equation}
As $X_\nr (G / Z(G)) \subset X^L (\fs) \subset X^G (\fs)$, every 
$\chi \in X_\nr (Z(G))$ extends in a unique way to a character of $\cH (G)^{X^G (\fs)}$.
In other words, we can identify
\begin{equation}\label{eq:3.45}
X_\nr (Z(G)) = X_\nr (G  ) \text{ in } \Irr (G / G^\sharp) / X^G (\fs) .
\end{equation}
All the bimodules involved in \eqref{eq:3.1} and \eqref{eq:3.2} carry a compatible action
of \eqref{eq:3.45}. We can restrict the proofs of these Morita equivalences to smooth 
functions supported on $G^1 = \bigcap_{\chi \in X_\nr (G  )} \ker \chi$. 
That leads to a Morita equivalence
\[
\big( \cH (G^\sharp Z(G))^\fs \big)^{X_\nr (Z(G))} \sim_M 
(\cH (G)^\fs )^{X^G (\fs) X_\nr (G  )} .
\]
Let us take another look at the Morita equivalence \eqref{eq:3.60}, between 
$\cH (G)^\fs$ and $\cH (M \rtimes \mf R_\fs^\sharp )^\fs$. The argument between 
\eqref{eq:3.60} and \eqref{eq:3.N} also works with $X^G (\fs) X_\nr (G  )$
instead of $X^G (\fs)$, and provides a Morita equivalence
\begin{equation}\label{eq:3.M}
(\cH (G)^\fs )^{X^G (\fs) X_\nr (G  )} \sim_M
(\cH (M \rtimes \mf R_\fs^\sharp)^\fs )^{X^G (\fs) X_\nr (G  )} .
\end{equation}
The isomorphism in Lemma \ref{lem:3.3} is $X_\nr (G  )$-equivariant, so
it restricts to
\[
(\cH (M \rtimes \mf R_\fs^\sharp)^\fs )^{X^G (\fs) X_\nr (G  )}  \cong
(\cH (M)^{\fs_M} )^{X^L (\fs) X_\nr (G  )} \rtimes \mf R_\fs^\sharp . 
\qedhere
\]
\end{proof}

We would like to formulate a version Lemma \ref{lem:3.11} with idempotents
in $\cH (G)$ and $\cH (M)$. Consider the types $(K_G,\lambda_G \otimes \gamma)$ for
$\gamma \in X^G (\fs)$. 

\begin{lem}\label{lem:4.11}
Let $\gamma, \gamma' \in X^G (\fs)$.
\enuma{
\item The $K_G$-representations $\lambda_G \otimes \gamma$ and $\lambda_G \otimes \gamma'$ 
are equivalent if and only if $\gamma^{-1} \gamma' \in X^L (\fs,\lambda)$.
\item For any $a,a' \in M$ the idempotents $a e_{\lambda_G \otimes \gamma} 
a^{-1}$ and $a' e_{\lambda_G \otimes \gamma'} (a')^{-1}$ are orthogonal if
$\gamma^{-1} \gamma' \in X^G (\fs) \setminus X^L (\fs)$.
}
\end{lem}
\begin{proof}
(a) Suppose first that $\gamma^{-1} \gamma' \in X^G (\fs) \setminus X^L (\fs)$. 
Then $(K,\lambda)$ and $(K,\lambda \otimes \gamma)$ are types for different Bernstein 
components of $M$, so $\lambda$ and $\lambda \otimes \gamma$ are not equivalent. 
As both $\lambda_G$ and $\gamma$ are trivial on $K_G \cap U$ and on 
$K_G \cap \overline{U}$, this implies that $\lambda_G$ and $\lambda_G \otimes \gamma$ 
are not equivalent either.

Now suppose that $\gamma \in X^L (\fs)$. By the definition of $\Stab (\fs,\lambda)$, the
$K$-representations $\lambda$ and $\lambda \otimes \gamma$ are equivalent if and only
if $\gamma \in \Stab (\fs,\lambda)$. By the same argument as above, this statement can be
lifted to $\lambda_G$ and $\lambda_G \otimes \gamma$.\\
(b) Consider the idempotents $a e_{\lambda \otimes \gamma} 
a^{-1}$ and $a' e_{\lambda \otimes \gamma'} (a')^{-1}$ in $\cH (M)$.
They belong to the subalgebras $\cH (M)^{\fs_M \otimes \gamma}$ and
$\cH (M)^{\fs_M \otimes \gamma'}$, respectively. Since $X^L (\fs) = X^M (\fs)$ and
$\gamma X^L (\fs) \neq \gamma' X^L (\fs)$, these are two orthogonal ideals of $\cH (L)$.
In particular the two above idempotents are orthogonal. 

Let $\langle K_G \cap U \rangle$ denote the idempotent, in the multiplier algebra of
$\cH (G)$, which corresponds to averaging over the group $K_G \cap U$. Then 
\begin{align*}
a e_{\lambda_G \otimes \gamma} a^{-1} & = a \langle K_G \cap U \rangle 
\langle K_G \cap \overline{U} \rangle e_{\lambda \otimes \gamma} a^{-1} \\
& = \langle a (K_G \cap U) a^{-1} \rangle \langle a( K_G \cap \overline{U}) a^{-1} \rangle 
a e_{\lambda \otimes \gamma} a^{-1} 
\end{align*}
Similarly
\[
a' e_{\lambda_G \otimes \gamma} (a')^{-1} = 
a e_{\lambda \otimes \gamma'} a^{-1}
\langle a (K_G \cap U) a^{-1} \rangle \langle a( K_G \cap \overline{U}) a^{-1} \rangle 
\]
Now we see from the earlier orthogonality result that 
\[
a e_{\lambda_G \otimes \gamma} a^{-1} a' e_{\lambda_G \otimes \gamma'} (a')^{-1} 
= 0 . \qedhere
\]
\end{proof}

Generalizing \eqref{eq:3.73} we define \label{i:72}
\[
X^G (\fs / \lambda) = X^G (\fs) \big/ \big( X^L (\fs) \cap \Stab (\fs,\lambda) \big) .
\]
By Lemma \ref{lem:4.11} the elements \label{i:05}\label{i:07}\label{i:08}
\begin{equation}\label{eq:3.71}
\begin{split}
& e_{\mu_G} := \sum\nolimits_{\gamma \in X^G (\fs / \lambda)} 
e_{\lambda_G \otimes \gamma} \in \cH (G) , \\
& e^\sharp_{\lambda_G} := \sum\nolimits_{a \in [L / H_\lambda]} 
a e_{\mu_G} a^{-1} \in \cH (G) , \\
& e^\sharp_{\lambda} := \sum\nolimits_{a \in [L / H_\lambda]} 
\sum\nolimits_{\gamma \in X^G (\fs / \lambda)} 
a e_{\lambda \otimes \gamma} a^{-1} \in \cH (M) 
\end{split}
\end{equation}
are idempotent. We will show that the latter two idempotents see precisely the categories 
of representations of $G$ and $G^\sharp$ (resp. $M$ and $M^\sharp$) associated to $\fs$. 
However, in general they do not come from a type, for the elements $a \in [L / H_\lambda]$ 
and $c_\gamma \in L$ need not lie in a compact subgroup of $G$.

\begin{lem}\label{lem:3.31}
Let $(\pi,V_\pi) \in \Irr^\fs (G)$. Then $e^\sharp_{\lambda_G} V_\pi$ intersects every
$G^\sharp$-isotypical component of $V_\pi$ nontrivially.
\end{lem}
\begin{proof}
The twisted group algebra $\C [X^G (\pi), \kappa_\pi]$ acts on $V_\pi$ via intertwining
operators. In view of \eqref{eq:2.2}, we have to show that $e^\sharp_{\lambda_G} V_\pi$
intersects the $\rho$-isotypical part of $V_\pi$ nontrivially, for every $\rho \in
\Irr \big( \C [X^G (\pi), \kappa_\pi] \big)$.

Choose $\chi \in X_\nr (L)$ such that $\pi$ is a subquotient of $I_P^G (\omega
\otimes \chi)$. Then 
\[
X^G (\pi) \cap X^L (\fs,\lambda) \subset X^L (\omega \otimes \chi) .
\]
As observed in the proof of Lemma \ref{lem:3.30}.d, every irreducible representation of\\ 
$\C [X^G (\pi) \cap X^L (\fs,\lambda), \kappa_{\omega \otimes \chi}]$ appears in
$e^\sharp_{\lambda_G} V_\pi$. The idempotents 
\begin{equation}\label{eq:3.85}
\{a e_{\lambda_G \otimes \gamma} a^{-1} : a \in [L / H_\lambda], \gamma \in X^G (\fs) \}
\end{equation}
are invariant under $X^L (\fs,\lambda)$ because $e_{\lambda_G}$ is, and they are mutually
orthogonal by Lemma \ref{lem:4.11}. As these idempotents sum to $e^\sharp_{\lambda_G}$, it 
follows that every every irreducible representation of 
$\C [X^G (\pi) \cap X^L (\fs,\lambda), \kappa_{\omega \otimes \chi}]$ already 
appears in one subspace $a e_{\lambda_G \otimes \gamma} a^{-1} V_\pi$. The quotient
group $X^G (\pi) / X^G (\pi) \cap X^L (\fs,\lambda)$ permutes the set of idempotents
\eqref{eq:3.85} faithfully. With Frobenius reciprocity we conclude that every irreducible
representation of $\C [X^G (\pi), \kappa_\pi]$ appears in $e^\sharp_{\lambda_G} V_\pi$.
\end{proof}

\begin{lem}\label{lem:3.H}
\enuma{
\item The algebras $e^\sharp_{\lambda_G} \cH (G)^{X^G (\fs)} e^\sharp_{\lambda_G} =
( e^\sharp_{\lambda_G} \cH (G) e^\sharp_{\lambda_G} )^{X^G (\fs)}$ and \\
$\big( \cH (G)^\fs \big)^{X^G (\fs)}$ are Morita equivalent. 
\item $\big( \cH (G)^\fs \big)^{X^G (\fs) X_\nr (G  )} \sim_M
( e^\sharp_{\lambda_G} \cH (G) e^\sharp_{\lambda_G} )^{X^G (\fs) X_\nr (G  )}$.
}
\end{lem}
\begin{proof}
(a) Because all the types $(K_G,\lambda_G \otimes \gamma)$
are for the same Bernstein component $\fs$, the idempotent $e^\sharp_{\lambda_G}$ sees
precisely the category of representations $\Rep^\fs (G)$.
Therefore the bimodules $e^\sharp_{\lambda_G} \cH (G)$ and
$\cH (G) e^\sharp_{\lambda_G}$ implement a Morita equivalence 
\begin{equation}\label{eq:3.74}
\cH (G)^\fs \sim_M e^\sharp_{\lambda_G} \cH (G) e^\sharp_{\lambda_G} .
\end{equation}
The same reasoning as in parts (b) and (c) of Proposition \ref{prop:3.2} establishes
Morita equivalences
\[
\big( \cH (G)^\fs \big)^{X^G (\fs)} \sim_M \cH (G)^\fs \rtimes X^G (\fs) \sim_M
\big( e^\sharp_{\lambda_G} \cH (G) e^\sharp_{\lambda_G} \big) \rtimes X^G (\fs) .
\]
To get from the right hand side to $\big( e^\sharp_{\lambda_G} \cH (G) e^\sharp_{\lambda_G} 
\big)^{X^G (\fs)}$ we follow the proof of Proposition \ref{prop:3.2}.d. This is
justified by Lemma \ref{lem:3.31}. \\
(b) The above argument also shows that the Morita 
equivalence of part (a) is implemented by the bimodules
\begin{equation}\label{eq:3.66}
e^\sharp_{\lambda_G} \cH (G)^{X^G (\fs)} \quad \text{and} \quad
\cH (G)^{X^G (\fs)} e^\sharp_{\lambda_G}.
\end{equation}
These bimodules are endowed with actions of $X_\nr (G)$. Taking invariants
under these group actions amounts to considering only functions supported on $G^1$.
We note that $e^\sharp_{\lambda_G}$ is supported on $G^1$ and that this is a normal subgroup
of $G$. Therefore the equations that make \eqref{eq:3.66} Morita bimodules restrict
to analogous equations for functions supported on $G^1$, which provides the desired
Morita equivalence.
\end{proof}

Recall the idempotents $e_\mu, e^\fs_M$ from \eqref{eq:3.61} and 
$e^\sharp_{\lambda_G}, e^\sharp_\lambda, e_{\mu_G}$ from \eqref{eq:3.71}.

\begin{prop}\label{prop:3.I}
There are algebra isomorphisms
\enuma{
\item $e^\sharp_{\lambda_G} \cH (G) e^\sharp_{\lambda_G} \cong 
e^\sharp_{\lambda} \cH (M \rtimes \mf R_\fs^\sharp) e^\sharp_{\lambda}$,
\item $(e^\sharp_{\lambda_G} \cH (G) e^\sharp_{\lambda_G})^{X^G (\fs)} \cong 
(e^\sharp_{\lambda} \cH (M \rtimes \mf R_\fs^\sharp) e^\sharp_{\lambda})^{X^G (\fs)} \cong
(e^\fs_M \cH (M) e^\fs_M )^{X^L (\fs)} \rtimes \mf R_\fs^\sharp$,
\item between the three algebras of $X_\nr (G )$-invariants in (b).
}
Moreover the isomorphisms in (b) and (c) can be chosen such that, for every $a_1,a_2 \in
[L / H_\lambda]$, they restrict to linear bijections
\begin{align*}
& (a_1 e_{\mu_G} \cH (G) e_{\mu_G} a_2^{-1})^{X^G (\fs)} \longleftrightarrow
(a_1 e_\mu \cH (M) e_\mu a_2^{-1} )^{X^L (\fs)} \rtimes \mf R_\fs^\sharp , \\
& (a_1 e_{\mu_G} \cH (G) e_{\mu_G} a_2^{-1})^{X^G (\fs) X_\nr (G )} 
\longleftrightarrow
(a_1 e_\mu \cH (M) e_\mu a_2^{-1} )^{X^L (\fs) X_\nr (G )} \rtimes \mf R_\fs^\sharp .
\end{align*}
\end{prop}
\begin{proof}
For any $\gamma \in X^L (\fs)$ and $w \in \mf R_\fs^\sharp$ there exists a
$\gamma' \in X^G (\fs)$ such that $w(\lambda \otimes \gamma) \cong \lambda \otimes
\gamma'$ as representations of $K$. Hence
\begin{align}
\nonumber & e^\sharp_{\lambda} \cH (M \rtimes \mf R_\fs^\sharp) e^\sharp_{\lambda} \cap 
\cH (M)^{w (\fs_M)} = w e^\fs_M w^{-1} \cH (M) w e^\fs_M w^{-1} , \\
\label{eq:3.52} & e^\sharp_{\lambda} \cH (M \rtimes \mf R_\fs^\sharp) e^\sharp_{\lambda} = 
\Big( \bigoplus\nolimits_{w \in \mf R_\fs^\sharp} w e^\fs_M \cH (M) e^\fs_M w^{-1} \Big) 
\rtimes \mf R_\fs^\sharp .
\end{align}
We note that the right hand side of \eqref{eq:3.52} is isomorphic to
\begin{equation}\label{eq:3.51}
e^\fs_M \cH (M) e^\fs_M \otimes \End_\C (\C \mf R_\fs^\sharp)
\end{equation}
The equality \eqref{eq:3.52} also shows that
\[
(e^\sharp_{\lambda} \cH (M \rtimes \mf R_\fs^\sharp) e^\sharp_{\lambda} )^{X^G (\fs)} = 
\Big( \bigoplus\nolimits_{w \in \mf R_\fs^\sharp} ( w e^\fs_M \cH (M) e^\fs_M w^{-1} 
)^{X^L (\fs)} \rtimes \mf R_\fs^\sharp \Big)^{X^G (\fs) / X^L (\fs)} .
\]
We can apply the argument from the proof of Lemma \ref{lem:3.3} to the right hand 
side, which gives a canonical isomorphism
\begin{equation}\label{eq:3.59}
(e^\sharp_{\lambda} \cH (M \rtimes \mf R_\fs^\sharp) e^\sharp_{\lambda})^{X^G (\fs)} \cong
(e^\fs_M \cH (M) e^\fs_M )^{X^L (\fs)} \rtimes \mf R_\fs^\sharp. 
\end{equation}
Notice that for $a \in L / H_\lambda$ the idempotents $a e_\mu a^{-1}$ and $a \sum_{\gamma 
\in X^G (\fs / \lambda)} e_{\lambda \otimes \gamma} a^{-1}$ are invariant under $\Stab (\fs, 
P \cap M)$ and $X^G (\fs)$, respectively. Hence we can write 
\begin{equation}\label{eq:3.D}
\begin{aligned}
( e^\sharp_{\lambda_G} \cH (G) e^\sharp_{\lambda_G} )^{X^G (\fs)} =
\bigoplus\nolimits_{a_1, a_2 \in [L / H_\lambda]} 
(a_1 e_{\mu_G} \cH (G) e_{\mu_G} a_2^{-1})^{X^G (\fs)} , \\
(e^\fs_M \cH (M) e^\fs_M )^{X^L (\fs)} \rtimes \mf R_\fs^\sharp = 
\bigoplus\nolimits_{a_1, a_2 \in [L / H_\lambda]} 
(a_1 e^\mu \cH (M) e_\mu a_2^{-1} )^{X^L (\fs) X_\nr (G )} \rtimes \mf R_\fs^\sharp .
\end{aligned}
\end{equation}
It is clear from the proof of Lemma \ref{lem:3.3} that the isomorphism \eqref{eq:3.59}
respects these decompositions. Moreover \eqref{eq:3.59} is equivariant with respect to the 
actions of $X_\nr (G  )$, so it restricts to
\[
(e^\sharp_{\lambda} \cH (M \rtimes \mf R_\fs^\sharp) 
e^\sharp_{\lambda})^{X^G (\fs) X_\nr (G  )} \cong (e^\fs_M \cH (M) 
e^\fs_M )^{X^L (\fs) X_\nr (G  } \rtimes \mf R_\fs^\sharp.  
\]
We have proved the second isomorphism of part (b) and of part (c). 

For every $\gamma \in X^G (\fs / \lambda)$ and $a \in [L / H_\lambda]$ one has
\[
a e_{\lambda_G \otimes \gamma} a^{-1} \cH (G) a e_{\lambda_G \otimes \gamma} a^{-1} =
a \alpha_\gamma^{-1} ( e_{\lambda_G \otimes \gamma} \cH (G) e_{\lambda_G \otimes \gamma})
a^{-1} \cong e_{\lambda_G \otimes \gamma} \cH (G) e_{\lambda_G \otimes \gamma} .
\]
By Lemma \ref{lem:4.11} these are mutually orthogonal subalgebras of 
$e^\sharp_{\lambda_G} \cH (G) e^\sharp_{\lambda_G}$. The inclusion 
\[
a e_{\lambda_G \otimes \gamma} a^{-1} \cH (G) a e_{\lambda_G \otimes \gamma} a^{-1} \to
e^\sharp_{\lambda_G} \cH (G) e^\sharp_{\lambda_G}
\]
is a Morita equivalence and for all $V \in \Rep (G)^\fs$:
\[
e^\sharp_{\lambda_G} V = \bigoplus_{a \in [L / H_\lambda]} 
\bigoplus_{\gamma \in X^G (\fs / \lambda)} a e_{\lambda_G \otimes \gamma} a^{-1} V .
\]
It follows that the $a e_{\lambda_G \otimes \gamma} a^{-1}$ form the idempotent matrix
units in some subalgebra $M_n (\C) \subset e^\sharp_{\lambda_G} \cH (G) e^\sharp_{\lambda_G}$, 
and that 
\[
e^\sharp_{\lambda_G} \cH (G) e^\sharp_{\lambda_G} \cong 
e_{\lambda_G} \cH (G) e_{\lambda_G} \otimes M_n (\C)
\]
where $n = |L / H_\lambda| \, |X^G (\fs / \lambda)|$. 

The same argument shows that
\[
e^\fs_M \cH (M) e^\fs_M \cong e_\lambda \cH (M) e_\lambda \otimes M_{n'} (\C) , 
\]
where $n' = |L / H_\lambda| \, |X^L (\fs / \lambda)|$. Since $(K_G,\lambda_G)$ is a cover
of $(K,\lambda)$, 
\[
e_\lambda \cH (M) e_\lambda \cong e_{\lambda_G} \cH (G) e_{\lambda_G} . 
\]
By Lemma \ref{lem:4.11}
\[
n' \, |\mf R_\fs^\sharp| = |L / H_\lambda| \, |X^L (\fs / \lambda)| \, |\mf R_\fs^\sharp| 
= |L / H_\lambda| \, |X^G (\fs / \lambda)| = n . 
\]
With \eqref{eq:3.51} we deduce that 
\begin{equation}\label{eq:3.56}
e^\sharp_{\lambda} \cH (M \rtimes \mf R_\fs^\sharp) e^\sharp_{\lambda} \cong
e_{\lambda_G} \cH (G) e_{\lambda_G} \otimes M_n (\C) \cong
e^\sharp_{\lambda_G} \cH (G) e^\sharp_{\lambda_G} ,
\end{equation}
proving part (a). It entails from \eqref{eq:3.60} that 
\begin{equation}\label{eq:3.53}
e^\sharp_{\lambda_G} C_c^\infty (G / U) e^\sharp_{\lambda} \quad \text{and} \quad
e^\sharp_{\lambda} C_c^\infty (U \backslash G ) e^\sharp_{\lambda_G}
\end{equation}
are bimodules for a Morita equivalence
\begin{equation}\label{eq:3.54}
e^\sharp_{\lambda} \cH (M \rtimes \mf R_\fs^\sharp) e^\sharp_{\lambda} \sim_M
e^\sharp_{\lambda_G} \cH (G) e^\sharp_{\lambda_G} .
\end{equation}
But by \eqref{eq:3.56} these algebras are isomorphic, so the bimodules are free
of rank 1 over both algebras.

Similarly, it follows from \eqref{eq:3.N} that 
\begin{equation}\label{eq:3.57}
( C_c^\infty (G / U)^\fs )^{X^G (\fs)} \quad \text{and} \quad
( C_c^\infty (U \backslash G )^\fs )^{X^G (\fs)}  
\end{equation}
are bimodules for a Morita equivalence between $(\cH (G)^\fs )^{X^G (\fs)}$ and
$(\cH (M \rtimes \mf R_\fs^\sharp)^\fs )^{X^G (\fs)}$.

By Lemma \ref{lem:3.H}, Proposition \ref{prop:3.2} and \eqref{eq:3.59} there is a 
chain of Morita equivalences
\begin{equation}\label{eq:3.65}
\begin{split}
(e^\sharp_{\lambda_G} \cH (G) e^\sharp_{\lambda_G})^{X^G (\fs)}
\sim_M (\cH (G)^\fs )^{X^G (\fs)} \sim_M
(\cH (M \rtimes \mf R_\fs^\sharp)^\fs )^{X^G (\fs)} \\
\sim_M (e^\sharp_{\lambda} \cH (M \rtimes \mf R_\fs^\sharp) e^\sharp_{\lambda})^{X^G (\fs)} .
\end{split}
\end{equation}
The respective Morita bimodules are given by \eqref{eq:3.66}, \eqref{eq:3.57} and
\eqref{eq:3.25}. In relation to \eqref{eq:3.59} we can rewrite \eqref{eq:3.25} as
\begin{equation}\label{eq:3.69}
e^\sharp_{\lambda} (\cH (M \rtimes \mf R_\fs^\sharp)^\fs )^{X^G (\fs)} \quad \text{and}
\quad (\cH (M \rtimes \mf R_\fs^\sharp)^\fs )^{X^G (\fs)} e^\sharp_{\lambda} .
\end{equation}
It follows that Morita bimodules for the composition of \eqref{eq:3.65} are
\begin{equation}\label{eq:3.67}
(e^\sharp_{\lambda_G} C_c^\infty (G / U) e^\sharp_{\lambda} )^{X^G (\fs)} \quad \text{and} \quad
(e^\sharp_{\lambda} C_c^\infty (U \backslash G ) e^\sharp_{\lambda_G} )^{X^G (\fs)} .
\end{equation}
As the modules \eqref{eq:3.57} are free of rank 1 over both the algebras 
\eqref{eq:3.54}, and the actions of $X^G (\fs)$ on \eqref{eq:3.67} and the involved
algebras come from the action on functions $G \to \C$, the modules \eqref{eq:3.67}
are again free of rank 1 over 
\begin{equation}\label{eq:3.68}
(e^\sharp_{\lambda_G} \cH (G) e^\sharp_{\lambda_G})^{X^G (\fs)} \quad \text{and} \quad
(e^\sharp_{\lambda} \cH (M \rtimes \mf R_\fs^\sharp) e^\sharp_{\lambda})^{X^G (\fs)} .
\end{equation}
Therefore these two algebras are isomorphic.
Since the idempotents $a e_{\mu_G} a^{-1}$ and $a \sum_{\gamma \in X^G (\fs / \lambda)} 
e_{\lambda \otimes \gamma} a^{-1}$ are $X^G (\fs)$-invariant, $e^\sharp_{\lambda_G} \cH (G) 
e^\sharp_{\lambda_G})^{X^G (\fs)}$ and the bimodules \eqref{eq:3.67} can be decomposed in
the same way as \eqref{eq:3.D}. It follows that the isomorphism between the algebras in
\eqref{eq:3.68}, as just constructed from \eqref{eq:3.67}, respects the decompositions
indexed by $a_1,a_2 \in [L / H_\lambda]$. This settles part (b). 

It remains to prove the first isomorphism of part (c), but here we encounter the
problem that the isomorphism between the algebras \eqref{eq:3.68} is not explicit. 
In particular we do not know for 
sure that it is equivariant with respect to $X_\nr (G  )$. Nevertheless,
we claim that the chain of Morita equivalences \eqref{eq:3.65} remains valid upon
taking $X_\nr (G  )$-invariants. For the first equivalence that is Lemma
\ref{lem:3.H}.b and for the second equivalence it was checked in \eqref{eq:3.M}.
For the third equivalence we can use the same argument as for the first, the equations
making \eqref{eq:3.69} Morita bimodules can be restricted to functions $G \to \C$
supported on $G^1 G^\sharp$. Composing these three steps, we obtain 
\begin{equation}\label{eq:3.72}
(e^\sharp_{\lambda_G} \cH (G) e^\sharp_{\lambda_G})^{X^G (\fs) X_\nr (G  )} \sim_M
(e^\sharp_{\lambda} \cH (M \rtimes \mf R_\fs^\sharp) 
e^\sharp_{\lambda})^{X^G (\fs) X_\nr (G  )} .
\end{equation}
with Morita bimodules 
\begin{equation}\label{eq:3.B}
(e^\sharp_{\lambda_G} C_c^\infty (G / U) e^\sharp_{\lambda} )^{X^G (\fs) X_\nr (G  )} 
\quad \text{and} \quad (e^\sharp_{\lambda} C_c^\infty (U \backslash G ) 
e^\sharp_{\lambda_G} )^{X^G (\fs) X_\nr (G  )} .
\end{equation}
Since the modules \eqref{eq:3.67} are free of rank one over the algebras 
\eqref{eq:3.68}, the modules \eqref{eq:3.B} are free of rank one over both the 
algebras in \eqref{eq:3.72}. Therefore these two algebras are isomorphic.
The isomorphism respects the decompositions indexed by $a_1,a_2 \in [L / H_\lambda]$
for the same reasons as in part (b).
\end{proof}

We normalize the Haar measures on $G, G^\sharp$ and $G^\sharp Z(G)$ such that $K_G$ and
$K_G \cap G^\sharp$ and $K_G \cap G^\sharp Z(G)$ have the same volume. 
Consider the idempotent $e^\sharp_{\lambda_G} \in \cH (G)$ as a function $G \to \C$ and let 
$e^\sharp_{\lambda_{G^\sharp}}$ (respectively $e^\sharp_{\lambda_{G^\sharp Z(G)}}$) be its 
restriction to $G^\sharp$ (respectively $G^\sharp Z(G)$). It turns out that \label{i:09}
\begin{equation}\label{eq:3.98}
e^\sharp_{\lambda_{G^\sharp Z(G)}} \in \cH (G^\sharp Z(G)) \; \text{ and } \;
e^\sharp_{\lambda_{G^\sharp}} \in \cH (G^\sharp) .
\end{equation}
are idempotents. We describe the associated subalgebras of $\cH (G^\sharp Z(G))$ 
(respectively $\cH (G^\sharp)$ in two separate but analogous theorems.

\begin{thm}\label{thm:3.17}
The element $e^\sharp_{\lambda_{G^\sharp Z(G)}} \in \cH (G^\sharp Z(G))$ is idempotent and
\[
e^\sharp_{\lambda_{G^\sharp Z(G)}} \cH (G^\sharp Z(G)) e^\sharp_{\lambda_{G^\sharp Z(G)}} 
\cong e^\sharp_{\lambda_G} \cH (G)^{X^G (\fs)} e^\sharp_{\lambda_G} 
\cong (e^\fs_M \cH (M) e^\fs_M )^{X^L (\fs)} \rtimes \mf R_\fs^\sharp .
\]
These algebras are Morita equivalent with $\cH (G^\sharp Z(G))^\fs$ and with 
$\big( \cH (G)^\fs \big)^{X^G (\fs)}$.
\end{thm}
\begin{proof}
Consider the $l$-th congruence subgroup $C_l \subset GL_m (\mf o_D)$, as in the proof 
of Lemma \ref{lem:3.1}. We choose the level $l$ so high that all representations in 
$\Rep^\fs (G)$ have nonzero $C_l$-fixed vectors and that $e^\sharp_{\lambda_G}$ is 
$C_l$-biinvariant. Put $C'_l = C_l \cap G^\sharp Z(G)$. The proof of Lemma \ref{lem:3.1} 
shows that the algebra isomorphism
\begin{equation}\label{eq:4.37}
\cH (G^\sharp Z(G),C'_l)^\fs \to \big( \cH (G,C_l)^\fs \big)^{X^G (\fs)}
\end{equation}
coming from \eqref{eq:3.70} maps $e^\sharp_{\lambda_{G^\sharp Z(G)}}$ to $e^\sharp_{\lambda_G}$. 
As $e^\sharp_{\lambda_G}$ is idempotent, so is $e^\sharp_{\lambda_{G^\sharp Z(G)}}$. 
It follows that \eqref{eq:4.37} restricts to an isomorphism
\begin{multline}\label{eq:4.39}
e^\sharp_{\lambda_{G^\sharp Z(G)}} \cH (G^\sharp Z(G)) e^\sharp_{\lambda_{G^\sharp Z(G)}} =
e^\sharp_{\lambda_{G^\sharp Z(G)}} \cH (G^\sharp Z(G),C'_l)^\fs 
e^\sharp_{\lambda_{G^\sharp Z(G)}} \to \\
e^\sharp_{\lambda_G} \big( \cH (G,C_l)^\fs \big)^{X^G (\fs)} e^\sharp_{\lambda_G} = 
e^\sharp_{\lambda_G} \cH (G)^{X^G (\fs)} e^\sharp_{\lambda_G} . 
\end{multline}
The second isomorphism of the lemma is Proposition \ref{prop:3.I}.b. 
By Lemma \ref{lem:3.H} these algebras are Morita equivalent with 
$\big( \cH (G)^\fs \big)^{X^G (\fs)}$, and by 
Lemma \ref{lem:3.1} also with $\cH (G^\sharp Z(G))^\fs$.
\end{proof}

\begin{thm}\label{thm:3.18}
The element $e^\sharp_{\lambda_{G^\sharp}} \in \cH (G^\sharp)$ is idempotent and
\[
e^\sharp_{\lambda_{G^\sharp}} \cH (G^\sharp) e^\sharp_{\lambda_{G^\sharp}} \cong
e^\sharp_{\lambda_G} \cH (G)^{X^G (\fs) X_\nr (G  )} e^\sharp_{\lambda_G} \cong
(e^\fs_M \cH (M) e^\fs_M )^{X^L (\fs) X_\nr (G  )} \rtimes \mf R_\fs^\sharp .
\]
These algebras are Morita equivalent with $\cH (G^\sharp)^\fs$ and with 
$\big( \cH (G)^\fs \big)^{X^G (\fs) X_\nr (G  )}$.
\end{thm}
\begin{proof}
Recall that $(K_G,\lambda_G)$ is a type for the single Bernstein component $\fs$.
The representations in $\Rep^\fs (G)$ contain only one character of $Z(G) \cap G^1$,
so we must have 
\[
Z(G) \cap K_G = Z(G) \cap G^1 = \mf o_F^\times \cdot 1_m . 
\]
Because $\lambda_G$ is irreducible as a representation of $K_G$,
$Z(G) \cap K_G$ acts on it by a character, say $\zeta_\lambda$. 

Endow $Z(G)$ with the Haar measure for which $Z(G) \cap K_G$ gets volume 
$| Z(G) \cap G^\sharp |$. There is an equality
\[
e^\sharp_{\lambda_{G^\sharp Z(G)}} = e^\sharp_{\lambda_{G^\sharp}} e_{\zeta_\lambda} 
\]
of distributions on $G^\sharp Z(G)$, where $e_{\zeta_\lambda}$ denotes the idempotent 
associated to $(Z(G) \cap G^1, \zeta_\lambda)$. Then
\begin{equation}\label{eq:4.51}
e^\sharp_{\lambda_{G^\sharp}} \cH (G^\sharp) e^\sharp_{\lambda_{G^\sharp}} \to
e^\sharp_{\lambda_{G^\sharp Z(G)}} \cH (G^\sharp Z(G)) e^\sharp_{\lambda_{G^\sharp Z(G)}}  :
f \mapsto f e_{\zeta_\lambda}
\end{equation}
is an injective algebra homomorphism with image
\[
e^\sharp_{\lambda_{G^\sharp Z(G)}} \cH (G^\sharp (Z(G) \cap G^1)) 
e^\sharp_{\lambda_{G^\sharp Z(G)}} .
\]
This is precisely the subalgebra of $e^\sharp_{\lambda_{G^\sharp Z(G)}} 
\cH (G^\sharp Z(G)) e^\sharp_{\lambda_{G^\sharp Z(G)}}$ which is invariant under
$X_\nr (G^\sharp Z(G)  )$. Under the isomorphism \eqref{eq:4.39} it 
corresponds to
\[
e^\sharp_{\lambda_G} \cH (G)^{X^G (\fs) X_\nr (G  )} e^\sharp_{\lambda_G} .
\]
That algebra is isomorphic to
\[
(e^\fs_M \cH (M) e^\fs_M )^{X^L (\fs) X_\nr (G  )} \rtimes \mf R_\fs^\sharp
\]
by Proposition \ref{prop:3.I}.c and Morita equivalent to 
$\big( \cH (G)^\fs \big)^{X^G (\fs) X_\nr (G  )}$ by Lemma \ref{lem:3.H}.b.
In Lemma \ref{lem:3.11} we already observed that this last algebra is Morita
equivalent with $\cH (G^\sharp)^\fs$.
\end{proof}

\section{The structure of the Hecke algebras}
\label{sec:Hecke}

\subsection{Hecke algebras for general linear groups} \
\label{subsec:Secherre}

Consider the inertial equivalence class $\fs = [L,\omega]_G$.
Via the map $\chi \mapsto \omega \otimes \chi$ we identify $T_\fs$ with the complex torus
$X_\nr (L) / X_\nr (L,\omega)$. This gives us the lattices $X^* (T_\fs)$ and $X_* (T_\fs)$
of algebraic characters and cocharacters, respectively. We emphasize that this depends on
the choice of the basepoint $\omega$ of $T_\fs$. Under the Conditions \ref{cond}, any other
basepoint is the form $\omega' = \omega \otimes \chi'$ where $\chi' \in X_\nr (L)^{W_\fs}$.
There is a natural isomorphism $x \mapsto x'$ from $X^* (T_\fs)$ with respect to $\omega$
to $X^* (T_\fs)$ with respect to $\omega'$. As functions on $T_\fs$, it works out to
\begin{equation}\label{eq:3.39}
x' (\omega \otimes \chi) = x (\omega \otimes \chi) x (\omega \otimes \chi')^{-1} .
\end{equation}
The inertial equivalence class $\fs$ comes not only with the torus $T_\fs$ and the group
$W_\fs$, but also with a root system $R_\fs \subset X^* (T_\fs)$, whose Weyl group is $W_\fs$.
From \eqref{eq:3.39} we see that a character $x$ is independent of the choice of a 
basepoint of $T_\fs$ if it is invariant under $X_\nr (L)^{W_\fs}$, that is, if $x$ lies in
the lattice $\Z R_\fs$ spanned by $R_\fs$.

Let $\cH (X^* (T_\fs) \rtimes W_\fs,q_\fs)$ denote the affine Hecke algebra
associated to the root datum $(X^* (T_\fs), X_* (T_\fs), R_\fs, R^\vee_\fs)$
and the parameter function $q_\fs$ as in \cite{Sec3}. It has a standard basis 
$\{[x] : x \in X^* (T_\fs) \rtimes W_\fs \}$, with multiplication rules described first 
by Iwahori and Matsumoto \cite{IwMa}. \label{i:24}

We remark that here $q_\fs$ is not a just one real number, but a collection of parameters 
$q_{\fs,i} > 0$, one for each factor $M_i$ of $M$, or
equivalently one for each irreducible component of the root system $R_\fs$. 
The parameter $q_\fs$ has a natural extension to a map 
\[
q_\fs : X^* (T_\fs) \rtimes W_\fs \to \R_{>0} ,
\]
see \cite[\S 1]{Lus1}. On the part of $X^* (T_\fs)$ that is positive with respect to 
$P \cap M$ it can be defined as follows. Since $T_\fs$ 
is a quotient of $X_\nr (L)$, $X^* (T_\fs)$ is naturally isomorphic to a subgroup of 
$L / L^1$. In this way $q_\fs$ corresponds to $\delta_u^{-1}$, the inverse of the modular 
character for the action of $L$ on the unipotent radical of $P \cap M$.

Let us recall the Bernstein presentation of an affine Hecke algebra \cite[\S 3]{Lus1}.
For $x \in X^* (T_\fs)$ positive with respect to $P \cap M$ we write \label{i:56} 
\begin{equation}\label{eq:4.2}
\theta_x := q_\fs (x)^{-1/2} [x] = \delta_u^{1/2}(x) [x] .
\end{equation}
The map $x \mapsto \theta_x$ can be extended in a unique way to a group homomorphism
$X^* (T_\fs) \to \cH (X^* (T_\fs) \rtimes W_\fs,q_\fs)^\times$ \cite[2.6]{Lus1}, for
which we use the same notation. By \cite[Proposition 3.7]{Lus1}
\[
\{ \theta_x [w] : x \in X^* (T_\fs), w \in W_\fs \}
\] 
is a basis of $\cH (X^* (T_\fs) \rtimes W_\fs ,q_\fs)$. Furthermore the span of the
$\theta_x$ is a subalgebra \label{i: A} $\mc A$ isomorphic to $\C [X^* (T_\fs)] \cong \mc O (T_\fs)$ and
the span of the $[w]$ with $w \in W_\fs$ is the Iwahori--Hecke algebra 
$\cH (W_\fs,q_\fs)$. The multiplication map \label{i:25}
\begin{equation}\label{eq:3.40}
\mc A \otimes \cH (W_\fs,q_\fs) \to \cH (X^* (T_\fs) \rtimes W_\fs,q_\fs)
\end{equation}
is a linear bijection. The commutation relations between these two subalgebras are
known as the Bernstein--Lusztig--Zelevinsky relations. Let $\alpha \in R_\fs$ be
a simple root, with corresponding reflection $s \in W_\fs$. 
By \cite[Proposition 3.6]{Lus1}, for any $x \in X (T_\fs)$
\begin{equation}\label{eq:3.41}
\theta_x [s] - [s] \theta_{s(x)} = (q_\fs (s) - 1)(\theta_x - \theta_{s(x)})
(1 - \theta_{-\alpha})^{-1} \in \mc A.
\end{equation}
Since the elements $[s]$ generate $\cH (W_\fs,q_\fs)$, this determines the commutation
relations for $\mc A$ with all $[w] \; (w \in W_\fs)$. It follows from \eqref{eq:3.41} that
\begin{equation}\label{eq:4.30}
Z ( \cH (X^* (T_\fs) \rtimes W_\fs ,q_\fs) ) = \mc A^{W_\fs} .
\end{equation}
In view of \eqref{eq:3.39}, \eqref{eq:3.40} and \eqref{eq:3.41}, we can also regard
$\cH (X^* (T_\fs) \rtimes W_\fs ,q_\fs)$ as the algebra whose underlying vector space is
\begin{equation}\label{eq:3.42}
\mc O (T_\fs) \otimes \cH (W_\fs,q_\fs)
\end{equation}
and whose multiplication satisfies
\begin{equation}\label{eq:3.43}
f [s] - [s] (s \cdot f) = (q_\fs (s) - 1)(f - (s \cdot f)) (1 - \theta_{-\alpha})^{-1}
\qquad f \in \mc O (T_\fs) ,
\end{equation}
with respect to the canonical action of $W_\fs$ on $\mc O (T_\fs)$. The advantage is that,
written in this way, the multiplication does not depend on the choice of a basepoint
$\omega \in T_\fs$ used to define $X^* (T_\fs)$. We will denote this interpretation of 
$\cH (X^* (T_\fs) \rtimes W_\fs ,q_\fs)$ by $\cH (T_\fs,W_\fs,q_\fs)$. \label{i:23}

Let $\varpi_D$ be a uniformizer of $D$. Consider the group of diagonal matrices in $L$ 
all whose diagonal entries are powers of $\varpi_D$ and whose components in each $L_i$ are 
multiples of the identity. It can be identified with a sublattice of $X^* (X_\nr (L))$.
The lattice $X^* (T_\fs)$ can be represented in a unique way by such matrices, say by
the group $\widetilde{X^* (T_\fs)} \subset L$.

Recall that $(K,\lambda)$ is a type for $\fs_M$ and that $(K_L,\lambda_L)$ is a 
$\fs_L$-type. The next result is largely due to S\'echerre \cite{Sec3}.

\begin{thm}\label{thm:3.7}
For every $(w,\gamma) \in \Stab (\fs,P \cap M)$ there are isomorphisms
\begin{align*}
e_{w(\lambda_L ) \otimes \gamma} \cH (L) e_{w(\lambda_L) \otimes \gamma} & \cong 
\cH (L,w(\lambda_L) \otimes \gamma) \otimes \End_\C (V_{w(\lambda_L) \otimes \gamma}) \\
& \cong \mathcal O (T_\fs) \otimes \End_\C (V_{w(\lambda) \otimes \gamma}) ,\\
e_{w(\lambda) \otimes \gamma} \cH (M) e_{w(\lambda) \otimes \gamma} & \cong 
\cH (M,w(\lambda) \otimes \gamma) \otimes \End_\C (V_{w(\lambda) \otimes \gamma}) \\
& \cong \cH (T_\fs , W_\fs, q_\fs) \otimes \End_\C (V_{w(\lambda) \otimes \gamma}) .
\end{align*}
In both cases the first isomorphism is canonical, and the second depends only on the 
choice of the parabolic subgroup $P$. The support of these algebras is, respectively, 
$K_L \widetilde{X^* (T_\fs)} K_L$ and $K \widetilde{X^* (T_\fs)} W_\fs K$.
\end{thm}
\begin{proof}
Since all the types $(K,w(\lambda) \otimes \gamma)$ have the same properties, 
it suffices to treat the case $(w,\gamma) = (1,1)$. The first and third isomorphisms 
are instances of \eqref{eq:1.5}. The support of the algebras was determined in
\cite[\S 4]{Sec3}. S\'echerre also proved that the remaining 
isomorphisms exist, but some extra work is needed to make them canonical.

The $L$-representations $\omega \otimes \chi$ with $\chi \in X_\nr (L)$ paste to an
algebra homomorphism \label{i:14}
\begin{equation}\label{eq:3.26}
\mathcal F_L : e_{\lambda_L} \cH (L) e_{\lambda_L} \to \mathcal O (X_\nr (L))
\otimes \End_\C (e_{\lambda_L} V_\omega) ,
\end{equation}
which is injective because these are all irreducible representations in $\Rep^{\fs_L}(L)$.
By \cite[Th\'eor\`eme 4.6]{Sec3} $e_{\lambda_L} \cH (L) e_{\lambda_L}$ is isomorphic to 
$\mathcal O (T_\fs) \otimes \End_\C (V_{\lambda_L})$. Hence 
\begin{equation}\label{eq:3.28}
e_{\lambda_L} V_\omega \cong V_{\lambda_L} = V_\lambda
\end{equation} 
and \eqref{eq:3.26} restricts to a canonical isomorphism
\begin{equation}\label{eq:3.15}
\mathcal F_L : e_{\lambda_L} \cH (L) e_{\lambda_L} \to 
\mathcal O (T_\fs) \otimes \End_\C (V_{\lambda_L}) .
\end{equation}
Here $\mathcal O (T_\fs)$ is the centre of the right hand side, so it corresponds to 
$\cH (L,\lambda_L)$. Consider the isomorphism
\begin{equation}\label{eq:3.27}
e_{\lambda} \cH (M) e_{\lambda} \cong \cH (X^* (T_\fs) \rtimes W_\fs,q_\fs)
\otimes \End_\C (V_{\lambda}) .
\end{equation}
from \cite[Th\'eor\`eme 4.6]{Sec3}. It comes from $\cH (X^* (T_\fs) \rtimes W_\fs,q_\fs)
\cong \cH (M,\lambda)$. We define \label{i:15}
\begin{equation}\label{eq:4.76}
f_{x,\lambda} \in \cH (M,\lambda) 
\text{ as the image of } [x] \text{ under \eqref{eq:3.27}.}
\end{equation}
Because $(K,\lambda)$ is a cover of $(K_L,\lambda_L)$, we may use the results of 
\cite[\S 7]{BuKu3}. By \cite[Corollaries 7.2 and 7.11]{BuKu3} 
there exists a unique injective algebra homomorphism \label{i:57}
\begin{equation}\label{eq:3.21}
t_{P,\lambda} : \cH (L,\lambda_L) \to \cH (M,\lambda) 
\end{equation}
such that the diagram
\[
\xymatrix{
\Mod (\cH (M,\lambda)) \ar[r] \ar[d]_{t_{P,\lambda}^*} & \Rep (M) \ar[d]_{r^M_{P \cap M}} \\
\Mod (\cH (L,\lambda_L)) \ar[r] & \Rep (L)
}
\]
commutes. We note
that in \cite{BuKu3} unnormalized Jacquet restriction is used, whereas we prefer 
the normalized version. Therefore our $t_{P,\lambda}$ equals $t_{\delta_u^{1/2}}$
in the notation of \cite[\S 7]{BuKu3}, where $\delta_u$ denotes the modular character 
for the action of $L$ on the unipotent radical of $P \cap M$.

Consider the diagram 
\begin{equation}\label{eq:3.23}
\begin{array}{ccc}
\cH (M,\lambda) & \to & \cH (X^* (T_\fs) \rtimes W_\fs ,q_\fs) \cong 
\cH (T_\fs, W_\fs ,q_\fs) \\
 \uparrow {\scriptstyle t_{P,\lambda}} & & \uparrow \scriptstyle{i_{P,\lambda}} \\
\cH (L,\lambda_L) & \to & \mathcal O (T_\fs) \cong \C [X^* (T_\fs)] ,
\end{array}
\end{equation}
where the upper map is \eqref{eq:3.27} and the lower map comes from \eqref{eq:3.15}. 
The horizontal maps are isomorphisms and $t_{P,\lambda}$ is injective. We want to define 
the right vertical map $i_{P,\lambda}$ so that the diagram commutes.

The construction of the upper map in \cite[\S 4]{Sec3} shows that it is canonical on 
the subalgebra of $\cH (X^* (T_\fs) \rtimes W_\fs ,q_\fs)$ generated by the elements
$[s]$ with $s \in X^* (T_\fs) \rtimes W_\fs$ a simple affine reflection. This 
subalgebra has a basis $\{ [x] : x \in \Z R_\fs \rtimes W_\fs \}$, where $\Z R_\fs$ 
is the sublattice of $X^* (T_\fs)$ spanned by the root system $R_\fs$. In particular
the image $f_{x,\lambda} \in \cH (M,\lambda)$ of $[x]$ with $x \in \Z R_\fs \rtimes W_\fs$
is defined canonically.

By \cite[Th\'eor\`eme 4.6]{Sec3} the remaining freedom for \eqref{eq:3.27} boils down 
to, for each factor $M_i$ of $M$, the choice of a nonzero element in a 
one-dimensional vector space. This is equivalent to the freedom in the choice of
the basepoint $\omega$ of $T_\fs$.

Take a $x \in X^* (T_\fs)$ which is positive with respect to $P \cap M$, and let
$f_{x,\lambda_L}$ be the corresponding element of $\cH (L,\lambda_L)$. For such 
elements $t_{P,\lambda}$ is described explicitly by \cite[Theorem 7.2]{BuKu3}.
In our notation \label{i:16}
\[
t_{P,\lambda}(f_{x,\lambda_L}) = t_{\delta_u^{1/2}} (f_{x,\lambda_L}) =
\delta_u^{1/2}(x) f_{x,\lambda} .
\]
Suppose that furthermore $x \in \Z R_\fs$, considered as subset of $\C[X^* (T^\fs)]$.
Then its images three of the maps in \eqref{eq:3.23} are canonically determined.
In order that the diagram commutes, it is necessary that 
\begin{equation}\label{eq:3.24}
i_{P,\lambda} (x) = \theta_x,
\end{equation}
with $\theta_x$ as in \eqref{eq:4.2}.
The condition \eqref{eq:3.24} determines $i_{P,\lambda} (x)$ for all $x \in \Z R_\fs$. 
Now every way to extend $i_{P,\lambda}$ to the whole of $\C [X^* (T_\fs)]$ corresponds 
to precisely one choice of an isomorphism \eqref{eq:3.27}. Thus we can normalize
\eqref{eq:3.27} by requiring that \eqref{eq:3.24} holds for all $x \in X^* (T_\fs)$
which are positive with respect to $P \cap M$. 

In effect, we defined $i_{P,\lambda}$ to be the identity of $\mc O (T_\fs)$ with
respect to the isomorphisms 
\[
\mc A \cong \C[ X^* (T_\fs)] \cong \mc O (T_\fs).
\]
So we turned \eqref{eq:3.24} into an algebra homomorphism
\[
i_{P,\lambda} : \mc O (T_\fs) \to \cH (T_\fs,W_\fs,q_\fs) .
\]
A priori it depends on the choice of a basepoint of $T_\fs$, but since we use the
same basepoint on both sides and by \eqref{eq:3.39}, any other basepoint would
produce the same map $i_{P,\lambda}$. Thus \eqref{eq:3.27} becomes canonical if
we interpret the right hand side as 
$\cH (T_\fs, W_\fs,q_\fs) \otimes \End_\C (V_{\lambda})$.
\end{proof}

\subsection{Projective normalizers} \

We will subject the algebra $e^\fs_M \cH (M) e^\fs_M$ to a closer study, and 
describe its structure explicitly. At the same time we investigate how close
$e_\mu$ and $e^\fs_M$ from \eqref{eq:3.61} are to the idempotent of a type. A natural 
candidate for such a type would involve the projective normalizer of $(K,\lambda)$, 
but unfortunately it will turn out that this is in general not sufficiently sophisticated.

Recall the groups defined in \eqref{eq:3.86} and \eqref{eq:3.73} and consider
the vector spaces \label{i:61} \label{i:62}
\begin{equation}\label{eq:4.40}
\begin{aligned}
& V_{\mu^1} = V_{\mu_L^1} := 
\sum_{(w,\gamma) \in \Stab (\fs, P \cap M)^1} \! e_{w(\lambda_L) \otimes \gamma} 
V_\omega = \sum_{\gamma \in X^L (\fs)^1} \! e_{\lambda_L \otimes \gamma} V_\omega , \\
& V_{\mu_L} := V_\mu = 
\sum_{(w,\gamma) \in \Stab (\fs, P \cap M)} e_{w(\lambda_L) \otimes \gamma} V_\omega
= \sum_{\gamma \in X^L (\fs)} e_{\lambda_L \otimes \gamma} V_\omega .
\end{aligned}
\end{equation}
They carry in a natural way representations of $K_L$, namely \label{i:36}
\begin{equation}\label{eq:4.m}
\begin{aligned}
& \mu_L^1 = \bigoplus_{(w,\gamma) \in \Stab (\fs, P \cap M)^1 / \Stab (\fs,\lambda)} 
\! w(\lambda_L) \otimes \gamma
= \bigoplus_{\gamma \in X^L (\fs / \lambda)^1} \! \lambda_L \otimes \gamma , \\
& \mu_L  = \bigoplus_{(w,\gamma) \in \Stab (\fs, P \cap M) / \Stab (\fs,\lambda)} 
w(\lambda_L) \otimes \gamma
= \bigoplus_{\gamma \in X^L (\fs / \lambda)} \lambda_L \otimes \gamma .
\end{aligned}
\end{equation}
We lift them to representations
\begin{equation}\label{eq:3.13}
\begin{aligned}
& \mu^1 = \bigoplus\nolimits_{\gamma \in X^L (\fs / \lambda)^1} \lambda \otimes \gamma , \\
& \mu  = \bigoplus\nolimits_{\gamma \in X^L (\fs / \lambda)} \lambda \otimes \gamma 
\end{aligned}
\end{equation}
of $K$ by making it trivial on $K \cap U$ and on $K \cap \overline{U}$. 
In particular $\mu_L^1$ is the restriction of $\mu^1$ to $K \cap L$.
They relate to the idempotent $e^\fs_M$ by
\[
\sum_{a \in [L / H_\lambda]} \sum_{\gamma \in X^L (\fs / \lambda)} a e_{\gamma 
\otimes \mu^1} a^{-1} = \sum_{a \in [L / H_\lambda]} a e_{\mu } a^{-1} = e^\fs_M .
\]
It will turn out that $e_{\mu^1} \in \cH (K)$ is the idempotent of a type, for a
compact open subgroup of $M$ that contains $K$. 

The normalizer of the pair $(K_L,\lambda_L)$ is \label{i:38}
\[
N(K_L,\lambda_L) := \{ m \in N_L (K_L) \mid m \cdot \lambda_L \cong \lambda_L \} .
\]
\begin{lem}\label{lem:4.19}
$N(K_L,\lambda_L) = \widetilde{X^* (T_\fs)} K_L = K_L \widetilde{X^* (T_\fs)}$.
\end{lem}
\begin{proof}
By Theorem \ref{thm:3.7} 
\begin{equation}\label{eq:4.z}
N(K_L,\lambda_L) \subset K_L \widetilde{X^* (T_\fs)} K_L. 
\end{equation}
With conditions \ref{cond} and \eqref{eq:3.3}
we can be more precise. As discussed in the proof of Proposition \ref{prop:3.3},
\begin{equation}
K_L \widetilde{X^* (T_\fs)} = \prod\nolimits_i \Big( K_{L_i} (L_i \cap 
\widetilde{X^* (T_\fs)}) \Big)^{e_i} = \prod\nolimits_i \Big( K_{L_i} 
( D^\times 1_{L_i} \cap \widetilde{X^* (T_\fs)} ) \Big)^{e_i} .
\end{equation}
As $\lambda_L = \bigotimes\nolimits_i \lambda_{L_i}^{\otimes e_i}$, the group 
$N(K_L,\lambda_L)$ can be factorized similarly.
Consider any element of the form
\begin{equation}
k_i z_i \text{ with } k_i \in K_{L_i}, 
z_i \in D^\times 1_{L_i} \cap \widetilde{X^* (T_\fs)}.
\end{equation}
The group $K_{L_i}$, called $J(\beta,\mathfrak A)$ in \cite{Sec3}, is made from 
a stratum in $L_i = \GL_{m_i}(D)$, and therefore it is normalized by $z_i$, see
\cite[\S 2.3]{Sec3}. Furthermore $z_i$ 
belongs to the support of $e_{\lambda_i} \cH (L_i) e_{\lambda_i}$,
so it normalizes $(K_{L_i},\lambda_i)$. Knowing that, we can follow the proof of 
Proposition \ref{prop:3.3} with $k_i z_i$ in the role of $c_i$. It leads to the 
conclusion that $k_i z_i \in N(K_L,\lambda_L)$. It follows that 
\[
K_L \widetilde{X^* (T_\fs)} = \widetilde{X^* (T_\fs)} K_L \subset N(K_L,\lambda_L) .
\]
Combine this with \eqref{eq:4.z}.
\end{proof}

Inspired by \cite{BuKu2} we define two variants of the projective normalizer of
$(K_L,\lambda_L)$: \label{i:42}\label{i:43}
\begin{align*}
& PN (K_L,\lambda_L) := \{ m \in N_L (K_L) \mid m \cdot \lambda_L \cong
\lambda_L \otimes \gamma \text{ for some } \gamma \in X^L (\fs) \} , \\
& PN^1 (K_L,\lambda_L) := PN (K_L,\lambda_L) \cap L^1 .
\end{align*}

\begin{lem}\label{lem:3.4}
Recall the $K_L$-representations $\mu_L^1$ and 
$\mu_L$ from \eqref{eq:4.40} and \eqref{eq:4.m}.
\enuma{
\item $(\mu_L^1, V_{\mu_L^1})$ extends to an irreducible  
representation of $PN^1 (K_L,\mu_L)$.
\item $(PN^1 (K_L,\mu_L),\mu_L^1)$ is an $[L,\omega]_L$-type and 
\[
[PN^1 (K_L,\lambda_L) : K_L] = [\Stab (\fs,P \cap M)^1 : \Stab (\fs,\lambda)] = 
| X^L (\fs / \lambda)^1 | .
\]
\item $(\mu_L, V_{\mu_L })$ extends to an irreducible 
representation of $PN (K_L,\lambda_L)$ and
\[
[PN (K_L,\lambda_L) : N (K_L,\lambda_L)] = 
[\Stab (\fs,P \cap M) : \Stab (\fs,\lambda)] = | X^L (\fs / \lambda) | .
\]
}
\end{lem}
\begin{proof}
(a) Just as in \eqref{eq:3.26}, there is a canonical injective algebra homomorphism
\begin{equation}\label{eq:3.10}
\mathcal F_L : e_{\mu_L^1} \cH (L) e_{\mu_L^1} \to 
\mathcal O (X_{\nr}(L)) \otimes \End_\C (e_{\mu_L^1} V_\omega) .
\end{equation}
For $\gamma \in X^L (\fs)^1$ the element $c_\gamma \in L^1$ from Proposition 
\ref{prop:3.3} maps to $\End_\C (V_\omega)$ by the definition of $L^1$. Moreover
\[
e_{\lambda_L \otimes \gamma_1} \cH (L) e_{\lambda_L \otimes \gamma_2} =
c_{\gamma_1} e_{\lambda_L} \cH (L) e_{\lambda_L} c_{\gamma_2}^{-1} ,
\]
so by \eqref{eq:3.15} the image of \eqref{eq:3.10} is contained in
\[
\mathcal O (T_\fs) \otimes \End_\C (V_{\mu_L^1}) .
\]
As the different idempotents $e_{\lambda_L \otimes \gamma}$ are orthogonal, 
\begin{align*}
& V_{\mu_L^1} = \bigoplus\nolimits_{\gamma \in X^L (\fs / \lambda)^1} 
e_{\lambda_L \otimes \gamma} V_\omega ,\\
& \cH (K_L) e_{\mu_L^1} = e_{\mu_L^1} \cH (K_L) \cong
\bigoplus\nolimits_{\gamma \in X^L (\fs / \lambda)^1} 
\End_\C (e_{\lambda_L \otimes \gamma} V_\omega) .
\end{align*}
Furthermore $\mathcal F_L (\C \{ c_\gamma e_{\mu_L^1} : 
\gamma \in X^L (\fs / \lambda)^1 \})$
is a subspace of $\End_\C (V_{\mu_L^1})$ of dimension $|X^L (\fs / \lambda)^1|$.
So by the injectivity of \eqref{eq:3.10} the algebra homomorphism
\begin{equation}\label{eq:3.11}
\mathcal F_L : \cH (K_L) e_{\mu_L^1} \otimes \C \{ c_\gamma e_{\mu_L^1} : 
\gamma \in X^L (\fs / \lambda)^1 \} \to \End_\C (V_{\mu_L^1})
\end{equation}
is bijective. Consider any $m \in PN^1 (K_L,\lambda_L)$. It permutes the 
$\lambda_L \otimes \gamma$ with $\gamma \in X^L (\fs)^1$, so it commutes with 
$e_{\mu_L^1}$. Also $\mathcal F_L (m e_{\mu_L^1}) \in \End_\C (V_{\mu_L^1})$ 
because $m \in L^1$. So by the injectivity of \eqref{eq:3.10} and the 
surjectivity of \eqref{eq:3.11}, $m e_{\mu_L^1} = f e_{\mu_L^1}$ for some 
\[
f \in \cH (K_L \cdot \{ c_\gamma : \gamma \in X^L (\fs / \lambda)^1 \}) .
\]
Consequently $m \in K_L \cdot \{ c_\gamma : \gamma \in X^L (\fs / \lambda)^1 \}$ and
\begin{equation}\label{eq:3.12}
PN^1 (K_L,\lambda_L) = K_L \cdot \{ c_\gamma : \gamma \in X^L (\fs / \lambda)^1 \}
\end{equation}
Now \eqref{eq:3.11} shows that $V_{\mu_L^1} = e_{\mu_L^1} V_\omega$ 
is an irreducible representation of \eqref{eq:3.12}. \\
(b) All the pairs $(K_L,\lambda_L \otimes \gamma)$ are simple types for the same
supercuspidal equivalence class $[L,\omega]_L$, so by \cite[Corollary 7.3]{SeSt6}
the idempotents $e_{\lambda_L \otimes \gamma}$ are $L$-conjugate. (In fact 
$e_{\lambda_L}$ and $e_{\lambda_L \otimes \gamma}$ are conjugate by the element
$c_\gamma$ from Proposition \ref{prop:3.3}.)
Hence the category $\Rep^{\mu_L^1}(L)$ equals $\Rep^{\lambda_L}(L)$, and 
$(PN^1 (K_L,\lambda_L),\mu_L^1)$ is a type for this factor of $\Rep (L)$. 
The claims about the indices follow from \eqref{eq:3.12}. \\
(c) For every $\gamma \in X^L (\fs)$ the element $c_\gamma \in L$ is 
unique up to $N(K_L,\lambda_L)$, so 
\begin{equation}\label{eq:4.81}
PN (K_L,\lambda_L) = N(K_L,\lambda_L) \{ c_\gamma \mid \gamma \in X^L (\fs) \} .
\end{equation}
Together with Proposition \ref{prop:3.3} this proves the claims about 
$[PN (K_L,\lambda_L) : N(K_L,\lambda_L)]$. 

Part (a), and the map \eqref{eq:3.10} show that $\mu_L^1$ extends to an 
irreducible representation of $PN^1 (K_L,\lambda) N(K_L,\lambda_L)$.
The same holds for $\gamma \otimes \mu_L^1$ with $\gamma \in X^L (\fs)$. We have
\[
V_{\mu_L } = \bigoplus\nolimits_{\gamma \in X^L (\fs / \lambda)} 
V_{\gamma \otimes \mu_L^1}
\]
as representations of $PN^1 (K_L,\lambda_L) N(K_L,\lambda_L)$, and these subspaces 
are permuted transitively by the $c_\gamma$ with $\gamma \in X^L (\fs)$. 
This and \eqref{eq:4.81} show that $V_{\mu_L }$ extends to an irreducible 
representation of $PN (K_L,\lambda_L)$.
\end{proof}

Lemma \ref{lem:3.4} has an analogue in $M$. 
To state it we rather start with the sets \label{i:40} \label{i:41}
\begin{equation}\label{eq:4.5}
\begin{aligned}
& PN^1 (K,\lambda) := (K \cap \overline{U}) PN^1 (K_L,\lambda_L) (K \cap U) , \\
& PN (K,\lambda) \; := (K \cap \overline{U}) PN (K_L,\lambda_L) (K \cap U) .
\end{aligned}
\end{equation}

\begin{lem}\label{lem:3.5}
\enuma{
\item The multiplication map 
\[
(K \cap \overline{U}) \times PN^1 (K_L,\lambda_L) \times (K \cap U) \to PN^1 (K,\lambda)
\]
is bijective and $PN^1 (K,\lambda)$ is a compact open subgroup of $M$.
\item $\mu^1$ extends to an irreducible $PN^1 (K,\lambda)$-representation
and $(PN^1 (K,\lambda),\mu^1)$ is an $\fs_M$-type.
\item $PN (K,\lambda)$ is a group and the multiplication map
\[
(K \cap \overline{U}) \times PN (K_L,\lambda_L) \times (K \cap U) \to PN (K,\lambda)
\]
is bijective. Furthermore $\mu $ extends to an irreducible $PN (K,\lambda)$-representation.
}
\end{lem}
\emph{Remark.} We will show later that $PN^1 (K,\lambda)$ and $PN (K,\lambda)$ 
are really the projective normalizers of $(K,\lambda)$ in $M^1$ and $M$, respectively.
\begin{proof}
(a) By \cite[Proposition 5.3]{Sec3} the multiplication map
\begin{equation}\label{eq:3.32}
(K \cap \overline{U}) \times (K \cap L) \times (K \cap U) \to K 
\end{equation}
is a homeomorphism. By \eqref{eq:3.12} and because $c_\gamma \in L^1$ normalizes 
$K \cap U$ and $K \cap \overline{U}$ (see the proof of Proposition \ref{prop:3.3}), 
the analogue of \eqref{eq:3.32} for $PN^1 (K,\lambda)$ holds as well. At the same time 
this shows that $PN^1 (K,\lambda)$ is compact, for its three factors are. 
As $PN^1 (K_L,\lambda_L)$ normalizes $K \cap U$ and $K \cap \overline{U}$ this 
factorization also proves that $PN^1 (K,\lambda)$ is a group. Since $K$ is
open in $M$, so is the larger group $PN^1 (K,\lambda)$.\\
(b) In view of Lemma \ref{lem:3.4}.a and \eqref{eq:3.13} we can extend $\mu^1$ to
$PN^1 (K,\lambda)$ by
\[
\mu^1 (\overline{n} m n) := \mu_L^1 (m) ,
\]
where $\overline{n} m n$ is as in the decomposition from part (a). Then $\mu^1$ 
is irreducible because $\mu_L^1$ is. Because $(K,w(\lambda) \otimes \gamma)$ is an 
$\fs_M$-type, for each $(w,\gamma) \in \Stab (\fs, P \cap M)$, the category 
$\Rep^{\mu^1}(M)$ equals $\Rep^\lambda (M) = \Rep^{\fs_M}(M)$, and 
$(PN^1 (K,\lambda), \mu^1)$ is an $\fs_M$-type.\\
(c) The first two claims can be shown in the same way as part (a), using \eqref{eq:4.81}
instead of \eqref{eq:3.12}. For the last assertion we employ Lemma \ref{lem:3.4}.c
and set 
\[
\mu  (\overline{n} m n) := \mu_L  (m),
\]
with respect to the factorization we just established.
\end{proof}

Now we can determine the structure of $e^\fs_M \cH (M) e^\fs_M$ and some related algebras.

\begin{thm}\label{thm:3.6}
\enuma{
\item
There exist canonical algebra isomorphisms 
\[
e_{\mu_L^1} \cH (L) e_{\mu_L^1} \cong \cH (L,\mu_L^1) \otimes 
\End_\C (V_{\mu_L^1}) \cong \mathcal O (T_\fs) \otimes \End_\C (V_{\mu^1}) .
\]
The support of the left hand side is
$PN^1 (K_L,\lambda_L) \widetilde{X^* (T_\fs)} PN^1 (K_L,\lambda_L)$
\item Part (a) extends to algebra isomorphisms
\begin{align*}
& e_{\mu_L} \cH (L) e_{\mu_L} \cong \mathcal O (T_\fs) \otimes \End_\C (V_\mu) , \\
& e^\fs_L \cH (L) e^\fs_L \cong 
\mathcal O (T_\fs) \otimes \End_\C (V_\mu) \otimes M_{|L / H_\lambda|}(\C) .
\end{align*}
The support of $e_{\mu_L} \cH (L) e_{\mu_L}$ is 
$PN (K_L,\lambda_L) \widetilde{X^* (T_\fs)} PN (K_L,\lambda_L)$.
\item There exist algebra isomorphisms
\[
e_{\mu^1} \cH (M) e_{\mu^1} \cong \cH (M,\mu^1) \otimes \End_\C (V_{\mu^1}) 
\cong \cH (T_\fs, W_\fs, q_\fs) \otimes \End_\C (V_{\mu^1}) ,
\]
which are canonical up to the choice of the parabolic subgroup $P$. 
The support of the left hand side is
$PN^1 (K,\lambda) \widetilde{X^* (T_\fs)} W_\fs PN^1 (K,\lambda)$.
\item Part (c) extends to algebra isomorphisms
\begin{align*}
& e_\mu \cH (M) e_\mu \cong \cH (T_\fs, W_\fs, q_\fs) \otimes \End_\C (V_\mu) , \\
& e^\fs_M \cH (M) e^\fs_M \cong \cH (T_\fs, W_\fs, q_\fs)
\otimes \End_\C (V_\mu) \otimes M_{|L / H_\lambda|}(\C) .
\end{align*}
The support of $e_\mu \cH (M) e_\mu$ is 
$PN (K,\lambda) \widetilde{X^* (T_\fs)} W_\fs PN (K,\lambda)$.
}
\end{thm}
\begin{proof}
(a) By the Morita equivalence of $e^\fs_L \cH (L) e^\fs_L$ and
$e_{\lambda_L} \cH (L) e_{\lambda_L}$, there are isomorphisms of $e^\fs_L \cH (L)
e^\fs_L$-bimodules
\begin{equation}\label{eq:3.18}
\begin{split}
& e^\fs_L \cH (L) e^\fs_L \cong e^\fs_L \cH (L) e_{\lambda_L} 
\otimes_{e_{\lambda_L} \cH (L) e_{\lambda_L}} e_{\lambda_L} \cH (L) e^\fs_L \\
= & \bigoplus_{a_1,a_2 \in [L / H_\lambda]} \bigoplus_{\gamma_1 ,\gamma_2 \in 
X^L (\fs / \lambda)} a_1 e_{\lambda_L \otimes \gamma_1} a_1^{-1} \cH (L) 
e_{\lambda_L} \otimes_{e_{\lambda_L} \cH (L) e_{\lambda_L}} 
e_{\lambda_L} \cH (L) a_2 e_{\lambda_L \otimes \gamma_2} a_2^{-1} \\
= & \bigoplus_{a_1,a_2 \in [L / H_\lambda]} \bigoplus_{\gamma_1 ,\gamma_2 
\in X^L (\fs / \lambda)} a_1 c_{\gamma_1} e_\lambda c_{\gamma_1}^{-1} a_1^{-1} \cH (L) 
e_{\lambda_L} \cH (L) a_2 c_{\gamma_2} e_{\lambda_L} c_{\gamma_2}^{-1} a_2^{-1} \\
= & \bigoplus_{a_1,a_2 \in [L / H_\lambda]} 
\bigoplus_{\gamma_1 ,\gamma_2 \in X^L (\fs / \lambda)} 
a_1 c_{\gamma_1} e_{\lambda_L} \cH (L) e_{\lambda_L} c_{\gamma_2}^{-1} a_2^{-1}.
\end{split}
\end{equation}
Here the subalgebra $e_{\mu_L^1} \cH (L) e_{\mu_L^1}$ corresponds to 
\[
\bigoplus\nolimits_{\gamma_1 ,\gamma_2 \in X^L (\fs / \lambda)^1} 
c_{\gamma_1} e_{\lambda_L} \cH (L) e_{\lambda_L} c_{\gamma_2}^{-1} .
\]
In combination with \eqref{eq:3.15} it follows that the canonical map 
\eqref{eq:3.10} is an isomorphism
\begin{equation}\label{eq:3.30}
\mathcal F_L : e_{\mu_L^1} \cH (L) e_{\mu_L^1} \to
\mathcal O (T_\fs) \otimes \End_\C (V_{\mu^1}).
\end{equation}
This and Theorem \ref{thm:3.7} imply that the support of 
$e_{\mu_L^1} \cH (L) e_{\mu_L^1}$ is as indicated.
Notice that $\mathcal O (T_\fs)$ is the commutant of $\End_\C (V_{\mu^1})$
in $\mathcal O (T_\fs) \otimes \End_\C (V_{\mu^1})$. Hence it corresponds
to $\cH (L,\mu_L^1)$ under the canonical isomorphism
\[
e_{\mu_L^1} \cH (L) e_{\mu_L^1} \cong \cH (L,\mu_L^1)
\otimes \End_\C (V_{\mu^1}) .
\]
(b) By \eqref{eq:3.18} 
\[
e_{\mu_L} \cH (L) e_{\mu_L} = \bigoplus_{\gamma_1 ,\gamma_2 \in X^L (\fs / \lambda)} 
c_{\gamma_1} e_{\lambda_L} \cH (L) e_{\lambda_L} c_{\gamma_2}^{-1} ,
\]
and by part (a) and \eqref{eq:4.81} its support is 
$PN (K_L,\lambda_L) \widetilde{X^* (T_\fs)} PN (K_L,\lambda_L)$.
Also by \eqref{eq:3.18}, we can identify $e^\fs_L \cH (L) e^\fs_L$ as a
vector space with
\[
\bigoplus_{a_3 \in [L / H_\lambda], \gamma_3 \in X^L (\fs) / X^L (\fs)^1}
\C a_3 \gamma_3 \otimes \bigoplus_{a_1 \in [L / H_\lambda], \gamma_1 \in 
X^L (\fs) / X^L (\fs)^1} a_1 e_{\mu_L^1 \otimes \gamma_1} \cH (L)
e_{\mu_L^1 \otimes \gamma_1} a_1^{-1}.
\]
From \eqref{eq:3.30} we get an isomorphism
\begin{multline}\label{eq:4.70}
\bigoplus_{a_1 \in [L / H_\lambda], \gamma_1 \in 
X^L (\fs) / X^L (\fs)^1} a_1 e_{\mu_L^1 \otimes \gamma_1} \cH (L)
e_{\mu_L^1 \otimes \gamma_1} a_1^{-1} \to \\
\bigoplus_{a_1 \in [L / H_\lambda], \gamma_1 \in X^L (\fs) / X^L (\fs)^1} 
\mc O (T_\fs) \otimes \End_\C (a_1 V_{\mu_L^1 \otimes \gamma_1}) .
\end{multline}
Recall that 
\begin{multline*}
V_{\mu_L} = e_{\mu_L} \cH (L) e_{\lambda_L} \otimes_{e_{\lambda_L} \cH (L) e_{\lambda_L}} 
V_{\lambda_L} = e_{\mu_L} \cH (L) e_{\mu_L^1} \otimes_{e_{\mu_L^1} \cH (L) e_{\mu_L^1}}
V_{\mu_L^1} = \\
\bigoplus_{\gamma_1 \in X^L (\fs) / X^L (\fs)^1} c_{\gamma_1} e_{\mu_L^1} \otimes V_{\mu_L^1} 
= \bigoplus_{\gamma_1} c_{\gamma_1} V_{\mu_L^1} .
\end{multline*}
For $\gamma_3 \in X^L (\fs)$ the choice of $c_{\gamma_3}$ is unique up 
$N(K_L,\lambda_L)$, by Lemma \ref{lem:3.4}.c. The particular shape \eqref{eq:3.9} 
implies that it is in fact unique up to $N(K_L,\lambda_L)^{W_\fs}$, so 
\begin{equation}\label{eq:4.86}
c_{\gamma_3} c_{\gamma_1} \text{ differs from }
c_{\gamma_3 \gamma_1} \text{ by an element of } N(K_L,\lambda_L)^{W_\fs}.
\end{equation}
With Lemma \ref{lem:4.19} we deduce that left multiplication by $c_{\gamma_3}$ 
defines a bijection
\[
V_{\gamma_1 \otimes \mu_L^1} = c_{\gamma_1} V_{\mu_L^1} \to 
c_{\gamma_3 \gamma_1} V_{\mu_L^1} = V_{\gamma_3 \gamma_1 \otimes \mu_L^1}
\]
which depends on $\omega \otimes \chi \in \Irr^{\fs_L}(L)$ in an algebraic way. 
More precisely,
\begin{equation}\label{eq:4.82}
c_{\gamma_3} e_{\gamma_1 \otimes \mu^1_L} \in \mc O (T_\fs )^{W_\fs} \otimes 
\Hom_\C \big( V_{\gamma_1 \otimes \mu_L^1}, V_{\gamma_3 \gamma_1 \otimes \mu_L^1} \big).
\end{equation}
Consequently \eqref{eq:4.70} extends to an algebra isomorphism 
\begin{equation}\label{eq:4.83}
e_{\mu_L } \cH (L) e_{\mu_L } \to \mc O (T_\fs) \otimes \End_\C (V_{\mu_L }) .
\end{equation}
It is more difficult to see what \eqref{eq:4.70} should look like for elements of 
$[L / H_\lambda]$. For those we use a different, inexplicit argument.

For each $a \in [L / H_\lambda]$ the inclusion
\[
a e_{\mu_L } a^{-1} \cH (L) a e_{\mu_L } a^{-1} \to e^\fs_L \cH (L) e^\fs_L
\]
is a Morita equivalence, because the idempotents $a e_{\mu_L } a^{-1}, e^\fs_L$
and $e_{\lambda_L}$ all see exactly the same category of $L$-representations, namely
$\Rep^{\fs_L}(L)$. For every $V \in \Rep^{\fs_L}(L)$ we have
\[
e^\fs_L V = \bigoplus\nolimits_{a \in [L / H_\lambda]} a e_{\mu_L } a^{-1} V,
\]
where all the summands have the same dimension. It follows that
\[
e^\fs_L \cH (L) e^\fs_L \cong e_{\mu_L } \cH (L) e_{\mu_L } \otimes
M_{| L / H_\lambda|}(\C) .
\]
By \eqref{eq:4.83} the right hand side is isomorphic to
\[
\mc O (T_\fs) \otimes  \End_\C (V_{\mu_L }) \otimes M_{| L / H_\lambda|}(\C)  
\cong \mc O (T_\fs) \otimes  \End_\C (e_L^\fs V) .
\]
(c) Just like \eqref{eq:3.18} there is an isomorphism of $e^\fs_M \cH (M) e^\fs_M$-bimodules
\[
e^\fs_M \cH (M) e^\fs_M \cong \bigoplus_{a_1,a_2 \in [L / H_\lambda]} 
\bigoplus_{\gamma_1 ,\gamma_2 \in X^L (\fs / \lambda)} 
a_1 c_{\gamma_1} e_\lambda \cH (M) e_\lambda c_{\gamma_2}^{-1} a_2^{-1} ,
\]
and it extends 
\[
e_\mu \cH (M) e_\mu \cong \bigoplus_{\gamma_1 ,\gamma_2 \in X^L (\fs / \lambda)} 
c_{\gamma_1} e_\lambda \cH (M) e_\lambda c_{\gamma_2}^{-1} .
\]
For $x \in X^* (T_\fs) \rtimes W_\fs$ let $f_{x,\lambda} \in e_\lambda \cH (M) e_\lambda$ 
be the element that corresponds to 
\[
[x] \in \cH (X^* (T_\fs) \rtimes W_\fs,q_\fs) \cong \cH (T_\fs, W_\fs,q_\fs)
\] 
via Theorem \ref{thm:3.7}. The elements $f_{x,\lambda}$ commute with 
$e_\lambda \cH (K) e_\lambda \cong \End_\C (V_\lambda)$. It follows that the element
\begin{equation}\label{eq:3.17}
c_\gamma f_{x,\lambda} c_\gamma^{-1} = c_\gamma e_\lambda f_{x,\lambda} e_\lambda 
c_\gamma^{-1} = e_{\lambda \otimes \gamma} c_\gamma f_{x,\lambda}
c_\gamma^{-1} e_{\lambda \otimes \gamma}
\end{equation}
is independent of the choice of $c_\gamma$ in Proposition \ref{prop:3.3}. As conjugation
by $c_\gamma$ turns the commutative diagram \eqref{eq:3.23} into the corresponding
diagram for $\lambda \otimes \gamma$, we have
\begin{equation} 
c_\gamma f_{x,\lambda} c_\gamma^{-1} = f_{x,w(\lambda) \otimes \gamma},
\end{equation}
the image of $[x]$ in $e_{w(\lambda) \otimes \gamma} \cH (M) e_{\lambda \otimes \gamma}$
under the canonical isomorphisms from Theorem \ref{thm:3.7}. For every
$x \in X^* (T_\fs) \rtimes W_\fs$ we define
\begin{equation}\label{eq:3.62}
f_{x,\mu^1} := \sum_{\gamma \in X^L (\fs / \lambda)^1} c_\gamma f_{x,\lambda} c_\gamma^{-1} 
= \sum_{\gamma \in X^L (\fs / \lambda)^1} f_{x,\lambda \otimes \gamma}
\in e_{\mu^1} \cH (M) e_{\mu^1} .
\end{equation}
By \eqref{eq:3.17} $f_{x,\mu^1}$ commutes with $e_{\mu^1} \cH (K) e_{\mu^1}$ and
with the $c_\gamma$ for $\gamma \in X^L (\fs / \lambda)^1$, so it commutes with 
$e_{\mu^1} \cH (PN^1 (K,\lambda)) e_{\mu^1}$. 
By \eqref{eq:3.18} and Theorem \ref{thm:3.7}
\[
e_{\mu^1} \cH (M) e_{\mu^1} = \bigoplus\nolimits_{x \in X^* (T_\fs) \rtimes W_\fs}
\C f_{x,\mu^1} \otimes e_{\mu^1} \cH (PN^1 (K,\lambda)) e_{\mu^1} ,
\]
and the support of this algebra is $PN^1 (K,\lambda) \widetilde{X^* (T_\fs)} 
W_\fs PN^1 (K,\lambda)$.

The orthogonality of the different idempotents $e_{\lambda \otimes \gamma}$ implies 
that the $f_{x,\mu^1}$ satisfy the same multiplication rules as the $f_{x,\lambda}$. 
Hence the span of the $f_{x,\mu^1}$ is a subalgebra of $e_{\mu^1} \cH (M) e_{\mu^1}$ 
isomorphic with $\cH (X^* (T_\fs) \rtimes W_\fs ,q_\fs)$. 
We constructed an algebra isomorphism
\begin{equation}\label{eq:3.19}
e_{\mu^1} \cH (M) e_{\mu^1} \cong \cH (X^* (T_\fs) \rtimes W_\fs,q_\fs)
\otimes \End_\C (V_{\mu^1}) .
\end{equation}
Since $\cH (M,\mu^1)$ is the commutant of $\End_\C (V_{\mu^1})$ inside
\begin{equation}\label{eq:3.20}
\cH (M,\mu^1) \otimes \End_\C (V_{\mu^1}) \cong 
e_{\mu^1} \cH (M) e_{\mu^1} ,
\end{equation}
it corresponds to $\cH (X^* (T_\fs) \rtimes W_\fs ,q_\fs)$ under the isomorphisms 
\eqref{eq:3.19} and \eqref{eq:3.20}.

Tensored with the identity on $\End_\C (V_\lambda) ,\; t_{P,\lambda}$ 
from \eqref{eq:3.21} becomes a canonical injection
\begin{equation}\label{eq:3.92}
e_{\lambda_L} \cH (L) e_{\lambda_L} \cong \cH (L,\lambda_L) \otimes \End_\C (V_{\lambda}) 
\to \cH (M,\lambda) \otimes \End_\C (V_{\lambda}) \cong e_{\lambda} \cH (M) e_{\lambda} .
\end{equation}
Since $t_{P,\lambda}$ and the analogous map $t_{P,\mu^1}$ for $\mu^1$ are
uniquely defined by the same property, they agree in the sense that
\begin{equation}\label{eq:3.22}
t_{P,\lambda} \otimes \text{id} = t_{P,\mu^1} \otimes \text{id   on   }
e_{\lambda_L} \cH (L) e_{\lambda_L} \cong \mathcal O (T_\fs) \otimes \End_\C (V_\lambda) . 
\end{equation}
Consequently the isomorphisms \eqref{eq:3.19} and \eqref{eq:3.30} fit in a commutative
diagram
\begin{equation}\label{eq:3.29}
\begin{array}{ccc}
\cH (M,\mu^1) & \to & \cH (X^* (T_\fs) \rtimes W_\fs ,q_\fs) \cong
\cH (T_\fs, W_\fs ,q_\fs) \\
 \uparrow {\scriptstyle t_{P,\mu^1}} & & \uparrow \scriptstyle{i_{P,\mu^1}} \\
\cH (L,\mu_L^1) & \to & \mathcal O (T_\fs) \cong \C [X^* (T_\fs)] ,
\end{array}
\end{equation}
Here $i_{P,\mu^1}$ is defined like $i_{P,\lambda}$, see \eqref{eq:3.24}. 
In this sense \eqref{eq:3.19} is canonical.\\
(d) Part (c) works equally well with $\gamma_1 \otimes \mu^1$ instead of $\mu^1$. 
For all $\gamma_1 \in X^L (\fs)$ together that gives a canonical isomorphism 
\begin{equation}\label{eq:4.84}
\bigoplus_{\gamma_1 \in X^L (\fs) / X^L (\fs)^1} e_{\gamma_1 \otimes \mu^1} \cH (M) 
e_{\gamma_1 \otimes \mu^1} \to \cH (T_\fs, W_\fs, q_\fs) \otimes 
\bigoplus_{\gamma_1 \in X^L (\fs) / X^L (\fs)^1} \End_\C (V_{\gamma_1 \otimes \mu^1}) .
\end{equation}
The formula \eqref{eq:4.82} defines an element of $\cH (T_\fs, W_\fs, q_\fs) \otimes 
\End_\C (V_\mu)$ which commutes with $\cH (T_\fs, W_\fs, q_\fs)$. Therefore we can
extend \eqref{eq:4.84} to isomorphisms
\begin{align*}
& e_{\mu } \cH (M) e_{\mu } \to \cH (T_\fs, W_\fs, q_\fs) \otimes \End_\C (V_\mu ), \\
& e^\fs_M \cH (M) e^\fs_M \to \cH (T_\fs, W_\fs, q_\fs) \otimes \End_\C (e_M^\fs V)
\end{align*}
in the same way as we did in the proof of part (b). The support of 
$e_\mu \cH (M) e_\mu$ can also be determined as in part (b), using the support
of $e_{\mu^1} \cH (M) e_{\mu^1}$, as determined in part (c).
\end{proof}

We note that a vector space basis of $\cH (T_\fs, W_\fs, q_\fs) \subset e_{\mu } 
\cH (M) e_{\mu }$ is formed by the elements \label{i:17}
\begin{equation} \label{eq:4.85} 
f_{x,\mu } = \sum_{\gamma \in X^L (\fs / \lambda)} c_\gamma f_{x,\lambda} c_\gamma^{-1}
= \sum_{\gamma \in X^L (\fs / \lambda)} f_{x,\lambda \otimes \gamma} .
\end{equation}

We define the projective normalizer of $(K,\lambda)$ in $M$ as 
\begin{equation}\label{eq:3.34} 
\{ g \in N_{M}(K) \mid g \cdot \lambda \cong \lambda \otimes \gamma
\text{ for some } \gamma \in X^L (\fs) \} .
\end{equation}
Using the explicit information gathered in the above proof, 
we can show that it is none other than $PN (K,\lambda)$ as defined in \eqref{eq:4.5}.

\begin{lem}\label{lem:3.8}
Recall the $K$-representations $\mu$ and $\mu^1$ from \eqref{eq:3.13}.
\enuma{
\item $(PN^1 (K,\lambda),\mu^1)$ is a cover of $(PN^1 (K_L ,\lambda_L),\mu^1_L)$. 
\item $PN (K,\lambda)$ equals the projective normalizer of $(K,\lambda)$ in $M$.
\item $PN^1 (K,\lambda)$ equals the projective normalizer of $(K,\lambda)$ in $M^1$.
}
\end{lem}
\begin{proof}
(a) For the definition of a cover we refer to \cite[8.1]{BuKu3}. By Lemma \ref{lem:3.5}
$PN^1 (K,\lambda) \cap L = PN^1 (K_L ,\lambda_L)$ and by \cite[Proposition 5.5]{Sec3} 
$K$ admits an Iwahori decomposition with respect to any parabolic subgroup of $M$ with 
Levi factor $L$. Hence $PN^1 (K,\lambda)$ is also decomposed in this sense. The second 
condition for a cover says that $\mu \big|_{N (K_L,\mu_L)} = \mu_L$, which is true by 
definition. The third condition is about the existence of an invertible ``strongly 
positive" element in $\cH (M,\mu^1)$. By \cite[Proposition 5.5]{Sec3} $\cH (M, \lambda)$ 
contains such an element, in the notation of the proof of Theorem \ref{thm:3.6} 
it corresponds to $f_{x,\lambda}$ for a suitable $x \in X^* (T_\fs)$. 
Then $f_{x,\mu^1}$ and its image in $\cH (M,\mu^1)$ have the correct properties. \\
(b) By Lemma \ref{lem:3.5} $PN (K,\lambda)$ is contained in this normalizer.

Consider any $g$ in the group \eqref{eq:3.34}. Its intertwining property entails that 
$e_{\mu } g e_{\mu } \in e_{\mu } \cH (M) e_{\mu }$
has inverse $e_{\mu } g^{-1} e_{\mu }$. From Theorem \ref{thm:3.7} we
can see what the support of $e_{\mu } \cH (M) e_{\mu }$ is, namely
\[
PN (K,\lambda) \widetilde{X^* (T_\fs)} W_\fs PN (K,\lambda) .
\] 
Possibly adjusting $g$ from the left and from the right by an element of 
$PN (K,\lambda)$, we may assume that 
\[
g \in \widetilde{X^* (T_\fs)} W_\fs \subset N_M (L) .
\]
Then $g$ also normalizes $(K_L,\mu_L )$. With Conditions \ref{cond} we see easily 
that every element of $W_\fs$ normalizes $(K_L,\mu_L )$. Writing $g = x w$ with
$x \in \widetilde{X^* (T_\fs)}, w \in W_\fs$, we find that $x \in L$ normalizes
$(K_L,\mu_L )$ as well. In other words 
\[
x \in PN (K_L,\lambda_L) \subset PN (K,\lambda) .
\]
It follows that $w \in W_\fs$ must also normalize $(K,\mu )$.
By considering supports in Theorem \ref{thm:3.6}.c, we see that
\[
e_{\mu^1} w e_{\mu^1} \mapsto [w] \otimes \End_\C (V_{\mu^1}) \subset
\cH (T_\fs,W_\fs,q_\fs) \otimes \End_\C (V_{\mu^1}) .
\]
By \eqref{eq:4.84} there exists a unique $h_1 \in \bigoplus_{\gamma_1 \in X^L (\fs) 
/ X^L (\fs)^1} \End_\C (V_{\gamma_1 \otimes \mu^1})$ such that
\[
e_\mu w e_\mu = \sum_{\gamma_1 \in X^L (\fs) / X^L (\fs)^1} e_{\gamma_1 \otimes \mu^1}
w e_{\gamma_1 \otimes \mu^1} \mapsto [w] \otimes h_1 \in 
\cH (T_\fs,W_\fs,q_\fs) \otimes \End_\C (V_\mu) .
\]
Similarly $e_\mu w^{-1} e_\mu$ maps to $[w^{-1}] \otimes h_2$ under \eqref{eq:4.84},
for some $h_2 \in \End_\C (V_{\mu })$. Because $w$ lies in \eqref{eq:3.34},
\begin{align*}
& e_{\mu } w e_{\mu } \cdot e_{\mu } w^{-1} e_{\mu } = e_{\mu } , \\
& [w] \otimes h_1 \cdot [w^{-1}] \otimes h_2 = [w] \cdot [w^{-1}] \otimes h_1 h_2 = 
[1] \otimes \text{id} .
\end{align*}
In particular $[w] [w^{-1}] \in \C [1]$. The multiplication rules for $\cH (W_\fs,q_\fs)$ 
\cite[2.5.3]{Sec3}, applied with induction to the length of $w \in W_\fs$,
show that this is only possible if $w = 1$. Consequently $g = x \in PN (K,\lambda) \cap
\widetilde{X^* (T_\fs)}$.\\
(c) This follows immediately from (b). 
\end{proof}

\begin{rem}\label{rem:4.notype}
In Lemma \ref{lem:3.8} we construct a $\fs_M$-type with representation 
$\mu^1$, but we do not succeed in finding a $\fs_M$-type with representation $\mu$. 
The obstruction appears to be that some of the representations $\lambda 
\otimes \gamma$ are conjugate in $M$, but not via an element of $M^1$. Examples
\ref{ex:5.2} and \ref{ex:5.3} show that this can really happen when $G$ is not split.

Mainly for this reason we have been unable to construct types for all Bernstein
components of $G^\sharp$. In the special case where all the $K_G$-representations
$\lambda_G \otimes \gamma$ with $\gamma \in X^G (\fs)$ are conjugate via elements of
$G^1$, we can construct types for every Bernstein component $\mf t^\sharp \prec \fs$.
We did not include this in the paper because it is quite some work and it is not
clear how often these extra conditions are fulfilled.
\end{rem}

\subsection{Hecke algebras for the intermediate group} \

In \eqref{eq:3.61} we constructed an idempotent $e^\fs_M$, using the set 
$[L / H_\lambda]$ from Lemma \ref{lem:3.30}. In Proposition \ref{prop:3.2} and 
Lemma \ref{lem:3.1} that the algebras
\[
\big( e^\fs_M \cH (M) e^\fs_M \big)^{X^L (\fs)} 
\rtimes \mf R_\fs^\sharp  \quad \text{and} \quad 
\big( e^\fs_M \cH (M) e^\fs_M \big) \rtimes \Stab (\fs, P \cap M).
\]
are Morita equivalent with $\cH (G^\sharp Z(G))^\fs$. In Theorem \ref{thm:3.17} we
showed that the first one is even isomorphic to a subalgebra of $\cH (G^\sharp Z(G))^\fs$
determined by an idempotent.
In Lemma \ref{lem:3.3} we saw that the actions of $X^L (\fs)$ and $\mf R_\fs^\sharp$
both come from the action $\alpha$ of $\Stab (\fs,P \cap M)$ defined in \eqref{eq:3.35}. 

\begin{lem}\label{lem:4.9}
There is an equality
\[
\big( e^\fs_M \cH (M) e^\fs_M \big)^{X^L (\fs,\lambda)} = 
\bigoplus\nolimits_{a \in [L / H_\lambda]} 
\big( a e_\mu a^{-1} \cH (M) a e_\mu a^{-1} \big)^{X^L (\fs,\lambda)} ,
\]
and this algebra is $\Stab (\fs,P \cap M)$-equivariantly isomorphic to
\[
\bigoplus\nolimits_1^{|X^L (\omega,V_\mu)|} 
\big( e_\mu \cH (M) e_\mu \big)^{X^L (\fs,\lambda)} .
\]
\end{lem}
\emph{Remark.} Here and below we use the notation $\bigoplus_1^n$ for the direct
sum of $n$ copies of something.
\begin{proof}
As vector spaces 
\begin{equation}\label{eq:4.3}
e^\fs_M \cH (M) e^\fs_M = \bigoplus\nolimits_{a_1,a_2 \in [L / H_\lambda]} 
a_1 e_\mu \cH (M) e_\mu a_2^{-1} .
\end{equation}
On the other hand, for any $\cH (M)^{\fs_M}$-module $V$ we have the decompositions
\begin{equation}\label{eq:4.33} 
e^\fs_M V = \bigoplus_{a \in [L / H_\lambda]} a e_{\mu } a^{-1} V =
\bigoplus_{\rho \in \Irr (\C [X^L (\fs,\lambda) \cap X^L (\omega),\kappa_\omega])}
e^\fs_M V_\rho . 
\end{equation}
Here every $a e_\mu a^{-1} V$ equals $\bigoplus_{\rho \in I_a} e^\fs_M V_\rho$
for a suitable collection $I_a$ of $\rho$'s. If $f \in e^\fs_M \cH (M) e^\fs_M$ is
invariant under $X^L (\fs,\lambda) \cap X^L (\omega)$, then it commutes with the
idempotents $e_\rho \in \C [X^L (\fs,\lambda) \cap X^L (\omega),\kappa_\omega])$,
so it stabilizes each of the subspaces $e^\fs_M V_\rho$. Therefore it also preserves 
the rougher decomposition $e^\fs_M V = \bigoplus_{a \in [L / H_\lambda]} 
a e_{\mu } a^{-1} V$. In view of \eqref{eq:4.3}, this is only possible if
\[
f \in \bigoplus\nolimits_{a \in [L / H_\lambda]} 
a e_\mu a^{-1} \cH (M) a e_\mu a^{-1} ,
\] 
which proves the desired equality. 

By Lemma \ref{lem:3.30}.b conjugation by $a \in [L / H_\lambda]$ gives a 
$\Stab (\fs,P \cap M)$-equivariant isomorphism
\[
e_{\mu } \cH (M) e_{\mu } \to a e_{\mu } \cH (M) e_{\mu } a^{-1} =
a e_{\mu } a^{-1} \cH (M) a e_{\mu } a^{-1} . 
\]
By Lemma \ref{lem:3.13} $|L / H_\lambda| = | \Irr (X^L (\omega,V_\mu)) |$ and
this equals $| X^L (\omega,V_\mu) |$ since we dealing with an abelian group.
\end{proof}

It turns out that the direct sum decomposition from Lemma \ref{lem:4.9} can already 
be observed on the level of subalgebras of $\cH (G)$:

\begin{lem}\label{lem:4.12}
There are algebra isomorphisms
\begin{align*}
\bigoplus\nolimits_1^{|X^L (\omega,V_\mu)|} (e_\mu \cH (M) e_\mu )^{X^L (\fs)}
\rtimes \mf R_\fs^\sharp & \cong 
(e^\fs_M \cH (M) e^\fs_M )^{X^L (\fs)} \rtimes \mf R_\fs^\sharp \\
\cong (e^\sharp_{\lambda_G} \cH (G) e^\sharp_{\lambda_G} )^{X^G (\fs)} & =
\bigoplus\nolimits_{a \in [L / H_\lambda]} 
(a e_{\mu_G} \cH (G) e_{\mu_G} a^{-1})^{X^G (\fs)} \\
& \cong \bigoplus\nolimits_1^{|X^L (\omega,V_\mu)|} 
(e_{\mu_G} \cH (G) e_{\mu_G})^{X^G (\fs)} .
\end{align*} 
\end{lem}
\begin{proof}
The first isomorphism is a direct consequence of Lemma \ref{lem:4.9} and the second is
Proposition \ref{prop:3.I}.b. As shown in Proposition \ref{prop:3.I}, it can be 
decomposed as
\[
\bigoplus_{a_1,a_2 \in [L / H_\lambda]} (a_1 e_{\mu_G} \cH (G) e_{\mu_G} 
a_2^{-1})^{X^G (\fs)} \longleftrightarrow \bigoplus_{a_1,a_2 \in [L / H_\lambda]} 
(a_1 e^\mu \cH (M) e_\mu a_2^{-1} )^{X^L (\fs)} \rtimes \mf R_\fs^\sharp .
\]
But by Lemma \ref{lem:4.9} the summands with $a_1 \neq a_2$ are 0 on the right
hand side, so they are also 0 on the left hand side. This proves the equality
in the lemma. 

The final isomorphism is given by
\[
(e_{\mu_G} \cH (G) e_{\mu_G})^{X^G (\fs)} \to 
(a e_{\mu_G} \cH (G) e_{\mu_G} a^{-1})^{X^G (\fs)} : f \mapsto a f a^{-1} . \qedhere
\]
\end{proof}

Recall from Theorem \ref{thm:3.17} that the middle algebra in Lemma \ref{lem:4.12}
is isomorphic to $e^\sharp_{\lambda_{G^\sharp Z(G)}} \cH (G^\sharp Z(G))
e^\sharp_{\lambda_{G^\sharp Z(G)}}$, which is Morita equivalent with $\cH (G^\sharp
Z(G))^{\fs}$. In the above direct sum decomposition also holds on this level.
to formulate it, let $e_{\mu_{G^\sharp Z(G)}} \in \cH (G^\sharp Z(G))$ be the
restriction of $e_{\mu_G} : K_G \to \C$ to $G^\sharp Z(G) \cap K$.
By our choice of Haar measures, $e_{\mu_{G^\sharp Z(G)}}$ is idempotent. See also 
\eqref{eq:3.98}.

\begin{cor}\label{cor:4.13}
There are algebra isomorphisms
\begin{align*}
e^\sharp_{\lambda_{G^\sharp Z(G)}} \cH (G^\sharp Z(G)) 
e^\sharp_{\lambda_{G^\sharp Z(G)}} &
= \bigoplus\nolimits_{a \in [L / H_\lambda]} a e_{\mu_{G^\sharp Z(G)}} a^{-1} 
\cH (G^\sharp Z(G)) a e_{\mu_{G^\sharp Z(G)}} a^{-1} \\
& \cong \bigoplus\nolimits_1^{|X^L (\omega,V_\mu)|} e_{\mu_{G^\sharp Z(G)}} 
\cH (G^\sharp Z(G)) e_{\mu_{G^\sharp Z(G)}} \\
& \cong \bigoplus\nolimits_1^{|X^L (\omega,V_\mu)|} \big( \cH (T_\fs,W_\fs,q_\fs) \otimes 
\End_\C (V_\mu) \big)^{X^L (\fs)} \rtimes \mf R_\fs^\sharp .
\end{align*} 
\end{cor}
\begin{proof}
The equality comes from Lemma \ref{lem:4.12} and Theorem \ref{thm:3.17}. For all
$a \in [L / H_\lambda]$ the map
\[
e_{\mu_{G^\sharp Z(G)}} \cH (G^\sharp Z(G)) e_{\mu_{G^\sharp Z(G)}}  \to
a e_{\mu_{G^\sharp Z(G)}} a^{-1} \cH (G^\sharp Z(G)) 
a e_{\mu_{G^\sharp Z(G)}} a^{-1} : f \mapsto a f a^{-1}
\]
is an algebra isomorphism. 
The remaining isomorphism follows again from Lemma \ref{lem:4.12}.  
\end{proof}

Now we analyse the action of $\Stab (\fs,P \cap M)$ on $\cH (T_\fs,W_\fs,q_\fs) 
\otimes \End_\C (V_\mu)$ in Corollary \ref{cor:4.13}. For every $(w,\gamma) \in 
\Stab (\fs, P \cap M)$ there exists a $\chi_\gamma \in X_{\nr}(L)$ such that 
\begin{equation} \label{eq:3.36}
w (\omega) \otimes \gamma \cong \omega \otimes \chi_\gamma \in \Irr (L).
\end{equation}
Here $\omega \otimes \chi_\gamma$ is unique, so $\chi_\gamma$ is unique up
to $X_\nr (L,\omega)$. If $\gamma$ itself is unramified, then $w = 1$ by
Lemma \ref{lem:2.4}.d. Therefore we may and will assume that
\begin{equation}\label{eq:4.1}
\chi_\gamma = \gamma \quad \text{if} \quad \gamma \in 
X_{\nr}(L / L^\sharp Z(G)) = X_{\nr}(L) \cap X^L (\fs).
\end{equation}
In view of the conditions \ref{cond}, in particular 
$L = \prod_i L_i^{e_i}$, we may simultaneously assume that
\begin{equation}\label{eq:3.44}
\chi_\gamma = \prod\nolimits_i \chi_{\gamma,i}^{\otimes e_i} . 
\end{equation}
Notice that with this choice $\chi_\gamma \in X_\nr (L)$ is invariant under
$W_\fs = W(M,L)$. Via the map $\chi \mapsto \omega \otimes \chi$, $\chi_\gamma$ 
determines a $W_\fs$-invariant element of $T_\fs$. 

Recall the bijection
\[
J(\gamma,\omega \otimes \chi_\gamma^{-1}) \in \Hom_L (\omega \otimes \chi_\gamma^{-1},
w^{-1} (\omega) \otimes \gamma^{-1}) 
\]
from \eqref{eq:2.22}. It restricts to a bijection
\[
V_\mu = e_\mu V_{\omega \otimes \chi_\gamma^{-1}} \to
V_\mu = e_\mu V_{w^{-1} (\omega) \otimes \gamma^{-1}} .
\]
Clearly \eqref{eq:4.1} enables us to take
\begin{equation} \label{eq:3.55}
J(\gamma,\omega \otimes \chi_\gamma^{-1}) = \mathrm{id}_{V_\omega} 
\quad \text{if} \quad \gamma \in X_{\nr}(L / L^\sharp Z(G)) .
\end{equation}
Recall from \eqref{eq:3.89} and \eqref{eq:2.30} that 
\begin{equation}\label{eq:4.90}
J(\gamma,\omega \otimes \chi_\gamma^{-1}) |_{V_{\mu }} \in \C^\times 
\mathrm{id}_{V_{\mu }} \quad \text{if} \quad \gamma \in X^L (\omega,V_{\mu }) .
\end{equation}

\begin{lem}\label{lem:3.9}
The action $\alpha$ of $\Stab (\fs,P \cap M)$ on  
\[
e_{\mu } \cH (M) e_{\mu } \cong 
\cH (T_\fs, W_\fs, q_\fs) \otimes \End_\C (V_{\mu }) 
\]
preserves both tensor factors. On $\cH (T_\fs, W_\fs, q_\fs)$ it is given by
\[
\alpha_{(w,\gamma)}(\theta_x [v]) = \chi_\gamma^{-1}(x) \theta_{w(x)} [w v w^{-1}] 
\qquad x \in X^* (T_\fs), v \in W_\fs ,
\]
and on $\End_\C (V_{\mu })$ by
\[
\alpha_{(w,\gamma)}(h) = J(\gamma,\omega \otimes \chi_\gamma^{-1}) \circ h \circ
J(\gamma,\omega \otimes \chi_\gamma^{-1})^{-1} .
\]
Furthermore $X^L (\omega,V_\mu)$ is the subgroup of elements that act trivially. 
\end{lem}
\emph{Remark.} It is crucial that $\chi_\gamma$ is $W_\fs$-invariant and that 
$w$ normalizes $P \cap M$, for otherwise the above formulae would not define 
an algebra automorphism of $\cH (T_\fs, W_\fs, q_\fs)$. This can 
be seen with the Bernstein presentation, in particular \eqref{eq:3.41}.
\begin{proof}
By definition, for all 
$\chi \in X_{\nr}(L), f \in e_{\mu_L} \cH (L) e_{\mu_L}$
\begin{equation}\label{eq:3.38}
\begin{split}
(\omega \otimes \chi) & (\alpha_{(w,\gamma)}(f)) = 
(w^{-1}(\omega \otimes \chi \gamma^{-1}))(f) \\
& = J(\gamma,\omega \otimes \chi_\gamma^{-1}) \circ (\omega \otimes w^{-1}(\chi) 
\chi_\gamma^{-1}) (f) \circ J(\gamma,\omega \otimes \chi_\gamma^{-1})^{-1} . 
\end{split}
\end{equation}
For $f$ with image in $\End_\C (V_{\mu })$ the unramified characters $\chi$ and 
$w^{-1}(\chi) \chi_\gamma^{-1}$ are of no consequence. Thus \eqref{eq:3.38} implies 
the asserted formula for $\alpha_{(w,\gamma)}$ on $\End_\C (V_\mu)$.

Since $\cH (T_\fs, W_\fs, q_\fs)$ is the centralizer of $\End_\C (V_\mu)$ in
$e_{\mu } \cH (M) e_{\mu }$, it is also stabilized by $\Stab (\fs,P \cap M)$.
By considering supports we see that 
\begin{equation}
\alpha_{(w,\gamma)}([x]) \in \C [w x w^{-1}] 
\text{ for all } x \in X^* (T_\fs) \rtimes W_\fs . 
\end{equation}
For any simple reflection $s \in W_\fs$, $w s w^{-1}$ is again a simple
reflection in $W_\fs$, because $w \in \mf R_\fs^\sharp$ normalizes $P \cap M$. 
In $\cH (W_\fs, q_\fs)$ we have
\[
([s] + 1)([s] - q_\fs (s)) = 0 = ([w s w^{-1}] + 1) ([w s w^{-1}] - q_\fs (w s w^{-1})),
\]
where $q_\fs (s),  q_\fs (w s w^{-1}) \in \R_{>1}$. Since $\alpha_{(w,\gamma)}$ is an
algebra automorphism, we deduce that 
\[
q_\fs (s) = q_\fs (w s w^{-1})
\]
and that $\alpha_{(w,\gamma)}([s]) = [w s w^{-1}]$. Every $v \in W_\fs$ is a product
of simple reflections $s_i$, and then $[v]$ is a product of the $[s_i]$ in the same way.
Hence
\[
\alpha_{(w,\gamma)}([v]) = [w v w^{-1}] \text{ for all } v \in W_\fs . 
\]
The formula \eqref{eq:3.35} also defines an action $\alpha$ of $\Stab (\fs,P \cap M)$ on
\[
e_{\mu_L } \cH (L) e_{\mu_L } \cong \mathcal O (T_\fs) \otimes \End_\C (V_{\mu }) ,
\]
which for similar reasons stabilizes $\End_\C (V_{\mu })$.
Now we have two actions of \\
$\Stab (\fs,P \cap M)$ on $\End_\C (V_{\mu })$, depending
on whether we considere it as a subalgebra of $e_{\mu_L } \cH (L) e_{\mu_L }$ or of
$e_{\mu } \cH (M) e_{\mu }$. It is obvious from the definition of $\mu$ that these
two actions agree. 

The maps \eqref{eq:3.92} and \eqref{eq:3.29} lead to a canonical injection \label{i:27}
\[
i_{P,\mu} : \mc O(T_\fs) \to \cH (T_\fs, W_\fs, q_\fs) ,
\]
which is a restriction of the map $\bigoplus_{\gamma \in X^L (\fs) / X^L (\fs)^1}
i_{P,\mu^1 \otimes \gamma}$. Since it is canonical, $i_{P,\mu }$ commutes with
the respective actions $\alpha_{(w,\gamma)}$.

Now we take $x \in X^* (T_\fs) \subset \mathcal O (T_\fs)$ in \eqref{eq:3.38}. With
a Morita equivalence we can replace $\omega \otimes \chi$ by the one-dimensional
$\mathcal O (T_\fs)$-representation $[\omega \otimes \chi]$ with character 
$\omega \otimes \chi \in T_\fs$. Then \eqref{eq:3.38} becomes
\begin{multline}\label{eq:4.O}
[\omega \otimes \chi](\alpha_{(w,\chi)}(x)) = 
[\omega \otimes \chi_\gamma^{-1} w^{-1}(\chi)](x)  \\
= [\omega \otimes w^{-1}(\chi)](\chi_\gamma^{-1}(x) x) =
[\omega \otimes \chi] (\chi_\gamma^{-1}(x) w(x)) .
\end{multline}
Thus $\alpha_{(w,\gamma)}(x) = \chi_\gamma^{-1}(x) w(x)$. For $x \in X^* (T_\fs)$ 
positive $w(x)$ is also positive, as $w$ normalizes $P \cap M$. We obtain 
\begin{multline*}
\hspace{-3mm} \alpha_{(w,\gamma)} (\theta_x) = 
\alpha_{(w,\gamma)}(i_{P,\mu} (x)) =
i_{P,\mu} (\alpha_{(w,\chi)} (x)) = 
i_{P,\mu} (\chi_\gamma^{-1}(x) w(x)) = \chi_\gamma^{-1}(x) \theta_{w(x)} .
\end{multline*}
Since $\alpha_{(w,\gamma)}$ is an algebra homomorphism, this implies the same formula
for all $x \in X^* (T_\fs)$.

It is clear from \eqref{eq:4.90} and \eqref{eq:3.38} that the group $X^L (\omega,V_\mu)$
fixes every element of $e_\mu \cH (M) e_\mu$. Conversely, suppose that $(w,\gamma)$
acts trivially. From the formulas for the action on $\cH (T_\fs, W_\fs, q_\fs)$ we see
that $w = 1$ and $\chi_\gamma \in X_\nr (L,\omega)$, so $\gamma \in X^L (\omega)$. 
Then we deduce from \eqref{eq:3.38} that $\gamma \in X^L (\omega,V_\mu)$. 
\end{proof}

Since we have a type with idempotent $e_{\mu^1}$ but not with idempotent $e_\mu$,
we would like to reduce $(e_\mu \cH (M) e_\mu )^{X^L (\fs)}$ to 
$( e_{\mu^1} \cH (M) e_{\mu^1} )^{X^L (\fs)}$. The next lemma solves a part of the
problem, namely the elements $c_\gamma$ that do not lie in $L^1$.

\begin{lem}\label{lem:4.10}
Let $\gamma \in X^L (\fs)$ be such that 
\[
\chi (c_\gamma) = 1 \text{ for all } 
\chi \in X_\nr (L / L^\sharp Z(G)) \cap X_\nr (L,\omega).
\] 
Then there exists $x_\gamma \in \widetilde{X^* (T_\fs)}$ such that 
$x_\gamma c_\gamma \in L^1 L^\sharp$.
\end{lem}
\begin{proof}
Recall that every $\chi \in X_\nr (L,\omega)$ vanishes on $Z(L)$. Hence 
\[
X_\nr (L / L^\sharp Z(G)) \cap X_\nr (L,\omega) = 
X_\nr (L / L^\sharp) \cap X_\nr (L,\omega) ,
\]
and $c_\gamma$ determines a character of 
\begin{equation}\label{eq:4.X}
X_\nr (L / L^\sharp) / X_\nr (L / L^\sharp) \cap X_\nr (L,\omega) . 
\end{equation}
This is a subtorus of $T_\fs \cong X_\nr (L) / X_\nr (L,\omega)$,
so we can find $x_\gamma \in \widetilde{X^* (T_\fs)}$ such that $x_\gamma^{-1}$
restricts to the same character of \eqref{eq:4.X} as $c_\gamma$. Then
\[
x_\gamma c_\gamma \in \bigcap\nolimits_{\chi \in X_\nr (L / L^\sharp)} \ker \chi 
= L^1 L^\sharp . \qedhere 
\]
\end{proof}

In view of Lemma \ref{lem:4.19} we may replace $c_\gamma$ by $x_\gamma c_\gamma$
as in Lemma \ref{lem:4.10}. From now on we assume that this has been done
for all $\gamma$ to which Lemma \ref{lem:4.10} applies. Recall that we were already 
assuming that $c_\gamma \in L^1$ whenever this is possible.

This gives rise to groups \label{i:52} \label{i:75}
\begin{align*}
& X^L (\fs)^2 = \{ \gamma \in X^L (\fs) \mid c_\gamma \in L^1 L^\sharp \} ,\\
& \Stab (\fs,P \cap M)^2 = 
\{ (w,\gamma) \in \Stab (\fs, P \cap M) \mid c_\gamma \in L^1 L^\sharp \} ,
\end{align*}
and to an idempotent \label{i:11}
\[
e_{\mu^2} = \sum_{\gamma \in X^L (\fs)^2 / X^L (\fs)^1} e_{\mu^1} = 
\sum_{\gamma \in X^L (\fs)^2 / (X^L (\fs) \cap X^G (\fs,\lambda)} 
e_{\lambda \otimes \gamma} . 
\]
The tower of groups 
\[
X^L (\fs)^1 \subset X^L (\fs)^2 \subset X^L (\fs)
\]
corresponds to a tower of $K$-representations
\[
\mu^1 \subset \mu^2 \subset \mu ,
\]
where $V_{\mu^2} = \bigoplus_{\gamma \in X^L (\fs)^2 / X^L (\fs)^1} V_{\mu^1}$.
With these objects we can refine Corollary \ref{cor:4.13}.

\begin{thm}\label{thm:4.11}
The algebra $\cH (G^\sharp Z(G))^\fs$ is Morita equivalent with a direct sum of 
$|X^L (\omega,V_\mu)|$ copies of
\begin{align*}
& e_{\mu_{G^\sharp Z(G)}} \cH (G^\sharp Z(G)) e_{\mu_{G^\sharp Z(G)}} \cong \\
& \big( \cH (T_\fs,W_\fs,q_\fs) \otimes 
\End_\C (V_\mu) \big)^{X^L (\fs) / X^L (\omega,V_\mu)} \rtimes \mf R_\fs^\sharp \cong \\
& \big( \cH (T_\fs,W_\fs,q_\fs) \otimes 
\End_\C (V_{\mu^2}) \big)^{X^L (\fs)^2 / X^L (\omega,V_\mu)} \rtimes \mf R_\fs^\sharp .
\end{align*}
Here the actions $X^L (\fs)$ and $\mf R_\fs^\sharp$ come from the action of
$\Stab (\fs,P \cap M)$ described in Lemma \ref{lem:3.9}. In particular the group 
$X_\nr (L / L^\sharp Z(G))$ acts only via translations on $T_\fs$.
\end{thm}
\begin{proof}
By Theorem \ref{thm:3.17} and Corollary \ref{cor:4.13}
it suffices to establish the second isomorphism.
From Lemma \ref{lem:3.9} we see that $X_\nr (L / L^\sharp Z(G)) \cap X_\nr (L,\omega)$ 
fixes the subalgebra
\[
e_{\mu^2} \cH (M) e_{\mu^2} \subset e_{\mu } \cH (M) e_{\mu }  
\]
pointwise. Hence the same holds for 
\[
e_{\mu^2 \otimes \gamma} \cH (M) e_{\mu^2 \otimes \gamma}  = 
c_\gamma e_{\mu^2} \cH (M) e_{\mu^2} c_\gamma^{-1}
\]
for any $\gamma \in X^L (\fs)$. On the other hand, by Lemma \ref{lem:4.10}
$X_\nr (L / L^\sharp Z(G)) \cap X_\nr (L,\omega)$ does not fix any 
$c_\gamma \in L \setminus L^1 L^\sharp$. Therefore
\[
(e_{\mu } \cH (M) e_{\mu } )^{X^L (\fs)} = 
\Big( \bigoplus_{\gamma \in X^L (\fs) / X^L (\fs)^2} \!\!\!\!
e_{\mu^2 \otimes \gamma} \cH (M) e_{\mu^2 \otimes \gamma} \Big)^{X^L (\fs)}
\cong \Big( e_{\mu^2} \cH (M) e_{\mu^2} \Big)^{X^L (\fs)^2} . 
\]
The proof of Theorem \ref{thm:3.6}.d shows that
\[
e_{\mu^2} \cH (M) e_{\mu^2} \cong \cH (T_\fs,W_\fs,q_\fs) \otimes 
\End_\C (V_{\mu^2}) .
\]
For $(w,\gamma) \in \Stab (\fs,P \cap M)^2$ the intertwiner 
$J (\gamma, \omega \otimes \chi_\gamma^{-1})$ restricts to a bijection
\[
V_{\mu^2} = e_{\mu^2} V_{\omega \otimes \chi_\gamma^{-1}} \to
V_{\mu^2} = e_{\mu^2} V_{w^{-1} (\omega) \otimes \gamma^{-1}} .
\]
Hence the action of $\Stab (\fs,P \cap M)^2$ on $\cH (T_\fs,W_\fs,q_\fs) 
\otimes \End_\C (V_{\mu})$ described in Lemma \ref{lem:3.9} preserves the
subalgebra $\cH (T_\fs,W_\fs,q_\fs) \otimes \End_\C (V_{\mu^2})$. 
By equations \eqref{eq:3.55} and \eqref{eq:4.90} the groups 
$X_\nr (L / L^\sharp Z(G))$ and $X^L (\omega,V_{\mu })$ act as asserted.
\end{proof}

In combination with Lemma \ref{lem:3.13},
the above proof makes it clear that the use of the group $L / H_\lambda \cong
\Irr (X^L (\omega,V_{\mu }))$ is unavoidable. Namely, by \eqref{eq:2.2} the operators 
$I(\gamma,\omega \otimes \chi)$ associated to $\gamma \in X^L (\omega,V_{\mu })$ 
cause some irreducible representations of $G$ to split upon restriction to
$G^\sharp Z(G)$. But in the proof of Theorem \ref{thm:4.11} we saw that the
same operators act trivially on 
\[
e_{\mu } \cH (M) e_{\mu } \cong \cH (T_\fs, W_\fs, q_\fs) \otimes \End_\C (V_\mu ) . 
\]
This has to be compensated somehow, and we do so by adding a direct
summand for every irreducible representation of $X^L (\omega,V_{\mu })$. See also
Example \ref{ex:5.1}.

\subsection{Hecke algebras for the derived group} \

Let $\omega \in \Irr(L)$ be supercuspidal and let $\sigma^\sharp$ be an irreducible
subquotient of $\Res_{L^\sharp}^L (\omega)$, or equivalently of 
$\Res_{L^\sharp Z(G)}^L (\omega)$. Consider the Bernstein torus $T_{\mf t^\sharp}$ 
of $\mf t^\sharp = [L^\sharp,\sigma^\sharp]_{G^\sharp}$ and $T_{\mf t}$ of
$\mf t = [L^\sharp Z(G),\sigma^\sharp]_{G^\sharp Z(G)}$. Then $T_{\mf t^\sharp}$ 
is the quotient of $T_{\mf t}$ with respect to the action of 
$X_\nr (L^\sharp Z(G) / L^\sharp)$. In turn $T_{\mf t}$ clearly is a quotient of 
$X_{\nr}(L)$ via $\chi \to \sigma^\sharp \otimes \chi$. But it is not so obvious that
it is also a quotient of $T_\fs \cong X_{\nr}(L) / X_{\nr}(L,\omega)$, because the 
isomorphism $\omega \otimes \chi \cong \omega$ for $\chi \in X_{\nr}(L,\omega)$ might
be complicated. The next result shows that this awkward scenario does not occur and
that $T_{\mf t}$ is a quotient of $T_\fs$.

\begin{lem}\label{lem:4.7}
For $\omega$ and $\sigma^\sharp$ as above, 
$\sigma^\sharp \otimes \chi \cong \sigma^\sharp$
for all $\chi \in X_{\nr}(L,\omega)$.
\end{lem}
\begin{proof}
Let $\mu_L$ be as in \eqref{eq:4.m}. By Lemma \ref{lem:3.4} and Theorem
\ref{thm:3.6} $\cH (L)^{\fs_L}$ is Morita equivalent with
\begin{equation}\label{eq:4.28}
e^\fs_L \cH (L) e^\fs_L \cong 
\mc O (T_\fs) \otimes \End_\C (V_{\mu}) \otimes M_{|L / H_\lambda|}(\C) .
\end{equation}
The analogues of Lemma \ref{lem:3.1} and Proposition \ref{prop:3.2} for $L$ 
show that $\cH (L^\sharp Z(G))^{\fs_L}$ is Morita equivalent with 
\begin{equation}\label{eq:4.20}
e^\fs_L \cH (L)^{X_L (\fs)} e^\fs_L \cong \Big( \mc O (T_\fs) \otimes 
\End_\C (V_{\mu} ) \Big)^{X^L (\fs)} \otimes M_{|L / H_\lambda|}(\C) . 
\end{equation}
Under the first Morita equivalence, $\omega$ is mapped to $e^\fs_L V_\omega$, 
considered as an $\mc O (T_\fs)$-module with character $\omega \in T_\fs$. 
For every $\gamma \in X^L (\omega) ,\; \omega \otimes \gamma$ is mapped to 
the same module of \eqref{eq:4.28}.

Recall the operator $I(\gamma,\omega) \in \Hom_L (\omega \otimes \gamma,\omega)$ 
from \eqref{eq:2.30}. It restricts to 
\[
I_{e^\fs_L} (\gamma,\omega) \in \Hom_{e^\fs_L \cH (L) e^\fs_L} 
(e^\fs_L V_{\omega \otimes \chi}, e^\fs_L V_\omega ) .
\]
Since we are dealing with Morita equivalences and by \eqref{eq:2.1}, 
these operators give rise to algebra isomorphisms
\begin{equation}\label{eq:4.29}
\End_{e^\fs_L \cH (L)^{X_L (\fs)} e^\fs_L} 
(e^\fs_L V_\omega) 
\cong \End_{L^\sharp} (\omega)) \cong \C [X^L (\omega,\kappa_\omega)] .
\end{equation}
It follows from \eqref{eq:2.2} that there exists a unique 
$\rho \in \Irr (\C [X^L(\omega), \kappa_\omega])$ such that
$\sigma^\sharp \cong \Hom_{\C [X^L(\omega),\kappa_\omega]}(\rho,\omega)$. 
Then \eqref{eq:4.29} implies
\[
e^\fs_L V_{\sigma^\sharp} \cong 
\Hom_{\C [X^L(\omega),\kappa_\omega]}(\rho,e^\fs_L V_\omega) .
\]
But $X^L (\omega \otimes \chi) = 
X^L (\omega), \kappa_{\omega \otimes \chi} = \kappa_\omega$ and
\[
\sigma^\sharp \otimes \chi \cong 
\Hom_{\C [X^L(\omega),\kappa_\omega]}(\rho,\omega \otimes \chi) .
\]
Moreover $\omega$ and $\omega \otimes \chi$ correspond 
to the same module of \eqref{eq:4.28}, so
\[
e^\fs_L (V_{\sigma^\sharp \otimes \chi}) \cong 
\Hom_{\C [X^L(\omega),\kappa_\omega]}(\rho,e^\fs_L V_{\omega \otimes \chi}) =
\Hom_{\C [X^L(\omega),\kappa_\omega]}(\rho,e^\fs_L V_{\omega}) \cong
e^\fs_L V_{\sigma^\sharp} .
\]
In view of the second Morita equivalence above, this implies that 
$V_{\sigma^\sharp \otimes \chi} \cong V_{\sigma^\sharp}$ as $L^\sharp$-representations.
\end{proof}

Now we are finally able to give a concrete description of the Hecke 
algebras associated to $G^\sharp$. Let $T_\fs^\sharp$ be the restriction of 
$T_\fs$ to $L^\sharp$, that is, \label{i:59}
\begin{equation}
T_\fs^\sharp := T_\fs / X_\nr (L / L^\sharp) = T_\fs / X_\nr (G  ) 
\cong X_\nr (L^\sharp) / X_\nr (L,\omega) .
\end{equation}
With this torus we build an affine Hecke algebra $\cH (T_\fs^\sharp, W_\fs, q_\fs)$
like in \eqref{eq:3.42} and \eqref{eq:3.43}. Recall from \eqref{eq:2.23} that there are 
finitely many Bernstein components $\mf t^\sharp$ for $G^\sharp$ such that
\begin{equation}\label{eq:4.46}
\cH (G^\sharp)^\fs = 
\bigoplus\nolimits_{\mf t^\sharp \prec \fs} \cH (G^\sharp )^{\mf t^\sharp} . 
\end{equation}
By Lemma \ref{lem:4.7} the Bernstein torus associated to any
$\mf t^\sharp \prec \mf s$ is a quotient of $T_\fs^\sharp$ by a finite group.
However, we warn that in general $T_{\mf t^\sharp}$ is not equal to 
$T_\fs^\sharp$, see Example \ref{ex:torus}.

Define $\kappa ((w,\gamma),(w',\gamma')) \in \C^\times$ by
\begin{equation}\label{eq:4.4}
J (\gamma,\omega \otimes \chi) \circ J(\gamma',\omega \otimes \chi) = 
\kappa \big( (w,\gamma),(w',\gamma') \big) J(\gamma \gamma',\omega \otimes \chi) .
\end{equation}
By the formula for $\alpha_{(w,\gamma)}(h)$ in Lemma \ref{lem:3.9} this determines
a 2-cocycle of
\begin{equation}\label{eq:3.31}
\Stab (\omega \otimes \chi,P \cap M) :=
\Stab (\omega \otimes \chi) \cap \Stab (\fs,P \cap M) . 
\end{equation}
We note that by \eqref{eq:2.12}, \eqref{eq:2.5} 
and Lemma \ref{lem:2.4}.b the group \eqref{eq:3.31} is isomorphic to 
$X^G (I_P^G (\omega \otimes \chi))$, via projection on the second coordinate. 

Recall the idempotent $e^\sharp_{\lambda_{G^\sharp}}$ from \eqref{eq:3.98}.
We define $e_{\mu_{G^\sharp}} \in \cH (G^\sharp)$ similarly, as the restriction
of $e_{\mu_G} : K \to \C$ to $K \cap G^\sharp$. Our final and main result
translates Theorem \ref{thm:4.11} from $G^\sharp Z(G)$ to $G$. \label{i:06}

\begin{thm}\label{thm:3.12}
The algebra $\cH (G^\sharp)^\fs$ is Morita equivalent with
\begin{multline*}
e^\sharp_{\lambda_{G^\sharp}} \cH (G^\sharp) e^\sharp_{\lambda_{G^\sharp}} =
\bigoplus_{a \in [L / H_\lambda]} a e_{\mu_{G^\sharp}} a^{-1} \cH (G^\sharp)
a e_{\mu_{G^\sharp}} a^{-1} 
\cong \bigoplus\nolimits_1^{|X^L (\omega,V_\mu)|} 
e_{\mu_{G^\sharp}} \cH (G^\sharp) e_{\mu_{G^\sharp}} .
\end{multline*}
There are algebra isomorphisms 
\begin{align*}
e_{\mu_{G^\sharp}} \cH (G^\sharp) e_{\mu_{G^\sharp}}
& \cong \Big( \cH (T_\fs^\sharp, W_\fs, q_\fs) \otimes 
\End_\C (V_\mu) \Big)^{X^L (\fs) / X^L (\omega,V_\mu ) X_\nr (L / L^\sharp Z(G))} 
\rtimes \mf R_\fs^\sharp \\
& \cong \Big( \cH (T_\fs^\sharp, W_\fs, q_\fs) 
\otimes \End_\C (V_{\mu^2}) \Big)^{X^L (\fs)^2 / X^L (\omega,V_\mu ) 
X_\nr (L / L^\sharp Z(G))} \rtimes \mf R_\fs^\sharp .
\end{align*}
The actions of $X^L (\fs)$ and $\mf R_\fs^\sharp$ come from $\Stab (\fs,P \cap M)$  
via Lemma \ref{lem:3.9}, which involves a projective action on 
\[
V_{\mu } = \bigoplus\nolimits_{\gamma \in X^L (\fs) / X^L (\fs)^2} 
V_{\mu^2 \otimes \gamma} .
\] 
The restriction of the associated 2-cocycle to $\Stab (\omega \otimes \chi,P \cap M)$ 
corresponds to the 2-cocycle $\kappa_{I_P^G (\omega \otimes \chi)}$ from \eqref{eq:2.1}. 
Its cohomology class is trivial if $G = \GL_m (D)$ is split.
\end{thm}
\begin{proof}
By Theorem \ref{thm:3.18}
\[
\cH (G^\sharp )^\fs \sim_M 
e^\sharp_{\lambda_{G^\sharp}} \cH (G^\sharp) e^\sharp_{\lambda_{G^\sharp}} 
\cong e^\sharp_{\lambda_G} \cH (G)^{X^G (\fs) X_\nr (G)} e^\sharp_{\lambda_G} ,
\] 
while Theorem \ref{thm:3.17} tells us that
\[
e^\sharp_{\lambda_{G^\sharp Z(G)}} \cH (G^\sharp Z(G)) e^\sharp_{\lambda_{G^\sharp Z(G)}} 
\cong e^\sharp_{\lambda_G} \cH (G)^{X^G (\fs)} e^\sharp_{\lambda_G} .
\]
So the second line of the Theorem follows from Corollary \ref{cor:4.13} upon taking
invariants for $X_\nr (G) / X^G (\fs) \cap X_\nr (G) \cong X_\nr (G^\sharp Z(G))$.
Furthermore this gives an isomorphism
\[
\bigoplus\nolimits_1^{|X^L (\omega,V_\mu)|} 
e_{\mu_{G^\sharp}} \cH (G^\sharp) e_{\mu_{G^\sharp}} \cong
\bigoplus\nolimits_1^{|X^L (\omega,V_\mu)|} 
(e_\mu \cH (M) e_\mu)^{X^L (\fs) X_\nr (G  )} \rtimes \mf R_\fs^\sharp .
\]
The proof of Lemma \ref{lem:3.9} can also be applied to the action of 
$X_\nr (G  ) = X_\nr (L / L^\sharp)$ on 
\[
e_\mu \cH (M) e_\mu \cong \cH (T_\fs^\sharp, W_\fs, q_\fs) \otimes \End_\C (V_\mu) ,
\]
and it shows that $X_\nr (G  )$ acts only via translations of $T_\fs$.
The remaining action of $X^L (\fs)^2$ is trivial on 
$X_\nr (L / L^\sharp Z(G)) X^L (\omega,V_\mu)$ by Theorem \ref{thm:4.11}.
This gives the third isomorphism, from which the fourth follows (again using 
Theorem \ref{thm:4.11}).

The definitions \eqref{eq:2.30} and \eqref{eq:2.18} show that the 2-cocycle $\kappa$ of 
$\Stab (\omega,P \cap M)$ determined by \eqref{eq:4.4} is related to \eqref{eq:2.21} by
\[
\kappa \big( (w,\gamma),(w',\gamma') \big) = 
\kappa_{I_P^G (\omega \otimes \chi)}(\gamma,\gamma') .
\]
Let $\phi$ be the Langlands parameter of the Langlands quotient of 
$I_P^G (\omega \otimes \chi)$. Via the local Langlands correspondence 
$\kappa_{I_P^G (\omega \otimes \chi)}$ is related to a 2-cocycle of the component 
group of $\phi$, see \cite[Lemma 12.5]{HiSa} and \cite[Theorem 3.1]{ABPS3}. 
Hence $\kappa_{I_P^G (\omega \otimes \chi)}$ is trivial
if $\GL_m (D)$ is split, and otherwise it reflects the Hasse invariant of $D$. 
\end{proof}

\begin{rem}\label{rem:4.splitCocycle}
In the case $G^\sharp = \SL_n (F)$ Theorem \ref{thm:3.12} can be 
compared with \cite[Theorem 11.1]{GoRo2}. The algebra that Goldberg and Roche
investigate is $\cH (G^\sharp, \tau)$, where $\tau$ is an irreducible
subrepresentation of $\lambda_G$ as a representation of $PN^1 (K_G,\lambda_G)
\cap G^\sharp$. They show that this is a type for a single Bernstein component 
$\mf t^\sharp \prec \fs$. Then $\cH (G^\sharp,\tau)$ is a subalgebra of our 
$e_{\mu_{G^\sharp}} \cH (G^\sharp) e_{\mu_{G^\sharp}}$, because the idempotent $e_\tau$ 
is smaller than $e_{\mu_{G^\sharp}}$. In our terminology, \cite{GoRo2} shows that
\begin{align*}
& \cH (G^\sharp, \tau) \cong \cH (T_{\mf t^\sharp},W_\fs, q_\fs) 
\rtimes \mf R_{\mf t^\sharp} ,\\
& e_\tau \cH (G^\sharp) e_\tau \cong 
\cH (T_{\mf t^\sharp},W_\fs, q_\fs) \rtimes \mf R_{\mf t^\sharp} \otimes \End_\C (V_\tau) .
\end{align*}
Only about the part $\rtimes \mf R_{\mf t^\sharp}$ that Goldberg and Roche
are not so sure. In \cite{GoRo2} it is still conceivable that $\mf R_{\mf t^\sharp}$
(denoted $C$ there) is only embedded in $\cH (G^\sharp, \tau)$ as part of a twisted
group algebra $\C [\mf R_{\mf t^\sharp},\delta]$. With Theorem \ref{thm:3.12}, (actually
already with Proposition \ref{prop:3.2}) we see that the 2-cocyle $\delta$ from 
\cite[\S 11]{GoRo2} is always trivial.
\end{rem}

\section{Examples} 
\label{sec:exa}

This paper is rather technical, so we think it will be helpful for the reader 
to see some examples. These will also make clear that in general none of the 
introduced objects is trivial. Most of the notations used below are defined in
Subsections \ref{par:res1} and \ref{par:res2}. 

\begin{ex} [Weyl group in $G^\sharp$ bigger than in $G$] 
\ \label{ex:Weyl} 

Let $\zeta$ be a ramified character of $D^\times$ of order 3 and take
\[
G = \GL_6 (D),\; L = \GL_1 (D)^6,\; \omega = 1 \otimes 1 \otimes \zeta \otimes
\zeta \otimes \zeta^2 \otimes \zeta^2 .
\]
For $M = \GL_2 (D)^3 ,\; \fs = [L,\omega]_G$ we have
\[
T_\fs = X_\nr (L) \cong (\C^\times )^6,\; W_\fs = W (M,L) \cong (S_2 )^3 .
\]
Furthermore $X^L (\omega) = \{1\}$ and
\[
X^G (I_P^G (\omega)) = \{1,\zeta,\zeta^2\},\; \mf R_\fs^\sharp = 
\langle (135)(246) \rangle \subset S_6 \cong W(G,L) .
\]
The stabilizer of $\omega$ in $\Stab (\fs)$ is generated by $((135)(246),\zeta)$ and
\begin{align*}
\Stab (\fs) = \Stab (\omega) W_\fs X_\nr(G / Z(G)) .
\end{align*}
Let $\cH (\GL_2,q)$ denote an affine Hecke algebra of type $\GL_2$. 
By Theorems \ref{thm:4.11} and \ref{thm:3.12} there are Morita equivalences
\begin{align*}
& \cH (G^\sharp Z(G))^\fs \sim_M
(\cH (\GL_2,q)^{\otimes 3})^{X_\nr (G / Z(G))} \rtimes \mf R_\fs^\sharp , \\
& \cH (G^\sharp)^\fs \sim_M (\cH (\GL_2 ,q)^{\otimes 3})^{X_\nr (G  )} 
\rtimes \mf R_\fs^\sharp .
\end{align*}
Now we see that $\fs$ gives rise to a unique inertial class 
$\mf t^\sharp$ for $G^\sharp$. Hence 
\[
W_{\mf t^\sharp} = W_\fs^\sharp = W_\fs \rtimes \mf R_\fs^\sharp \supsetneq W_\fs .
\]
That is, the finite group associated by Bernstein to $\mf t^\sharp$ is strictly
larger than the finite group for $\fs$.
\end{ex}

\begin{ex}[Torus for $\fs$ in $G^\sharp$ smaller than expected] 
\ \label{ex:torus} 

Let $\sigma_2 \in \Irr (\GL_4 (F))$ be as in \cite[\S 4]{Roc}. Its interesting
property is $X^{\GL_4 (F)}(\sigma_2) = \langle \chi_0 \eta \rangle$. Here
$\chi_0$ and $\eta$ are characters of $F^\times$, both of order 4, with
$\chi_0$ unramified and $\eta$ (totally) ramified. There exists a similar
supercuspidal representation $\sigma_3$ with $X^{\GL_4 (F)}(\sigma_3) = 
\langle \chi_0^{-1} \eta \rangle$. Let 
\[
G = \GL_8 (F) ,\; L = \GL_4 (F)^2 ,\; \omega = \sigma_2 \otimes \sigma_3 .
\]
Then $X^L (\omega) = \{1,\eta^2 \chi_0^2\}$ and
\begin{align*}
& X^L (\fs) = \langle \eta \rangle X_\nr (L / L^\sharp Z(G)) , \\
& T_\fs = X_\nr (L) \cong ( \C^\times )^2 .
\end{align*}
The natural guess for the torus of an inertial class $\mf t^\sharp \prec \fs$ is 
\[
T_\fs^\sharp = X_\nr (L^\sharp) = T_\fs / X_\nr (L / L^\sharp).
\] 
Yet it is not correct in this example. 
Recall from \eqref{eq:3.36} that there exists a $\chi_\eta \in X_\nr (L)$ with
$\omega \otimes \eta \cong \omega \otimes \chi_\eta$. One can check that 
$\chi_\eta = \chi_0^{-1} \otimes \chi_0 \neq 1$ and 
\[
X^L (\fs) \neq X^L (\omega) X_\nr (L / L^\sharp Z(G)) .
\]
Upon restriction to $L^\sharp ,\; \omega$ decomposes as a sum of two irreducibles,
caused by the $L$-intertwining operator
\[
J (\eta^2,\omega) : \omega \to \omega \otimes \eta^{-2} \chi_\eta^2 \; \cong \; 
\omega \otimes \eta^2 \otimes (\chi_0^{-2} \otimes \chi_0^2) .
\]
Let $\sigma^\sharp$ be one of them. We may assume that $J(\eta,\omega)^2 = 
J(\eta^2,\omega)$, so $\eta$ stabilizes $\sigma^\sharp$ up to an unramified twist. 
Then, with $\mf t^\sharp = [L^\sharp,\sigma^\sharp ]_{G^\sharp}$:
\begin{align*}
& T_{\mf t} = T_\fs / X^L (\fs,\sigma^\sharp) = 
X_\nr (L^\sharp Z(G)) / \langle \chi_0^{-1} \otimes \chi_0 \rangle , \\
& T_{\mf t^\sharp} = X_\nr (L^\sharp) / \langle 1 \otimes \chi_0^2 \rangle .
\end{align*}
\end{ex}

\begin{ex}[$W_\fs^\sharp$ acts on torus without fixed points] 
\ \label{ex:free} 

This is the example from \cite[\S 4]{Roc}, worked out in our setup. 
\[
G = \GL_8 (F) ,\; L = \GL_2 (F)^2 \times \GL_4 (F)
\]
Take $\eta,\chi_0,\sigma_2$ as in the previous example, and let 
$\sigma_1 \in \Irr (\GL_2 (F))$ be supercuspidal, such that 
$X^{\GL_2 (F)} (\sigma_1) = \{1,\eta^2\}$. Write $\gamma = \eta \chi_0$ and 
$\omega = \sigma_1 \otimes \gamma \sigma_1 \otimes \sigma_2$. Then
\[
W_\fs = 1 ,\; \mf R_\fs^\sharp = W_\fs^\sharp = W(G,L) .
\]
Let $w$ be the unique nontrivial element of $W(G,L)$, it corresponds to the
permutation $(12)(34) \in S_8$. In this case $X^L (\omega) = 1$ and the action of 
$w$ on $T_\fs$ involves translation by $\gamma$:
\[
(w,\gamma) \cdot (\sigma_1 \chi_1 \otimes \gamma \sigma_1 \chi_2 \otimes \sigma_2 \chi_3) 
= (\gamma^2 \sigma_1 \chi_2 \otimes \gamma \sigma_1 \chi_1 \otimes \gamma \sigma_2 \chi_3) 
\cong (\chi_0^2 \sigma_1 \chi_2 \otimes \gamma \sigma_1 \chi_1 \otimes \sigma_2 \chi_3)
\]
for $\chi_i$ unramified. Using $\omega$ as basepoint and $\chi_1,\chi_2 \in
X_\nr (GL_2 (F)) \cong \C^\times, \chi_3 \in X_\nr (GL_4 (F)) \cong \C^\times$ as
coordinates, we obtain 
\begin{align*}
\chi_\gamma = (\chi_0^2, 1, 1) = (-1, 1, 1) \in (\C^\times )^3 ,\\
(w,\gamma)\cdot (\chi_1,\chi_2,\chi_3) = (-\chi_2,\chi_1,\chi_3) .
\end{align*}
This is a transformation without fixed points of $T_\fs = X_\nr (L)$, and also of
$X_\nr (L^\sharp)$ and of $X_\nr (L^\sharp Z(G))$. We have
\begin{align*}
& \Stab (\fs) = \Stab (\omega) X_\nr (L / L^\sharp Z(G)) ,\\
& \Stab (\omega) = \{1, (w,\gamma)\} ,
\end{align*}
and these groups act freely on $T_\fs$. It follows that $I_P^G (\omega)$ 
is irreducible and remains so upon restriction to $G^\sharp$. Writing 
$\fs^\sharp = [L^\sharp,\omega]_{G^\sharp}$, Theorems \ref{thm:4.11} and \ref{thm:3.12}
provide Morita equivalences
\begin{align*}
& \cH (G^\sharp Z(G))^\fs \sim_M \mc O (T_\fs) \rtimes \Stab (\fs) \sim_M 
\mc O (X_\nr (L^\sharp Z(G))) \rtimes \Stab (\omega) , \\
& \cH (G^\sharp )^{\fs^\sharp} \sim_M \mc O (X_\nr (L^\sharp)) \rtimes \Stab (\omega) .
\end{align*}
\end{ex}

\begin{ex}[Decomposition into 4 irreducibles upon restriction to $G^\sharp$] 
\ \label{ex:decomposition}

This and the next example are based on \cite[\S 6.3]{ChLi}.
Let $\phi$ be a Langlands parameter for $\GL_2 (F)$ with image 
$\big\{ \matje{a}{0}{0}{b} \mid a,b \in \{\pm 1\} \big\}$
and whose kernel contains a Frobenius element of the Weil group of $F$.
The representation $\pi \in \Irr (\GL_2 (F))$ associated to $\phi$ via the local
Langlands correspondence has $X^{\GL_2 (F)}(\pi)$ consisting of four ramified 
characters of $F^\times$ of order at most two, say $\{1,\gamma,\eta,\eta \gamma\}$. 
The cocycle $\kappa_\pi$ is trivial, so by \eqref{eq:2.1}
\[
\End_{\SL_2 (F)}(\pi) \cong \C \Big[ X^{\GL_2 (F)}(\pi) \Big] \cong \C^4, 
\]
and $\Res_{\SL_2 (F)}^{\GL_2 (F)}(\pi)$ 
consists of 4 inequivalent irreducible representations.\\
Next, let $\St$ be the Steinberg representation of $\GL_2 (F)$ and consider
\begin{equation}\label{eq:4.22}
\omega = \pi \otimes \St \otimes \gamma \St \otimes \eta \St \otimes \gamma \eta \St
\in \Irr (L), 
\end{equation}
where $L = \GL_2 (F)^5$, a Levi subgroup of $G = \GL_{10}(F)$. In this setting 
\[
W_\fs = 1 ,\; X^L (\omega) = 1,\; X^G (I_P^G (\omega)) = X^{\GL_2 (F)}(\pi) .
\]
Identifying $W(G,L)$ with $S_5$, we quickly deduce 
\[
\Stab (\omega) = \{ 1, ((23)(45),\gamma \eta ), ((24)(35),\eta ), ((25)(34),\gamma ) \},
\]
and $\mf R_\fs^\sharp = W_\fs^\sharp \cong (\Z / 2 \Z)^2$. However, the action of 
$\Stab (\omega)$ on $T_\fs$ does not involve translations, it is the same action as 
that of $\mf R_\fs^\sharp$. The cocycle $\kappa_{I_P^G (\omega)}$ can
be determined by looking carefully at the intertwining operators. Only in the first 
factor of $L$ something interesting happens, in the other four factors the intertwining 
operators can be regarded as permutations. Hence the isomorphism 
$\Stab (\omega) \to X^{\GL_2 (F)}(\pi)$ induces an equality 
$\kappa_{I_P^G (\omega)} = \kappa_\pi$. As we observed above, this cocycle is trivial, 
so by \eqref{eq:2.1}
\[
\End_{G^\sharp}(I_P^G (\omega)) \cong \C \Big[ \Stab (\omega) \Big] \cong \C^4
\]
and $\Res_{G^\sharp}^{G}(I_P^G (\omega))$ decomposes as a direct sum of 4 inequivalent 
irreducible representations. Theorems \ref{thm:4.11} and \ref{thm:3.12} tell us that
there are Morita equivalences
\begin{align*}
&\cH (G^\sharp Z(G))^\fs \sim_M \mc O (T_\fs)^{X_\nr (L / L^\sharp Z(G))} \rtimes 
\mf R_\fs^\sharp = \mc O \big( X_\nr (L^\sharp Z(G)) \big) \rtimes \mf R_\fs^\sharp ,\\
&\cH (G^\sharp)^\fs \sim_M \mc O (T_\fs)^{X_\nr (L / L^\sharp)} 
\rtimes \mf R_\fs^\sharp = \mc O (X_\nr (L^\sharp)) \rtimes \mf R_\fs^\sharp .
\end{align*}
\end{ex}

\begin{ex}[non-trivial 2-cocycles] 
\ \label{ex:cocycles} 

Let $D$ be a central division algebra of dimension 4 over $F$ and recall that $D^1$
denotes the group of elements of reduced norm 1 in $D^\times = \GL_1 (D)$, which is 
also the maximal compact subgroup and the derived group of $D^\times$.\\
Take $\phi,\pi,\gamma,\eta$
as in the previous example and let $\tau \in \Irr (D^\times)$ be the image of $\pi$
under the Jacquet--Langlands correspondence. Equivalently, $\tau$ has Langlands
parameter $\phi$. Then
\[
X^{D^\times} (\tau) = X^{\GL_2 (F)} (\pi) = \{ 1, \gamma, \eta, \gamma \eta \} .
\]
As already observed in \cite{Art}, the 2-cocycle $\kappa_\tau$ of $X^{D^\times}(\tau)$
is nontrivial. The group $X^{D^\times}(\tau)$ has one irreducible 
projective non-linear representation, of dimension two. Therefore
\[
\End_{D^1}(\tau) \cong \C \Big[ X^{D^\times}(\tau), \kappa_\tau \Big]
\cong M_2 (\C)
\]
and $\Res^{D^\times}_{D^1}(\tau) \cong \tau^\sharp \oplus \tau^\sharp$
with $\tau^\sharp$ irreducible.\\
Now we consider $G = \GL_5 (D), L = \GL_1 (D)^5$ and
\[
\sigma = \tau \otimes 1 \otimes \gamma \otimes \eta \otimes \gamma \eta \in
\Irr (L).
\] 
This representation is the image of \eqref{eq:4.22} under the Jacquet--Langlands
correspondence. It is clear that 
\[
X^L (\sigma) = 1 ,\; X^G (I_P^G (\sigma)) = X^{D^\times}(\tau) 
\text{ and } W_\fs = 1,
\]
where $\fs = [L,\sigma]_G$. Just as in the previous example, we find
\begin{align*}
& X^L (\fs) = X_\nr (L / L^\sharp Z(G)) \cong \Z / 10 \Z, \\
& W_\fs^\sharp = \mf R_\fs^\sharp \cong (\Z / 2 \Z)^2 \subset W(G,L) , \\
& \Stab (\sigma) = 
\{ 1, ((23)(45),\gamma \eta ), ((24)(35),\eta ), ((25)(34),\gamma ) \},\\
& X^G (\fs) = \Stab (\sigma) X^L (\fs) .
\end{align*}
We refer to Subsection \ref{par:res2} for the definitions of these groups.
The same reasoning as for $\kappa_\pi$ shows that $\kappa_{I_P^G (\sigma)} = \kappa_\tau$
via the isomorphism $\Stab (\sigma) \to X^{D^\times}(\tau)$. Hence
\[
\End_{G^\sharp}^G (I_P^G (\sigma)) \cong 
\C \Big[ \Stab (\sigma),\kappa_{I_P^G (\sigma)} \Big] \cong M_2 (\C) ,
\]
and $\Res^G_{G^\sharp}(I_P^G (\sigma))$ is direct sum of two isomorphic irreducible
$G^\sharp$-representations.\\
To analyse the Hecke algebras associated to $\fs$, we need to exhibit some types. 
A type for $\gamma$ as a $D^\times$-representation is $(D^1, \lambda_\gamma = 
\gamma \circ \Nrd)$. The same works for other characters of $D^\times$. We know that 
$\tau$ admits a type, and we may assume that it is of the form $(D^1,\lambda_\tau)$. It is 
automatically stable under $X^{D^\times} (\tau)$ and $\dim (\lambda_\tau) > 1$ because
$\tau$ is not a character. Then
\[
\Big( (D^1)^5 , \lambda = \lambda_\tau \otimes \lambda_1 \otimes \lambda_\gamma 
\otimes \lambda_\eta \otimes \lambda_{\eta \gamma} \Big) 
\]
is a type for $[L,\sigma]_L$. The underlying vector space $V_\lambda$ can be identified
with $V_{\lambda_\tau}$.\\
We note that $M = L$ and $T_\fs = X_\nr (L) \cong (\C^\times )^5$. Proposition 
\ref{prop:3.2} and Theorem \ref{thm:3.7} show that there is a Morita equivalence
\[
\cH (G Z(G))^\fs \sim_M \Big( \mc O (T_\fs) \otimes \End_\C (V_\lambda) 
\Big)^{X_{\nr}(L / L^\sharp Z(G))} \rtimes \mf R_\fs^\sharp .
\]
The action of $\mf R_\fs^\sharp$ on $\End_\C (V_\lambda)$ comes from a projective 
representation of $X^{D^\times}(\tau)$ on $V_{\lambda_\tau}$. It does not lift a to 
linear representation because $\kappa_\tau$ is nontrivial. 
Therefore $\cH (G Z(G))^\fs$ is not Morita equivalent with 
\[
\mc O ( T_\fs)^{X_{\nr}(L / L^\sharp Z(G))} \rtimes \mf R_\fs^\sharp =
\mc O \big( X_{\nr}(L^\sharp Z(G) \big) \rtimes \mf R_\fs^\sharp .
\]
Similarly the algebras $\cH (G^\sharp)^\fs$ and
\begin{equation}\label{eq:4.23}
\Big( \mc O (T_\fs) \otimes \End_\C (V_\lambda) \Big)^{X_{\nr}(L / L^\sharp)} 
\rtimes \mf R_\fs^\sharp = \Big( \mc O (X_\nr (L^\sharp)) \otimes 
\End_\C (V_\lambda) \Big) \rtimes \mf R_\fs^\sharp
\end{equation}
are Morita equivalent, but 
\begin{equation}\label{eq:4.24}
\mc O (T_\fs)^{X_{\nr}(L / L^\sharp)} \rtimes \mf R_\fs^\sharp =
\mc O (X_\nr (L^\sharp)) \rtimes \mf R_\fs^\sharp
\end{equation}
has a different module category. One can show that \eqref{eq:4.23} and 
\eqref{eq:4.24} are quite far apart, in the sense that they have different 
periodic cyclic homology.
\end{ex}

\begin{ex}[Type does not see all $G^\sharp$-subrepresentations] 
\ \label{ex:5.1}

Take $G = \GL_2 (F)$ and let $\chi_-$ be the unique unramified character 
of order 2. There exists a supercuspidal $\omega \in \Irr (G)$ with
$X_\nr (G,\omega) = \{1,\chi_-\}$. Then $\chi_- \in X^G (\omega)$ and
\[
I(\omega,\chi_-) \in \Hom_G (\omega,\omega \otimes \chi_-).
\] 
This operator can be normalized so that its square is the identity on 
$V_\omega$. Let $G^1$ be the subgroup of $G$ generated by all compact
subgroups. The $+1$-eigenspace and the $-1$-eigenspace of 
$I(\omega,\chi_-)$ are irreducible $G^1$-subrepresentations of
$\mathrm{Res}^G_{G^1}(\omega)$, and these are conjugate via an element
$a \in G \setminus G^1 Z(G)$. 
Any type for $[G,\omega]_G$ is based on a subgroup of $G^1$, so it sees
only one of the two irreducible $G^1$-subrepresentations of $\omega$. \\
This phenomenon forces us to introduce the group $L / H_\lambda$ in Lemma
\ref{lem:3.30} (here $G / G^1 Z(G) \cong \{1,a\}$) and carry it with us 
through a large part of the paper.
\end{ex}

\begin{ex}[Types conjugate in $G$ but not in $G^1 G^\sharp$] 
\ \label{ex:5.2}

Consider a supercuspidal representation $\omega$ of $\GL_m (D)$ which
contains a simple type $(K,\lambda)$. Fix a uniformizer $\varpi_D$ of $D$ and denote
the unit of $\GL_m (D)$ by $1_m$. Assume that there exists $\gamma \in X^G (\fs)$ 
such that $\varpi_D 1_m$ normalizes $K$ and 
\[
\varpi_D^{-1} 1_m \cdot \lambda \cong \lambda \otimes \gamma \not\cong \lambda.
\] 
Then $\lambda$ and $\lambda \otimes \gamma$ are conjugate in $G$ but not in 
$G^1 G^\sharp$.\\
This can be constructed as follows. For simplicity we consider the case where
$K = \GL_m (\mf o_D)$ and $\lambda$ has level zero. Then $\lambda$ is inflated from
a cuspidal representation $\sigma$ of the finite group $\GL_m (k_D)$, where $k_D$ 
denotes the residue field of $D$. On this group conjugation by $\varpi_D$ has the 
same effect as some field automorphism of $k_D / k_F$. We assume that it is the 
Frobenius automorphism $x \mapsto x^q$, where $q = |k_F|$. Recall
that $k_D / k_F$ has degree $d$, so $|k_D| = q^d$.\\
We need a $\sigma \in \Irr_{\cusp}(\GL_m (k_D))$ such that
\begin{equation}\label{eq:5.3}
\sigma \circ \mathrm{Frob}_{k_F} \cong \sigma \otimes \bar \gamma \not\cong \sigma ,
\end{equation}
where $\bar \gamma \in \Irr (\GL_m (k_D) / k_D^\times \SL_m (k_D))$ is induced
by $\gamma$.\\
To find an example, we recall the classification of the characters
of $\GL_m (k_D)$ by Green \cite{Gre}. In his notation, every irreducible cuspidal
character of $\GL_m (k_D)$ is of the form $(-1)^{m-1} I_m^k [1]$, where $k \in
\Z / (q^{dm} - 1) \Z$ is such that $k,k q^d, \ldots, k q^{d(m-1)}$ are $m$ 
different elements of $\Z / (q^{dm} - 1) \Z$. 

Let us make the class function $I_m^k [1]$ a bit more explicit. 
\cite[Theorems 12 and 13]{Gre} entail that it is determined by its values on principal 
elements, that is, elements of $\GL_m (k_D)$ which do not belong to any proper parabolic 
subgroup. Let $k_E$ be a field with $q^{md}$ elements, which contains $k_D$. Suppose 
that $x \in k_E^\times$ is a generator and let $f_x \in \GL_m (k_D)$ be such
that $\det (t - f_x)$ is the minimal polynomial of $x$ over $k_D$. Then 
$f_x \in \GL_m (k_D)$ is principal, and every principal element is of this form.
We identify $\Irr (k_E^\times)$ with $\Z / (q^{dm} - 1) \Z$ by fixing a generator, 
a character $\theta : k_E^\times \to \C^\times$ of order $q^{dm} - 1$. From 
\cite[\S 3]{Gre} one can see that $I_m^k [1](f_x) = \theta^k (x)$.

In this setting, the above condition on $k$ becomes that 
$\theta^{k q^{ds}} \neq \theta^k$ for any divisor $s$ of $m$ with $1 \leq s < m$.
Let us call such a character of $k_E^\times$ regular. The Galois group 
Gal$(k_E / k_D) = \langle \mathrm{Frob}_{k_D} \rangle$ acts on $\Irr (k_E^\times)$
by $\theta^k \circ \mathrm{Frob}_{k_D} = \theta^{k q^d}$, and the regular characters
are precisely those whose orbit contains $m = |\mathrm{Gal}(k_E / k_D)|$ elements.
Now \cite[Theorem 13]{Gre} sets up a bijection
\begin{equation}\label{eq:5.1}
\begin{array}{ccc}
\Irr (k_E^\times)_{\reg} / \mathrm{Gal}(k_E / k_D) & \to & \Irr_{\cusp} (\GL_m (k_D)) \\
\theta^k & \mapsto & \sigma_k ,
\end{array}
\end{equation}
determined by
\begin{equation}\label{eq:5.2}
\mathrm{tr}(\sigma_k (f_x)) = (-1)^{m-1} \theta^k (x) 
\text{ for every generator } x \text{ of } k_E^\times .
\end{equation}
Let us describe the effects of tensoring with elements of $\Irr (G / G^\sharp Z(G))$
and of conjugation with powers of $\varpi_D$ in these terms. As noted above, the
conjugation action of $\varpi_D$ on $\GL_m (D)$ is the same as entrywise application
of $\mathrm{Frob}_{k_F} \in \mathrm{Gal}(k_D / k_F)$. With \eqref{eq:5.2} we deduce
\[
\varpi_D^{-1} \cdot \sigma_k = \sigma_k \circ \mathrm{Frob}_{k_F} = \sigma_{k q}.
\]
This corresponds to the natural action of Gal$ (k_D / k_F)$ on
$\Irr (k_E^\times)_{\reg} / \mathrm{Gal}(k_E / k_D)$.\\
Consider a $\gamma \in \Irr (G / G^\sharp Z(G))$ which is trivial on 
$\ker (\GL_m (\mf o_D) \to \GL_m (k_D))$. It induces a character of $\GL_m (k_D)$
of the form
\[
\bar \gamma = \gamma' \circ N_{k_D / k_F} \circ \det 
\quad \text{with} \quad \gamma' \in \Irr (k_F^\times) .
\]
We assume that $\gamma' = \theta |_{k_F^\times}$. By \eqref{eq:5.2}
\begin{multline*}
\bar \gamma (f_x) = \theta \Big( N_{k_D / k_F} (N_{k_E / k_D} (x)) \Big) =
\theta \Big( N_{k_E / k_F} (x) \Big) \\
= \theta \Big( x^{(q^{dm} - 1) / (q - 1)} \Big) 
=\theta^{(q^{dm} - 1) / (q - 1)} (x) . 
\end{multline*}
Comparing with \eqref{eq:5.3} we find that we want to arrange that
\[
kq \equiv k + \frac{q^{dm} - 1}{q-1} \mod q^{dm} - 1, \text{ but }
kq \notin \{k , k q^d, \ldots, k q^{d (m-1)} \} \mod q^{dm} - 1.
\]
For example, we can take $q = 3, d = 3$ and $m = 2$. Then $q^{dm} = 729$,\\
$(q^{dm} - 1) / (q-1) = 364$ and suitable $k$'s are 
182 or $182 q^d \equiv -110$. 
\end{ex}

\begin{ex}[Types conjugate in $G^1 G^\sharp$ but not in $G^1$] 
\ \label{ex:5.3}

Take $G = \GL_{2m} (D)$, $L = \GL_m (D)^2$, $\omega = \omega_1 \otimes \omega_2$
with $\omega_1,K_1,\lambda_1$ as in the previous example (but there
without subscripts). Also assume that the supercuspidal $\omega_2 \in \Irr (\GL_m (D))$
contains a simple type $(K_2,\lambda_2)$ such that 
\[
\varpi_D 1_m \cdot \lambda_2 \cong \lambda_2 \otimes \gamma \not\cong \lambda_2 .
\]
Then $(K,\lambda) = (K_1 \times K_2,\lambda_1 \otimes \lambda_2)$ is a type for
$[L,\omega]_L$. It is conjugate to $(K,\lambda \otimes \gamma)$ by
$(\varpi_D^{-1} 1_m,\varpi_D 1_m) \in G^\sharp \setminus G^\sharp \cap G^1$, but not
by an element of $G^1$.
\end{ex}

\section{Index of notations}  

This index is supplementary to the notations and conventions in Section \ref{sec:not}.\\

\begin{Parallel}[v]{0.48\textwidth}{0.48\textwidth}
\ParallelLText{

$\mc A$, \iref{i: A}

$\alpha_{(\omega, \gamma)}(f)$, \iref{i:01}

$c_\gamma$, \iref{i:c}

$C^\infty(G/U)^\fs$, \iref{i:02}

$C^\infty(U\backslash G)^\fs$, \iref{i:03}

$e_{\mu}$, \iref{i:04}

$e_{\mu_G}$, \iref{i:05}

$e_{\mu_{G^\sharp}}$, \iref{i:06}

$e^{\sharp}_{\lambda}$, \iref{i:07}

$e^{\sharp}_{\lambda_G}$, \iref{i:08}

$e^{\sharp}_{\lambda_{G^\sharp}}$, \iref{i:09}

$e_{\mu_L}$, \iref{i:10}

$e_{\mu^2}$, \iref{i:11}

$e^{\fs}_L$, \iref{i:12}

$e^{\fs}_M$, \iref{i:13}

$\mathcal{F}_L$, \iref{i:14}

$f_{x, \lambda}$, \iref{i:15}

$f_{x, \lambda_L}$, \iref{i:16}

$f_{x, \mu}$, \iref{i:17}

$G^{\sharp} = \GL_m(D)_{\der}$, \iref{i:18}

$H^{\sharp}$, \iref{i:19}

$\cH (G)$, \iref{i:89}

$\mathcal{H}(G,\lambda)$, \iref{i:20}

$\cH (G^\sharp Z(G))^\fs$, \iref{i: H}

$H_{\lambda}$, \iref{i:21}

$\mathcal{H}(M \rtimes \mathfrak{R}^{\sharp}_{\fs})^{\fs}$, \iref{i:22}

$\mathcal{H}(T_\fs, W_\fs, q_\fs)$, \iref{i:23}

$\mathcal{H}(X^*(T_{\fs}) \rtimes W_\fs, q_\fs)$, \iref{i:24}

$\mathcal{H}(W_\fs, q_\fs)$, \iref{i:25}

$I(\gamma,\pi)$, \iref{i:26}

$i_{P, \mu}$, \iref{i:27}

$\Irr^{\fs_L}(L)$, \iref{i:28}

$\Irr^{\fs}(G^\sharp)$, \iref{i:Irr}

$(J_P, \lambda_P)$, \iref{i:29}

$J(w, I^G_P(\omega \otimes \chi))$, \iref{i:31}

$(K,\lambda)$, \iref{i:82}

$(K_G, \lambda_G)$, \iref{i:32}

$(K_L, \lambda_L)$, \iref{i:83}

$L$, \iref{i:85}

$L^1$, \iref{i:33}

$[L/H_{\lambda}]$, \iref{i:34}

$M$, \iref{i:84}

$\sim_{M}$, \iref{i:35}

$\mu, \mu^1$, \iref{i:36}

$\mu_L, \mu^1_L$, \iref{i:36}

$N_G(\fs_L)$, \iref{i:37}

$N(K_L, \lambda_L)$, \iref{i:38}

$\Nrd$, \iref{i:39} 
} \ParallelRText{

$PN(K, \lambda)$, \iref{i:40}

$PN^1(K, \lambda)$, \iref{i:41}

$PN(K_L, \lambda_L)$, \iref{i:42}

$PN^1(K_L, \lambda_L)$, \iref{i:43}

$\Rep^{\lambda}(G)$, \iref{i:44}

$\Rep^\fs (G^\sharp)$, \iref{i: Rep}

$\mathfrak{R}_{\sigma^{\sharp}}$, \iref{i:45}

$\mf R^\sharp_\fs$, \iref{i:30}

$\mathfrak{R}^{\sharp}_{\omega}$, \iref{i:46}

$\fs_L = [L,\omega]_L$, \iref{i:47}

$\Stab(\fs)$, \iref{i:48}

$\Stab(\fs, \lambda)$, \iref{i:49}

$\Stab(\fs, P \cap M)$, \iref{i:50}

$\Stab(\fs, P \cap M)^1$, \iref{i:51}

$\Stab(\fs, P \cap M)^2$, \iref{i:52}

$\Stab(\omega)$, \iref{i:53}

$\ft^{\sharp}$, \iref{i:54}

$\ft^\sharp  \prec \fs$, \iref{i:55}

$\theta_x$, \iref{i:56}

$t_{P,\lambda}$, \iref{i:57}

$T_\fs$, \iref{i:58}

$T^{\sharp}_\fs$, \iref{i:59}

$V_{\mu}$, \iref{i:60}

$V_{\mu_L}$, \iref{i:61}

$V_{\mu^1} = V_{\mu^1_L}$, \iref{i:62}

$V_{\pi}$, \iref{i:63}

$V_{\omega}$, \iref{i:64}

$V_{\omega, \rho}$, \iref{i:65}

$W_{\fs}$, \iref{i:66}

$W_{\pi}$, \iref{i:67}

$W_\fs^\sharp$, \iref{i:68}

$W^{\sharp}_{\omega}$, \iref{i:69}

$X^G(\pi)$, \iref{i:70}

$X^G(\fs)$, \iref{i:71}

$X^G(\fs/ \lambda)$, \iref{i:72}

$X^L(\fs)$, \iref{i:73}

$X^L(\fs)^1$, \iref{i:74}

$X^L(\fs)^2$, \iref{i:75}

$X^L(\fs,\lambda)$, \iref{i:76}

$X^L(\fs / \lambda)$, \iref{i:77}

$X^L(\fs / \lambda)^1$, \iref{i:78}

$X^L(\omega, V_{\mu})$, \iref{i:79}

$X_\nr (L)$, \iref{i:80}

$X_\nr (L,\omega)$, \iref{i:81}

$Z(G)$, \iref{i:86} 
} \end{Parallel}

\end{document}